\documentclass[hidelinks,12pt]{amsart}
\usepackage{tikz,hyperref,stmaryrd,a4wide,amssymb,enumerate,amstext,vmargin,euscript,amssymb,amsfonts,amsmath,hyperref,faktor,bbold,graphicx,enumitem,upgreek}
\usetikzlibrary{arrows.meta}
\usetikzlibrary{calc} 
\usepackage{graphicx}
\usepackage{indentfirst}
\usepackage{faktor} 
\usepackage{dsfont}
\usepackage{listings}
\usepackage{xcolor}
\usepackage{color}
\usepackage[utf8]{inputenc}
\usepackage[all]{xy}
\usepackage[dvipsnames]{xcolor}
\usepackage[
backend=biber,
style=alphabetic,
sorting=nyt
]{biblatex}

\title[A Graph-Based Method for Invariant Densities]{A Graph-Based Method for Invariant Densities of Multidimensional Continued Fractions}

 \author[D.~Siukaev]{David Siukaev}

\date{\today}

\subjclass[2010]{}

\newtheorem{lemma}{Lemma}[section]

\newtheorem{conjecture}[lemma]{Conjecture}
\newtheorem{proposition}[lemma]{Proposition}
\newtheorem{corollary}[lemma]{Corollary}
\theoremstyle{remark}

\newtheorem{remark}[lemma]{Remark}

\newtheorem{definition}[lemma]{Definition}

\numberwithin{equation}{section}

\DeclareMathOperator{\Id}{Id}
\DeclareMathOperator{\Leb}{Leb}
\DeclareMathOperator{\sgn}{sgn}

\address{Université Sorbonne Paris Nord, LAGA, CNRS, UMR 7539, F--93430, Villetaneuse, France}
\email{siukaev@math.univ-paris13.fr}

\begin{document} 

\begin{abstract}
     Multidimensional continued fraction algorithms generalize the classical continued fraction setting, which includes both the Euclidean algorithm and the Gauss  algorithm (its projectivized version). We propose a novel method for computing invariant densities of certain multidimensional continued fraction algorithms. Inspired by Rauzy induction, our approach builds on the formalism of simplicial systems developed by Fougeron. We introduce a win-lose induction on a graph that is conjugate to the original algorithm, and construct its natural extension by introducing the notion of a dual graph. This method explicitly reconstructs the complete dynamics of the algorithm, yielding a partition of the invariant domain of the natural extension into pieces that map to one another. We further study the ergodic properties of the algorithms within this framework; in particular, we prove that the Modified Triangle algorithm in any dimension admits a unique ergodic measure equivalent to the Lebesgue measure.
\end{abstract}

\maketitle
\tableofcontents

\section{Introduction}

To estimate the dynamics of a \textit{multidimensional continued fraction 
(MCF) algorithm}, it is crucial to establish the existence of an ergodic 
invariant measure absolutely continuous with respect to the Lebesgue measure 
and, when possible, to obtain an explicit expression for it; see Lagarias' conditions~\cite[$\mathsection$~4]{lagarias} and Oseledets~\cite{oseledets} for the precise role of such a measure in the study of Diophantine approximation and the convergence of MCF algorithms.

A \textit{linear} (homogeneous) $d$-dimensional MCF algorithm acts on the positive cone $\mathbb{R}^{d+1}_+$ (or a subcone thereof) and can be locally given as a map
$$F: \textbf{x} \mapsto M^{-1}\cdot \textbf{x},$$
where $M$ is a non-negative integer matrix that depends on the region of $\mathbb{R}^{d+1}_+$ containing $\textbf{x}$ (see Section~\ref{sec:mcf} for a general definition). The associated \textit{projective} (non-homogeneous) version
$$f: \textbf{x} \mapsto \frac{M^{-1}\cdot \textbf{x}}{\|M^{-1}\cdot \textbf{x}\|}$$
acts on the standard simplex $\Delta^d := \big\{\textbf{x} \in \mathbb{R}^{d+1}_+\;\big|\; \sum_{i=0}^dx_i=1\big\}$.

To find the invariant measure of a particular MCF algorithm, a classical and powerful approach consists in constructing its \textit{natural extension}, a notion introduced by Rohlin~\cite{rohlin}. Recall that the natural extension of a measurable dynamical system $(X, T)$ is the minimal invertible dynamical system $(\hat{X}, \hat{T})$ of which $(X, T)$ is a factor under a surjective measure-preserving map; it is unique up to metric isomorphism~\cite{rohlin}. For continued fraction maps, the natural extension often admits a geometric model as a skew-product map on a product of cones, which makes it possible to read off the invariant density explicitly; see~\cite[$\mathsection$~3]{arnouxlabbe}  and~\cite[$\mathsection$~4]{arnouxschmidt}.

This strategy was developed in the one-dimensional setting by Arnoux and Schmidt\allowbreak~\cite{arnouxschmidt}, who introduced a heuristic method based on an iterated function system to construct planar models of natural extensions for piecewise homographic continued fraction maps, recovering classical results such as the Gauss measure and its natural extension due to Nakada--Ito--Tanaka~\cite{nakadaItotanaka}, and computing the invariant densities for several other maps. Their approach shows that if a naturally associated two-dimensional system has positive Lebesgue measure, then it is the natural extension of the original map, and its invariant measure is the marginal of the two-dimensional system.

In the multidimensional setting, Arnoux and Nogueira~\cite{arnouxnogueira} proposed a geometric framework for constructing natural extensions of MCF algorithms. Their method was further developed and systematized by Arnoux and Labbé~\cite{arnouxlabbe}, who construct the natural extension $\widehat{F}$ of the map $F$ as a skew-product acting on a product of positive cones:
$$\widehat{F}: (\textbf{x},\textbf{a}) \mapsto (M^{-1} \cdot \textbf{x},\; M^{T} \cdot \textbf{a}).$$
They find an invariant domain $D_{\mathrm{AL}} \subset \mathbb{R}_+^{d+1}\times\mathbb{R}_+^{d+1}$ and a partition
\begin{equation}\tag{1}\label{eq:0}
    D_{\mathrm{AL}} = \mathbb{R}_+^{d+1}\times\bigg[\bigsqcup_i\Lambda_i^*\bigg] = \bigg[\bigsqcup_j\Lambda_j\bigg]\times\Lambda^*\quad\quad\big(\mathbb{R}_{+}^{d+1} = \bigsqcup_j\Lambda_j\big)
\end{equation}
such that the restriction $\widehat{F}\big|_{D_{\mathrm{AL}}}$ is a bijection and $\widehat{F}\big(\Lambda_j\times\Lambda^*\big) = \mathbb{R}_+^{d+1}\times\Lambda_i^*$. The invariant measure for the projective MCF algorithm $f$ is then obtained by projecting the Lebesgue measure on $\mathbb{R}_+^{d+1}\times\mathbb{R}_+^{d+1}$ to $\Delta^d$ via the flow $\varphi_t: (\textbf{x},\textbf{a})\mapsto(e^t\textbf{x},e^{-t}\textbf{a})$; see~\cite[$\mathsection$~3]{arnouxlabbe}. Using this approach, Arnoux and Labbé compute the invariant densities for the Brun, Reverse, and Cassaigne algorithms explicitly, and discuss the difficulties encountered for other algorithms such as Jacobi--Perron and Arnoux--Rauzy--Poincaré, for which the invariant domain appears to have a fractal structure. Since the natural extension is unique up to metric isomorphism, any alternative construction that is metrically isomorphic to $\widehat{F}$ yields the same invariant measure; the interest of finding a different model lies in the possibility of simplifying the problem of identifying the invariant domain.

While MCF algorithms are typically defined on the positive cone or, in the projective case, on a simplex, the Rauzy–Veech induction, originally defined for interval exchange transformations (see~\cite{Rauzy1979, Veech1982}), operates on multiple copies of these domains, each labeled by a combinatorial datum. This construction generalizes to the broader notion of a \textit{Rauzy induction type algorithm}; see~\cite{chaikanogueiraerg}. The key observation is that in this setting, each copy of the domain can be associated with a vertex of a graph, and the induction rule describes transitions along the edges. The formalism of \textit{simplicial systems}, developed in~\cite{fougeronpp}, provides a rigorous framework connecting MCF algorithms to this graph-theoretic structure: given an MCF algorithm $F$, one constructs a labeled directed graph — the simplicial system — together with an algorithm on the graph, called the \textit{win-lose induction}, which is conjugate to $F$. Throughout this paper, we also refer to our approach as the \textit{graph method}, since a simplicial system is precisely a labeled graph associated with an MCF algorithm; we adopt this terminology consistently, without repeating the identification each time.

An important feature of the simplicial systems formalism is that it allows one to study the ergodic properties of MCF algorithms from a combinatorial point of view. In particular, the graph structure encodes information about the invariant measure and the ergodic behavior of the algorithm, making it possible to reduce analytic questions to combinatorial ones. This combinatorial perspective is one of the principal interests of the approach developed in this paper.

We construct a natural extension for the win-lose induction by introducing a \textit{dual graph}. As we will see in Section~\ref{sec:measure}, this extension is a bijection on the entire domain. However, the constructed extension does not automatically preserve the standard scalar product on $\mathbb{R}_+^{d+1}$. Since the win-lose induction operates on multiple copies of the positive cone, each associated with a vertex of the graph, one has the freedom to assign an independent bilinear form matrix to each vertex. Our approach consists in finding a labeling of the dual graph and a collection of bilinear form matrices, one for each vertex, such that the modified bilinear form is preserved under the win-lose induction. The preserved bilinear form matrices then correspond, via the connection to $\widehat{F}$ established in Section~\ref{sec:measure}, to the pieces of the partition~\eqref{eq:0} of the invariant domain $D_{\mathrm{AL}}$, which are mapped to one another under the algorithm. Thus, the graph method reduces the problem of finding the invariant domain to that of finding a compatible dual labeling — a purely combinatorial datum — rather than directly searching for the invariant domain in the product of cones, as developed in~\cite{arnouxschmidt, arnouxlabbe}. This method explicitly reconstructs the complete dynamics of the algorithm and yields a partition of the invariant domain into pieces, each corresponding to a specific vertex of the simplicial system.

Let us briefly describe the contents of this paper. In Section~\ref{sec:definitions}, we introduce the notion of a simplicial system and the win-lose induction. In Section~\ref{sec:measure}, we construct a natural extension for the win-lose induction via a dual graph, establish its connection to the natural extension of~\cite{arnouxnogueira, arnouxlabbe}, and explain how to derive invariant measures for both the linear and projective win-lose inductions. In Section~\ref{sec:mcf}, we recall the general definition of (non-ordered) MCF algorithms. In Section~\ref{sec:examples}, we apply the method to the Farey, Gauss, Selmer, and Triangle algorithms: we construct the corresponding simplicial systems and compute their invariant measures. In Section~\ref{sec:extss}, we introduce the notion of an extended simplicial system, construct it for the Brun and Modified Triangle algorithms, and compute invariant measures for the Brun algorithm in dimensions $d = 2$ and $d = 3$. We conclude the section by remarking that the Modified Triangle algorithm exhibits a markedly more intricate structure. As discussed in Section~\ref{sec:fractdomain}, we conjecture that the natural invariant domain for this algorithm possesses a fractal structure, in contrast to the simplex domains arising in the classical Triangle and Selmer cases. An analysis of this fractal domain, together with a computation of the corresponding invariant measure, is deferred to an upcoming paper, which will examine cases where the 
graph method yields extended simplicial systems with fractal invariant domains. Finally, in Section~\ref{sec:ergodicity}, we address ergodicity of the invariant measures constructed for each algorithm. For the Triangle algorithm, we prove that the invariant measure constructed in the present paper is the unique ergodic invariant measure equivalent to the Lebesgue measure. We note that this invariant measure coincides with the one recently obtained independently by Garrity and Lehmann Duke~\cite[$\mathsection$~6]{garrityerg} via a substantially different approach based on transfer operators, and motivated primarily by connections to the theory of integer partitions. For the Modified Triangle algorithm, we establish ergodicity of the invariant measure in arbitrary dimension.

\subsection*{Acknowledgments}

The author is deeply grateful to Charles Fougeron, Valérie Berthé, and 
Julien Barral for numerous long and fruitful discussions, without which 
this work would not have been possible. The author also thanks Thomas 
Garrity for explaining his approach to the study of the Triangle algorithm.

This work was supported by the FMSP grant funded by Qube Reseach \& Technologies, and partly by the 
ERC grant DynAMiCs (101167561) of the European Research Council.

\section{Definitions}\label{sec:definitions}

We follow Fougeron's ideas \cite{fougeronpp} and construct for a given MCF algorithm a graph called a \textit{simplicial system}, and then associate the MCF algorithm with a random walk on this graph. The vertices of the graph are copies of the simplex $\Delta^d$ in the projective (non-homogeneous) case, and copies of $\mathbb{R}_+^{d+1}$ in the linear (homogeneous) case. Its edges are directed and labeled (usually by numbers from $0$ to $d$). Following a particular edge (we say that this edge \textit{loses}) means subtracting the corresponding coordinate from all the coordinates which correspond to the other outgoing edges of the same vertex (we say that these edges \textit{win}). 

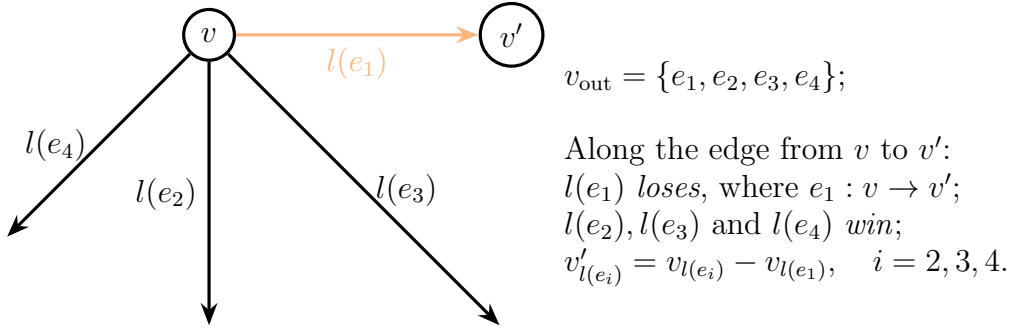
\begin{figure}
    \begin{minipage}[ht]{0.49\linewidth}
        \center{\begin{tikzpicture}[node distance={40mm}, > = Stealth, very thick, main/.style = {draw, circle}] 
        \node[main] (1) {$v$}; 
        \node[main] (2) [right of=1] {$v'$};
        \node (3) [below of=1] {$ $};
        \node (4) [below of=2] {$ $};
        \node (5) [below left of=1] {$ $};
        \draw[Apricot, arrows={->[scale=1]}, Apricot] (1) -- node[midway, below] {$l(e_1)$} (2);
        \draw[arrows={->[scale=1]}] (1) -- node[midway, left] {$l(e_2)$} (3);
        \draw[arrows={->[scale=1]}] (1) -- node[midway, right] {$l(e_3)$} (4);
        \draw[arrows={->[scale=1]}] (1) -- node[midway, left] {$l(e_4)$} (5);
        \end{tikzpicture}}
    \end{minipage}
    \hfill
    \begin{minipage}[ht]{0.49\linewidth}
        $v_{\text{out}} = \{e_1, e_2, e_3, e_4\}$; \\
        $ $ \\
        Along the edge from $v$ to $v'$:\\
        $l(e_1)$ \textit{loses}, where $e_1: v \to v'$;\\
        $l(e_2), l(e_3)$ and $l(e_4)$ \textit{win}; \\
        $v'_{l(e_i)} = v_{l(e_i)}-v_{l(e_1)},\quad i=2,3,4$.
    \end{minipage}
    \caption{The step of the win-lose induction along the edge $e_1: v \to v'$.}
    \label{fig:examplewlind}
\end{figure}

The terminology of winners and losers between the edges gives us the name \textit{win-lose induction}, which is inspired, as is all this approach, by the notion of \textit{Rauzy-Veech induction} (see \cite{rauzy}). This induction is a type of an algorithm which can be considered as an action on a graph (called the Rauzy graph), quite similar to the graph we informally defined before: the vertices are irreducible permutations of $d+1$ elements, and the edges are labeled according to the algorithm by labels of the coordinates. Note that the Gauss algorithm is both a MCF and a Rauzy-Veech induction at the same time (see \cite[Fig.~1]{fougeronpp} and Section~\ref{ssec:gauss}).

Let $G = (V,E)$ be a graph with a finite number of vertices and edges, and let the function 
$$l: E \rightarrow \{0,\ldots,d\}$$ 
map the edges to numbers in $\{0,\ldots,d\}$ (in the $d$-dimensional case). We require that this function be injective on $v_{\text{out}}$, the set of outgoing edges of any vertex $v\in V$. Let $e \in v_{\text{out}}$. Consider the partition (up to a set of measure zero) of the simplex
$$\Updelta_e^{d}=\Big\{(x_i)_{i=0,\ldots, d}\in\Updelta^{d}\;|\; x_{l(e)} < x_j \;\;\forall j\in l(v_{\text{out}}), j \neq l(e)\Big\}.$$
Consider the transition matrix $M_e = \Id + \sum\limits_{\substack{j\in l(v_{\text{out}})\\j\neq l(e)}}E_{j,l(e)}$. Note that 
$$M^{-1}_e\cdot(\mathbb{R}_+\cdot\Delta^d_e) = \mathbb{R}_+^{d+1}.$$
Let
$$\widetilde{T}: V\times\mathbb{R}_+^{d+1} \rightarrow V\times\mathbb{R}_+^{d+1}$$ 
be the map such that
$$\displaystyle \widetilde{T}(v, \textbf{x})=(v', \widetilde{T}_e(\textbf{x}
)), \;\; \widetilde{T}_e: \begin{cases}
\mathbb{R}_+\cdot\Updelta_e^{d} \rightarrow \mathbb{R}_+^{d+1} \\ 
\textbf{x} \mapsto M^{-1}_e \cdot \textbf{x}
\end{cases} \text{for} \;\; \textbf{x}\in\mathbb{R}_+\cdot\Updelta_e^{d}, \; e: v \rightarrow v'.$$
If $|v_{\text{out}}|>1$, we say that the label $i \in \{0,\ldots,d\}$ \textit{loses} on the edge $e \in v_{\text{out}}$ if $l(e) = i$ and \textit{wins} on the edge $e$ if there exists $e'\in v_{\text{out}}$ \; such that \; $l(e) \neq l(e')=i$.
\begin{definition}\label{def:ss}
    \cite[$\mathsection$~2.1]{fougeronpp}
    The graph $G$ is called a \textit{simplicial system}, and the algorithm $\widehat{T}$ is called an associated \textit{linear} (or homogeneous) \textit{win-lose induction}. 
    
    We are also able to define a \textit{projective} (or non-homogeneous) win-lose induction $T: V\times\Updelta^d \rightarrow V\times\Updelta^d$ on the same graph:
    $$\displaystyle T(v, \textbf{x})=(v', T_e(\textbf{x})), \quad\quad T_e: \begin{cases}
    \Updelta_e^{d} \rightarrow \Updelta^d \\ 
    \textbf{x} \mapsto \frac{M^{-1}_e \cdot \textbf{x}}{\|M^{-1}_e \cdot \textbf{x}\|_1}
    \end{cases} \text{for} \;\; \textbf{x}\in\Updelta_e^{d}, \; e: v \rightarrow v'.$$
\end{definition}

\begin{remark}
    We use the notation $\Delta^d_{i(j_1,\ldots,j_k)} := \{\textbf{x}\in\Delta^d|\; x_i<x_{j_m},\;m=1,\ldots, k\}$.
\end{remark}

Figure~\ref{fig:wlind2out} illustrates the choice of the edge we follow depending on the position of the vector in the simplex. In this example, $d=2$ and we start from a vertex with $2$ outgoing edges.

\begin{figure}
    \begin{minipage}[ht]{0.49\linewidth}
        \center{\begin{tikzpicture}[node distance={40mm}, > = Stealth, very thick, main/.style = {draw, circle}] 
        \node[main][fill=lightgray] (1) {$\Delta^{d}$}; 
        \node[main][fill=lightgray] (2) [right of=1] {$\Delta^{d}$};
        \node[main][fill=lightgray] (3) [below of=1] {$\Delta^{d}$};
        \draw[SpringGreen, arrows={->[scale=1]}, SpringGreen] (1) -- node[midway, below] {$1$} (2);
        \draw[Apricot, arrows={->[scale=1]}, Apricot] (1) -- node[midway, left] {$2$} (3);
        \end{tikzpicture}}
    \end{minipage}
    \hfill
    \begin{minipage}[ht]{0.49\linewidth}
        \center{\begin{tikzpicture}[yscale=0.7,xscale=0.7]
        \fill[Apricot] (-2,0)--(0,0)--(0,3)--cycle;
        \fill[SpringGreen] (2,0)--(0,0)--(0,3)--cycle;
        \fill[lightgray] (4,0)--(6,3)--(8,0)--cycle;
        \node at (-1,0.3){$1>2$};
        \node at (1,0.3){$1<2$};
        \node at (-2,0)[below left]{$1$};
        \node at (2,0)[below right]{$2$};
        \node at (0,3)[above]{$3$};
        \draw(-2,0)--(0,3)--(2,0)--cycle;
        \draw(0,3)--(0,0);
        \draw(4,0)--(6,3)--(8,0)--cycle;
        \draw[->] (1.5,1.5)-- node[midway, above] {$T\big|_{\Delta^{d}}$} (4.5,1.5);
        \node at (2.5,3.5)[right]{$\Delta^{d}_{1(2)}$};
        \draw[->] (2.5,3.5) to [out=180, in=40, looseness=1] (0.5,1);
        \node at (-2.5,3.5)[left]{$\Delta^{d}_{2(1)}$};
        \draw[->] (-2.5,3.5) to [out=0, in=140, looseness=1] (-0.5,1);
        \node at (6,1){$\Delta^{d}$};
        \end{tikzpicture}   \\}
    \end{minipage}
    \caption{A step of the win-lose induction from a vertex with $2$ outgoing edges.}
    \label{fig:wlind2out}
\end{figure}
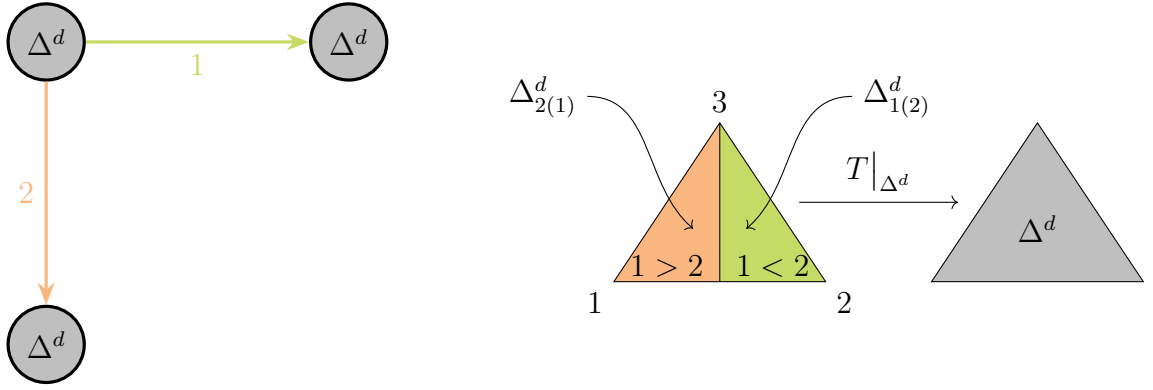

Figure~\ref{fig:wlind3out} is an example for $d=2$ again, but we start from a vertex with $3$ outgoing edges.

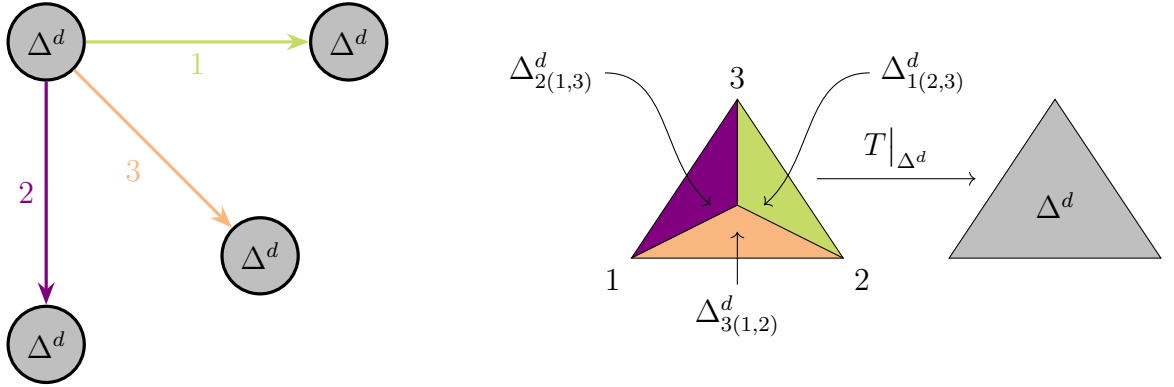
\begin{figure}
    \begin{minipage}[ht]{0.49\linewidth}
        \center{\begin{tikzpicture}[node distance={40mm}, > = Stealth, very thick, main/.style = {draw, circle}] 
        \node[main][fill=lightgray] (1) {$\Delta^{d}$}; 
        \node[main][fill=lightgray] (2) [right of=1] {$\Delta^{d}$};
        \node[main][fill=lightgray] (3) [below of=1] {$\Delta^{d}$};
        \node[main][fill=lightgray] (4) [below right of=1] {$\Delta^{d}$};
        \draw[SpringGreen, arrows={->[scale=1]}, SpringGreen] (1) -- node[midway, below] {$1$} (2);
        \draw[violet, arrows={->[scale=1]}, violet] (1) -- node[midway, left] {$2$} (3);
        \draw[Apricot, arrows={->[scale=1]}, Apricot] (1) -- node[midway, below left] {$3$} (4);
        \end{tikzpicture}}
    \end{minipage}
    \hfill
    \begin{minipage}[ht]{0.49\linewidth}
        \center{\begin{tikzpicture}[yscale=0.7,xscale=0.7]
        \fill[violet] (-2,0)--(0,1)--(0,3)--cycle;
        \fill[SpringGreen] (2,0)--(0,1)--(0,3)--cycle;
        \fill[Apricot] (-2,0)--(0,1)--(2,0)--cycle;
        \fill[lightgray] (4,0)--(6,3)--(8,0)--cycle;
        \node at (-2,0)[below left]{$1$};
        \node at (2,0)[below right]{$2$};
        \node at (0,3)[above]{$3$};
        \draw(-2,0)--(0,3)--(2,0)--cycle;
        \draw(0,3)--(0,1);
        \draw(-2,0)--(0,1);
        \draw(2,0)--(0,1);
        \draw(4,0)--(6,3)--(8,0)--cycle;
        \draw[->] (1.5,1.5)-- node[midway, above] {$T\big|_{\Delta^{d}}$} (4.5,1.5);
        \node at (2.5,3.5)[right]{$\Delta^{d}_{1(2,3)}$};
        \draw[->] (2.5,3.5) to [out=180, in=40, looseness=1] (0.5,1);
        \node at (-2.5,3.5)[left]{$\Delta^{d}_{2(1,3)}$};
        \draw[->] (0,-0.5) to (0,0.5);
        \node at (0,-0.5)[below]{$\Delta^{d}_{3(1,2)}$};
        \draw[->] (-2.5,3.5) to [out=0, in=140, looseness=1] (-0.5,1);
        \node at (6,1){$\Delta^{d}$};
        \end{tikzpicture}   \\}
    \end{minipage}
    \caption{A step of the win-lose induction from a vertex with $3$ outgoing edges.}
    \label{fig:wlind3out}
\end{figure}

\begin{remark}
    In the case of a linear win-lose induction, there is a copy of the cone $\mathbb{R}_+^{d+1}$ in each vertex instead of $\Updelta^d$ for the projective version.
\end{remark}

\section{Constructing the measure from the dual graph}\label{sec:measure}

We define an invariant measure and a density for algorithms on simplicial systems. For the simplicial system $G=(V,E)$ and the associated win-lose induction $T$ (linear or projective), we say that 
$$\mu: V \times \mathcal{B}\big(\Delta\big) \to \mathbb{R}_+,\quad \text{where $\Delta$ is a parameter space and $\mathcal{B}\big(\Delta\big) :=\{A|\;A\subset \Delta\}$},$$
is an \textit{invariant measure} if $\mu\big(T^{-1}(X)\big) = \mu (X)$ for any measurable set $X\subset V\times\Delta$ and is \textit{pseudo-homogeneous} if $\mu(\alpha X) = \mu(X)$ for any $\alpha\in\mathbb{R}_+$, where 
$$\alpha X = \big\{(v,\alpha \textbf{x}) \in V \times \Delta|\;\;(v,\textbf{x}) \in X\big\}.$$ 
If an invariant measure is absolutely continuous with respect to the Lebesgue measure, then its density function $h: V \times \Delta\to \mathbb{R}_+$ is called an \textit{invariant density}.
\subsection{Dual graph and natural extension}\label{sec:dualss}
Let us consider the simplicial system $G$ and the corresponding linear win-lose induction $\widetilde{T}$, i.e.,
$$\widetilde{T}(v, \textbf{x})=(v', \widetilde{T}_e(\textbf{x})).$$
This map is a priori not invertible, so we want to construct an invertible extension $\widehat{T}$ of $\widetilde{T}$.  We define the \textit{dual} (for $G$) simplicial system $S = (V,E')$ as a directed labeled graph whose vertices are the same as in $G$, whose edges $e' \in E'$ are inverse compared to $G$, but whose labeling does not depend on $G$ (our aim is to find, if it exists, a ''suitable'' labeling for $S$ using the corresponding ''inverse'' linear win-lose induction ''$\widetilde{T}^{-1}$''; see the detailed explanation below). Using this dual graph $S$, we define the map

\begin{center}
    $\displaystyle \widehat{T}: V\times\mathbb{R}_+^{d+1}\times\mathbb{R}_+^{d+1} \to V\times\mathbb{R}_+^{d+1}\times\mathbb{R}_+^{d+1}$, \\
    $ $ \\
    $\widehat{T}(v, \textbf{x}, \textbf{a})=(v', \widetilde{T}_e \textbf{x}, \widetilde{T}^{-1}_{e'} \textbf{a})$, 
\end{center}
where
\begin{center}
    $\displaystyle \widetilde{T}^{-1}_{e'} \textbf{a} : \begin{cases}
     \mathbb{R}_+^{d+1} \rightarrow \mathbb{R}_+\cdot\Updelta_{e'}^d, \\ 
    \textbf{a} \mapsto M_{e'} \textbf{a}
    \end{cases} \text{for } e': v' \rightarrow v$ in $E'$.
\end{center}

\begin{remark}
    We refer the reader to Section~\ref{sec:definitions} for the definition of $\widetilde{T}_e \textbf{x}$.
\end{remark}

The realization $\widehat{T}$ of the natural extension transforms the triple $(v, \textbf{x}, \textbf{a})$ into $(v', M^{-1}_e\textbf{x}, M_{e'}\textbf{a})$ for $\textbf{x} \in \mathbb{R}_+\cdot\Delta_e^d$ and is invertible on $V\times\mathbb{R}_+^{d+1}\times\mathbb{R}_+^{d+1}$. 

In Figure~\ref{fig:exampledualss}, we follow the edge $e:v\to v'$ on $G$ and the edge $e':v'\to v$ on $S$. The extension $\widehat{T}$ sends the triple $(v,\textbf{x},\textbf{a})$ to the triple $\Big(v',\Big[\begin{smallmatrix} 1 & 0 & 0\\ -1 & 1 & 0\\ 0 & 0 & 1 \end{smallmatrix}\Big]\textbf{x}, \Big[ \begin{smallmatrix} 1 & 0 & 0\\ 0 & 1 & 1\\ 0 & 0 & 1 \end{smallmatrix}\Big]\textbf{a}\Big)$.

\begin{figure}
    \centering
    \begin{tikzpicture}[yscale=0.7,xscale=0.7, > = Stealth, very thick]
        \node at (2,2){$G$};
        \draw[Apricot] (0,0) circle (0.5);
        \node[Apricot] at (0,0){$v$};
        \draw (4,0) circle (0.5);
        \node at (4,0.08){$v'$};
        \draw[Apricot, arrows={->[scale=2]}, Apricot] (0.5,0) -- node[midway, below] {$1$} (3.5,0);
        \draw[->] (0.35355,-0.35355) to node[midway, below left] {$2$} (1.5,-1.5);
        \draw[->] (-2,0.5) to [out=0, in=145, looseness=1] node[midway, above] {$3$}(-0.35355,0.35355);
        \draw[->] (-2,-2) to node[midway, above left] {$2$} (-0.35355,-0.35355);
        \draw[->] (4,0.5) to node[midway, right] {$1$} (4,1.8);
        \draw[->] (5.5,-1) to (4.35355,-0.35355);
        \node at (5.1,-0.4){$2$};
        \draw[->] (3.64644,-0.35355) to (2.5,-1.5);
        \node at (3.5,-1){$3$};
        \node at (12,2){$S$};
        \draw (10,0) circle (0.5);
        \node at (10,0){$v$};
        \draw[Apricot] (14,0) circle (0.5);
        \node[Apricot] at (14,0.08){$v'$};
        \draw[Apricot, arrows={->[scale=2]}, Apricot] (13.5,0) -- node[midway, above] {$3$} (10.5,0);
        \draw[->] (11.5,-1.5) to node[midway, below left] {$2$} (10.35355,-0.35355);
        \draw[->] (9.64644,0.35355) to [out=145, in=0, looseness=1] node[midway, above] {$3$}(8,0.5);
        \draw[->] (9.64644,-0.35355) to node[midway, above left] {$2$} (8,-2);
        \draw[->] (14,1.8) to node[midway, right] {$1$} (14,0.5);
        \draw[->] (14.35355,-0.35355) to (15.5,-1);
        \node at (15.1,-0.4){$2$};
        \draw[->] (12.5,-1.5) to (13.64644,-0.35355);
        \node at (13.5,-1){$3$};
    \end{tikzpicture} \\
    $ $ \\
    \begin{tikzpicture}[yscale=0.7,xscale=0.7]
        \fill[Apricot] (2,0)--(0,0)--(0,3)--cycle;
        \fill[lightgray] (4,0)--(6,3)--(8,0)--cycle;
        \node at (1,0.3){$1<2$};
        \node[below] at (-0.15,0){$x_3=0$};
        \node[above left, rotate=56.30993247402] at (-0.4,2.4){$x_2=0$};
        \node[above right, rotate=303.690067526] at (0.4,2.4){$x_1=0$};
        \draw(-2,0)--(0,3)--(2,0)--cycle;
        \draw(0,3)--(0,0);
        \draw(4,0)--(6,3)--(8,0)--cycle;
        \node at (3,1.5){$\times$};
        \node at (2.5,3.5)[right]{$\mathbb{R}_+\cdot\Delta^{d}_{1(2)}$};
        \draw[->] (2.5,3.5) to [out=180, in=110, looseness=1.5] (0.5,1);
        \node at (6,1){$\mathbb{R}_+\cdot\Delta^{d}$};
        \node[below] at (5.85,0){$a_3=0$};
        \node[above left, rotate=56.30993247402] at (5.6,2.4){$a_2=0$};
        \node[above right, rotate=303.690067526] at (6.4,2.4){$a_1=0$};
        \node at (-4.5,1.5){$\{v\}$};
        \node at (-3,1.5){$\times$};
        \end{tikzpicture} \\
        \begin{tikzpicture}[yscale=0.7,xscale=0.7]
        \draw[->] (0,0)-- node[midway, right] {$\widehat{T}$} (0,-2);
        \draw[white](-2.5,0)--(0,0);
    \end{tikzpicture} \\
    \begin{tikzpicture}[yscale=0.7,xscale=0.7]
        \fill[lightgray] (-2,0)--(2,0)--(0,3)--cycle;
        \fill[Apricot] (4,0)--(7,1.5)--(8,0)--cycle;
        \node[below] at (-0.15,0){$x_3=0$};
        \node[above left, rotate=56.30993247402] at (-0.4,2.4){$x_2=0$};
        \node[above right, rotate=303.690067526] at (0.4,2.4){$x_1=0$};
        \draw(-2,0)--(0,3)--(2,0)--cycle;
        \draw(4,0)--(6,3)--(8,0)--cycle;
        \node at (3,1.5){$\times$};
        \node at (0,1){$\mathbb{R}_+\cdot\Delta^{d}$};
        \node[below] at (5.85,0){$a_3=0$};
        \node[above left, rotate=56.30993247402] at (5.6,2.4){$a_2=0$};
        \node[above right, rotate=303.690067526] at (6.4,2.4){$a_1=0$};
        \draw(4,0)--(7,1.5);
        \node at (6,0.3){$2>3$};
        \node at (8.5,3.5)[right]{$\mathbb{R}_+\cdot\Delta^{d}_{3(2)}$};
        \draw[->] (8.5,3.5) to [out=250, in=330, looseness=1.5] (7,0.5);
        \draw[white](-8.3,0)--(-2,0);
        \node at (-4.5,1.5){$\{v'\}$};
        \node at (-3,1.5){$\times$};
    \end{tikzpicture}
    \caption{A step of the invertible extension of the linear win-lose induction.}
    \label{fig:exampledualss}
\end{figure}
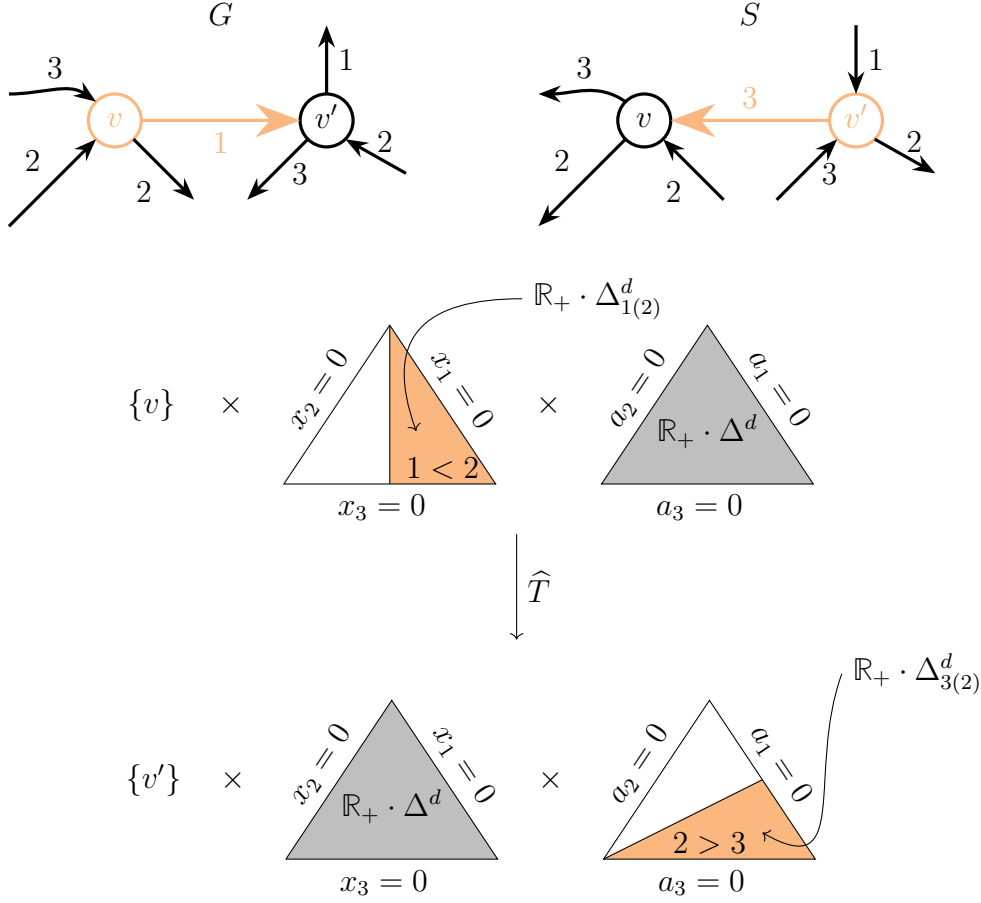

Note that $\widehat{T}$ preserves the Lebesgue measure since it is a bijection 
with Jacobian equal to $1$; but $\widehat{T}$ does not preserve the standard scalar product $\langle \textbf{x},\textbf{a} \rangle = \textbf{x}^T\cdot \textbf{a}$, and the Lebesgue measure on $V\times\mathbb{R}_+^{d+1}\times\mathbb{R}_+^{d+1}$ is not very useful, because even its projection onto the first two coordinates has infinite density. Let us restrict the domain. To do this, we define the bilinear form associated with $\widehat{T}$ as
$$\displaystyle \langle \textbf{x},\textbf{a}\rangle_{v} := \textbf{x}^T\Omega_v\textbf{a}\quad\quad\text{if $\textbf{x}$ is in the vertex $v$},$$
where $\Omega_v$ is a non-negative (not necessarily symmetrical) matrix of size $(d+1)\times (d+1)$. This new bilinear form is preserved if and only if for any pair $v,v' \in V$ the following equality holds:
$$\displaystyle \textbf{x}^T\Omega_v\textbf{a} =  \textbf{x}^T\big(M^{-1}_e\big)^T\Omega_{v'}M_{e'}\textbf{a}.$$
Since it must hold for arbitrary vectors $\textbf{x}$ and $\textbf{a}$, one has
$$\displaystyle\Omega_v = \big(M^{-1}_e\big)^T\Omega_{v'}M_{e'},$$
\begin{equation}\tag{2}\label{eq:1}
    \displaystyle M^T_e\Omega_v = \Omega_{v'}M_{e'}\quad \text{for any pair } v,v' \in V.
\end{equation}

Note that the map $\widehat{T}$ is determined by a choice of solution $(\Omega_v)_{v\in V}$ to~(\ref{eq:1}); we therefore refer to it as the \textit{realization} of the natural extension. Let us understand the connection between $\widehat{T}$ and the natural extension $\widehat{T}_{\text{AL}}$, which was defined in \cite[$\mathsection$~3]{arnouxlabbe}:
\begin{center}
    $\displaystyle \widehat{T}_{AL}: V\times\mathbb{R}_+^{d+1}\times\mathbb{R}_+^{d+1} \to V\times\mathbb{R}_+^{d+1}\times\mathbb{R}_+^{d+1}$, \\
    $\widehat{T}_{AL}(v, \textbf{x}, \textbf{a})=(v', M^{-1}_e \textbf{x}, M_e^T \textbf{a})$.
\end{center}

The map $\widehat{T}_{\text{AL}}$ has Jacobian $1$ and preserves the Lebesgue measure. We recall that the main problem in \cite{arnouxlabbe} is to find a domain $D_{\text{AL}} \subset  V\times\mathbb{R}_+^{d+1}\times\mathbb{R}_+^{d+1}$ on which $\widehat{T}_{\text{AL}}$ is a bijection (see \cite[$\mathsection$~3, Remark~5]{arnouxlabbe}).

\begin{proposition}
    Assume that $\big(\Omega_{v}\big)_{v\in V}$ is a solution of the system of equations~(\ref{eq:1}), and that every matrix $\Omega_v$ has full rank. Then $\widehat{T}$ is a bijection on
    $$\displaystyle D_{\text{AL}} = V\times \mathbb{R}_+^{d+1}\times\bigsqcup_{v\in V}\Omega_v\mathbb{R}_+^{d+1}.$$
\end{proposition}

\begin{proof}
    Consider the change of coordinates $\textbf{b}(\textbf{a}) := \Omega_v\cdot \textbf{a}$. Then
    \begin{align}
        \widehat{T}_{AL}\big(v,\textbf{x},\textbf{b}(\textbf{a})\big) 
        &= \widehat{T}_{AL}(v,\textbf{x},\Omega_v\cdot \textbf{a}) 
        = (v', M^{-1}_e \textbf{x}, M_e^T\Omega_v\cdot \textbf{a}) \notag\\
        &= (v', M^{-1}_e \textbf{x}, \Omega_{v'}M_{e'}\cdot \textbf{a}) 
        = \big(v', M^{-1}_e \textbf{x}, \textbf{b}(\widehat{T}(\textbf{a}))\big).
    \end{align}
    This means that the diagram \\ 
    \begin{minipage}[c]{1\textwidth}
        \centering
        \begin{tikzpicture}[node distance={25mm}]
        \centering
            \node (1) {$V\times\mathbb{R}_+^d\times\mathbb{R}_+^d$};
            \node (2) [right of=1] {};
            \node (3) [below of=1] {$\displaystyle V\times \mathbb{R}_+^d\times\bigsqcup_{v\in V}\Omega_v\mathbb{R}_+^d$};
            \node (4) [right of=2] {$V\times\mathbb{R}_+^d\times\mathbb{R}_+^d$};
            \node (5) [below of=4] {$\displaystyle V\times \mathbb{R}_+^d\times\bigsqcup_{v\in V}\Omega_v\mathbb{R}_+^d$};
            \draw[->] (1) -- node[midway, above] {$\widehat{T}$} (4);
            \draw[->] (3) -- node[midway, above] {$\widehat{T}_{\text{AL}}$} (5);
            \draw[->] (1) -- node[midway, left] {$\Id\times \textbf{b}$} (3);
            \draw[->] (4) -- node[midway, right] {$\Id \times \textbf{b}$} (5);
        \end{tikzpicture}
    \end{minipage}\\ 
    is commutative, and $\widehat{T}_{\text{AL}}$ is a bijection on $V\times \mathbb{R}_+^{d+1}\times\bigsqcup_{v\in V}\Omega_v\mathbb{R}_+^{d+1}$, since both $\widehat{T}$ and $\textbf{b}$ are bijections.
\end{proof}

\begin{remark}
    Another possibility for constructing a natural extension is described by Mercat~\cite{mercat} for matrix graphs, which also include simplicial systems. In this case, the labels in the dual graph remain the same as in the initial one, and then the determinization of the dual graph gives the part of the invariant domain for the corresponding vertex (that coincides with our bilinear form matrices).
\end{remark}

\subsection{Invariant measure for the natural extension}
We start constructing invariant measures from the natural extension $\widehat{T}$.

\begin{lemma}
    Assume that there is a labeling $E'$ for which the system of equations~(\ref{eq:1}) has a non-zero solution $\big(\Omega_{v}\big)_{v\in V}$. Then the natural extension $\widehat{T}$ preserves the hypersurface 
    $$\displaystyle H:= \big\{(v,\textbf{x},\textbf{a})\in V\times\mathbb{R}_+^{d+1}\times\mathbb{R}_+^{d+1}\big|\;\;\langle \textbf{x},\textbf{a}\rangle_{v}=1\big\},$$
    i.e., $\widehat{T}(H) = H$, and there is a $\widehat{T}$-invariant measure $\mu_{\widehat{T}}$ on $H$ defined as
    $$\mu_{\widehat{T}}(A) := \text{Leb}(\text{C}_H A)\quad\text{for }A\subset H,$$
    where $\text{C}_{H} A := \big\{\{v\}\times(\alpha\textbf{x},\alpha\textbf{a})\in V\times\mathbb{R}_+^{d+1}\times\mathbb{R}_+^{d+1}\big|\;\;(v,\textbf{x},\textbf{a})\in A,\;\alpha\in[0,1]\big\}$ is the subcone of finite $\mathbb{R}^{2d+2}$-Lebesgue measure corresponding to $A$.
\end{lemma}

\begin{proof}
    The hypersurface $H$ is $\widehat{T}$-invariant by the definition of the matrices $\big(\Omega_{v}\big)_{v\in V}$. Indeed, if $(v,\textbf{x},\textbf{a}) \in H$, then
    $$\big(M^{-1}_e\textbf{x}\big)^T\Omega_{v'}M_{e'}\textbf{a} = \textbf{x}^T\Omega_{v}\textbf{a} = 1,$$
    hence $\widehat{T}(v,\textbf{x},\textbf{a})\in H$. Note that $H$ is transverse to the cones in $V\times\mathbb{R}^{d+1}_+\times\mathbb{R}^{d+1}_+$, i.e.,
    $$\displaystyle \big|H\cap \big\{(v,\lambda\textbf{x},\lambda\textbf{a})|\;\lambda\geq0\big\}\big| = 1\quad\quad\text{for any }(v,\textbf{x},\textbf{a})\in V\times\mathbb{R}^{d+1}_+\times\mathbb{R}^{d+1}_+.$$
    Then the measure $\mu_{\widehat{T}}$ is $\widehat{T}$-invariant because $\widehat{T}$ is a bijection which preserves $H$ and preserves the Lebesgue measure: 
    $$\displaystyle \mu_{\widehat{T}}\big(\widehat{T}^{-1}A\big) = \text{Leb}\big(\text{C}_H\widehat{T}^{-1}A\big)=\text{Leb}\big(\widehat{T}^{-1}(\text{C}_HA)\big)=\text{Leb}(\text{C}_H A)=\mu_{\widehat{T}}(A)\quad\quad\text{for }A\subset H.$$
\end{proof}

\begin{figure}[ht]
    \center{\begin{tikzpicture}[>=Stealth, scale=1.5]
        \draw[->, thick] (0,0) -- (7,0) node[below] {$\mathbf{x}$};
        \draw[->, thick] (0,0) -- (0,5) node[left] {$\mathbf{a}$};
        \fill[green!30] (0.5,4) .. controls (0.3,2.4) and (2.5,1) .. (4.5,1.2)
                -- (5.8, 2.9) .. controls (5,2.7) and (3.5,4) .. (2,5.6)
                -- (0.5,4) -- cycle;
        \node at (1.5,4) {\Large H};
        \fill[blue!30, opacity=0.7] (0,0) -- (1,3.2) -- (2,2.7) -- cycle;
        \fill[blue!30, opacity=0.7] (0,0) -- (3.7,2.95) -- (4.3,2.2) -- cycle;
        \node[rotate=60] at (0.8,1.6) {\tiny$C_H \hat{T}^{-1}(\mathbf{A})$};
        \node[rotate=37] at (2.1,1.4) {\tiny$C_H \mathbf{A}$};
        \draw[black, thick,dashed] (0,0) -- (1,3.2);
        \draw[black, thick,dashed] (0,0) -- (2,2.7);
        \draw[black, thick,dashed] (0,0) -- (3.7,2.95);
        \draw[black, thick,dashed] (0,0) -- (4.3,2.2);
        \fill[orange!40] (1.0,3.2) .. controls (1.3,3.6) and (1.7,3.7) .. (2.0,3.5)
                 .. controls (2.2,3.3) and (2.2,2.9) .. (2.0,2.7)
                 .. controls (1.8,2.5) and (1.4,2.5) .. (1.2,2.7)
                 .. controls (1.0,2.9) and (0.9,3.0) .. (1.0,3.2);
        \draw[red, thick] (1.0,3.2) .. controls (1.3,3.6) and (1.7,3.7) .. (2.0,3.5)
                  .. controls (2.2,3.3) and (2.2,2.9) .. (2.0,2.7)
                  .. controls (1.8,2.5) and (1.4,2.5) .. (1.2,2.7)
                  .. controls (1.0,2.9) and (0.9,3.0) .. (1.0,3.2);
        \node[red] at (1.5,3.1) {$\hat{T}^{-1}(\mathbf{A})$};
        \fill[orange!40] (3.5,2.8) .. controls (3.8,3.2) and (4.2,3.3) .. (4.4,3.1)
                 .. controls (4.6,2.9) and (4.6,2.5) .. (4.4,2.3)
                 .. controls (4.2,2.1) and (3.8,2.2) .. (3.5,2.4)
                 .. controls (3.3,2.5) and (3.3,2.6) .. (3.5,2.8);
        \draw[red, thick] (3.5,2.8) .. controls (3.8,3.2) and (4.2,3.3) .. (4.4,3.1)
                  .. controls (4.6,2.9) and (4.6,2.5) .. (4.4,2.3)
                  .. controls (4.2,2.1) and (3.8,2.2) .. (3.5,2.4)
                  .. controls (3.3,2.5) and (3.3,2.6) .. (3.5,2.8);
        \node[red] at (4.1,2.7) {$\mathbf{A}$};
        \draw[->, very thick] (1.5,2.8) -- (3.7,2.6) node[midway, above] {$\hat{T}$};
        \draw[->, very thick] (1.1,2.2) -- (3.3,2.0) node[midway, above] {$\hat{T}$};
        \node[align=left] at (5.5,4.3) {$A \subset H$\\
                                 $\mu_{\hat{T}} := Leb(C_H A)$\\
                                 $\hat{T}_* \mu_{\hat{T}} = \mu_{\hat{T}}$};
    \end{tikzpicture}}
    \caption{The action of $\widehat{T}$ on cones preserves the hypersurface $H$.}
    \label{fig:actioncones}
\end{figure}
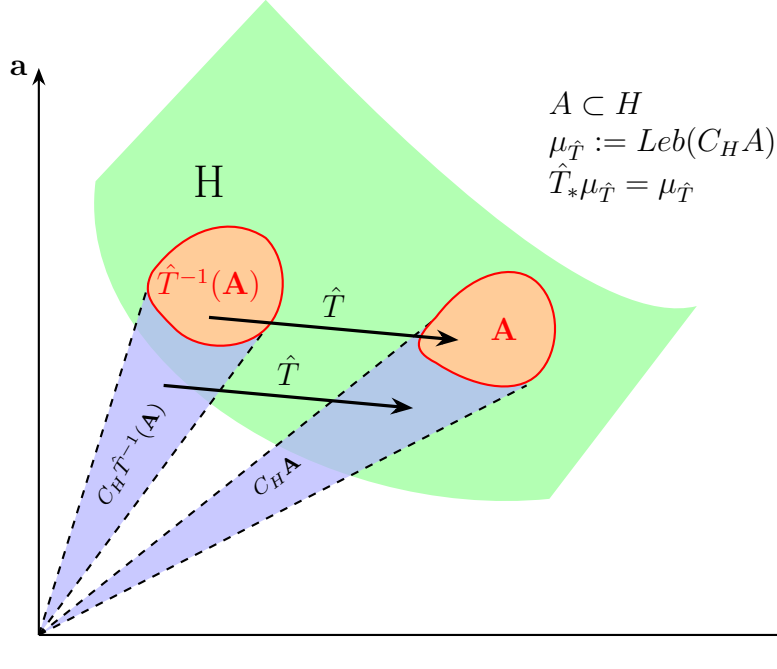

\subsection{Invariant measure for the win-lose induction}
To come back to the linear win-lose induction $\widetilde{T}$, we project the measure $\mu_{\widehat{T}}$, which is invariant for the natural extension, onto the first two coordinates and then construct a $\widetilde{T}$-invariant measure with finite density. 

\begin{lemma}
    Consider the projection $\widehat{\pi}$ defined as\\ \\
    \begin{minipage}[ht]{0.49\linewidth}
        \centering
        $\widehat{\pi}: H \to V\times\mathbb{R}_+^{d+1},$ \\
        $\widehat{\pi}(v,\textbf{x},\textbf{a})=(v,\textbf{x}),$
    \end{minipage}
    \hfill
    \begin{minipage}[ht]{0.49\linewidth}
        \begin{tikzpicture}[node distance={25mm}]
            \centering
            \node (1) {$H$}; 
            \node (2) [right of=1] {$H$};
            \node (3) [below of=1] {$V \times \mathbb{R}_+^{d+1}$};
            \node (4) [below of=2] {$V \times \mathbb{R}_+^{d+1}$};
            \draw[->] (1) -- node[midway, above] {$\widehat{T}\big|_H$} (2);
            \draw[->] (3) -- node[midway, above] {$\widetilde{T}$} (4);
            \draw[->] (1) -- node[midway, left] {$\widehat{\pi}$} (3);
            \draw[->] (2) -- node[midway, right] {$\widehat{\pi}$} (4);
        \end{tikzpicture}
    \end{minipage} \\ \\
    and the measure $\mu_{\widetilde{T}}$ on $V\times\mathbb{R}_+^{d+1}$ defined as 
    $$\mu_{\widetilde{T}}=\widehat{\pi}_*\mu_{\widehat{T}}(X) = \mu_{\widehat{T}}(\widehat{\pi}^{-1}X)\quad\text{for }X \in V\times\mathbb{R}_+^{d+1},$$
    where $\widehat{\pi}^{-1}X = \big\{(v,\textbf{x},\textbf{a})|\;(v,\textbf{x})\in X, \textbf{a}\in \mathbb{R}_+^{d+1},\;\langle \textbf{x},\textbf{a} \rangle_v=1\big\}$.\\
    The measure $\mu_{\widetilde{T}}$ is $\widetilde{T}$-invariant and pseudo-homogeneous.
\end{lemma}

\begin{proof}
    The measure $\mu_{\widetilde{T}}=\widehat{\pi}_{*}\mu_{\widehat{T}}$ is $\widetilde{T}$-invariant because $\widetilde{T} = \widehat{\pi}\circ\widehat{T}$ and $\mu_{\widehat{T}}$ is $\widehat{T}$-invariant. It is also pseudo-homogeneous, i.e., $\mu_{\widetilde{T}}(\alpha X) = \mu_{\widetilde{T}}(X)$ for any $\alpha>0$ and $X\in V\times\mathbb{R}_+^{d+1}$: 
    \begin{align*}
        \mu_{\widetilde{T}}(\alpha X) 
        &= \int_{\{(v,\,\mathbf{y}=\lambda\alpha\mathbf{x},\,\mathbf{b}=\lambda\mathbf{a}):\;
           (v,\mathbf{x})\in X,\;
           \mathbf{a}\in \mathbb{R}_+^{d+1},\;
           \langle \alpha\mathbf{x},\mathbf{a} \rangle_v=1,\;
           \lambda>0\}}
           d\mathbf{b}\,d\mathbf{y} \\
        &= [\mathbf{y}=\alpha\mathbf{u}] \\
        &= \int_{\{(v,\,\mathbf{u}=\lambda\mathbf{x},\,\mathbf{b}=\lambda\mathbf{a}):\;
           (v,\mathbf{x})\in X,\;
           \mathbf{a}\in \mathbb{R}_+^{d+1},\;
           \alpha\langle \mathbf{x},\mathbf{a} \rangle_v=1,\;
           \lambda>0\}}
           \alpha^{d+1}\,d\mathbf{b}\,d\mathbf{u} \\
        &= [\mathbf{b}=\alpha^{-1}\mathbf{c}] \\
        &= \int_{\{(v,\,\mathbf{u}=\lambda\mathbf{x},\,\mathbf{c}=\lambda\mathbf{a}):\;
           (v,\mathbf{x})\in X,\;
           \mathbf{a}\in \mathbb{R}_+^{d+1},\;
           \langle \mathbf{x},\mathbf{a} \rangle_v=1,\;
           \lambda>0\}}
           d\mathbf{c}\,d\mathbf{u} \\
        &= \mu_{\widetilde{T}}(X).
    \end{align*}
\end{proof}
 
To calculate the invariant density $h_{\widetilde{T}}$, we integrate $\mu_{\widehat{T}}$ over all possible $\textbf{a}$ such that $\langle \textbf{x},\textbf{a}\rangle_{v} = 1$ for fixed $\textbf{x}$ and $v$:
\begin{align*}
    \mu_{\widetilde{T}}(X) 
    &= \int_{X} h_{\widetilde{T}}(\mathbf{x})\,d\mathbf{x} \\
    &= \int_{\{\mathbf{y}=\lambda\mathbf{x}:\, \mathbf{x}\in X,\,\lambda\in[0,1]\}}
       \int_{\{\mathbf{b}=\lambda\mathbf{a}:\,\mathbf{x}^T\Omega_v\mathbf{a} = 1,\,\lambda\in[0,1]\}}
       d\mathbf{b}\,d\mathbf{y} \\
    &= \int_{\{\mathbf{y}=\lambda\mathbf{x}:\, \mathbf{x}\in X,\,\lambda\in[0,1]\}}
       \frac{1}{(d+1)!}\prod_{i=0}^d\frac{1}{\bigl(\mathbf{x}^T\Omega_v\bigr)_i}\,d\mathbf{y} \\
    &= \int_X\int_0^1
       \frac{1}{(d+1)!}\prod_{i=0}^d\frac{1}{\bigl(\mathbf{x}^T\Omega_v\bigr)_i}\,d\lambda\,d\mathbf{x} \\
    &= \int_X
       \frac{1}{(d+1)!}\prod_{i=0}^d\frac{1}{\bigl(\mathbf{x}^T\Omega_v\bigr)_i}\,d\mathbf{x}.
\end{align*}
This implies
\begin{equation}\tag{3}\label{eq:2}
    \displaystyle h_{\widetilde{T}}(\textbf{x}) = \frac{1}{(d+1)!}\prod_{i=0}^d\frac{1}{\big(\textbf{x}^T\Omega_v\big)_i}.
\end{equation}

To define the invariant measure for the non-homogeneous win-lose induction $T$, we need to project the pseudo-homogeneous measure $\mu_{\widetilde{T}}$ onto a 
simplex. Recall that at the previous step, passing from the natural extension 
$\widehat{T}$ to the linear win-lose induction $\widetilde{T}$, the invariant 
measure was defined via the associated cone construction. However, the simplex $\Delta^d$  is a set of codimension $1$ in $\mathbb{R}_+^{d+1}$, and the same cone construction would yield a measure with infinite density, since $\mu_{\widetilde{T}}$ is pseudo-homogeneous. This reflects the fact that passing from the linear to the projective setting requires a different approach: instead of using the cone construction, we define the invariant measure on $\Delta^d$ by realizing it as a limit of sets of positive measure $\mu_{\widetilde{T}}$:
\begin{center}
    $\displaystyle \mu_T(X) := \lim_{\varepsilon\to0^+}\frac{\mu_{\widetilde{T}}\big(\big\{s\textbf{x}|\;\textbf{x}\in X, s\in[1,1+\varepsilon]\big\}\cap\mathbb{R}_+^{d+1}\big)}{(d+1)\varepsilon}$\quad for $X \subset \Delta^d$.
\end{center}
Let us denote the set in the numerator of this fraction as $X_{\varepsilon}$. Then, for $\varepsilon$ small enough, $h_{\widetilde{T}}(\textbf{x})$ is bounded away from zero for $\textbf{x} \in X_{\varepsilon}$, and hence, applying Fubini's theorem, we are able to compute the invariant density of the non-homogeneous map:
\begin{align*}
    \mu_T(X) &= \lim_{\varepsilon\to0^+}\frac{1}{(d+1)\varepsilon}\int_{X_{\varepsilon}}h_{\widetilde{T}}(\textbf{x})\,d\textbf{x} \\
    &= \int_{X}\lim_{\varepsilon\to0^+}\frac{1}{(d+1)\varepsilon}\int_1^{1+\varepsilon}h_{\widetilde{T}}(s\textbf{x})(d+1)s^d\,ds\,d\textbf{x} \\
    &= \int_X h_{\widetilde{T}}(\textbf{x})\,d\textbf{x} = \int_X h_T(\textbf{x})\,d\textbf{x}.
\end{align*}
This implies
$$\displaystyle h_{T}(\textbf{x}) = h_{\widetilde{T}}(1-x_1-\ldots-x_d,x_1,\ldots,x_d).$$

\begin{remark}
    The spaces $\Delta^d$ and $\mathbb{P}\Lambda^{d+1}_+$ (see 
    Section~\ref{sec:mcf}), on which MCF algorithms usually act, are 
    fundamental domains for $\widetilde{T}$, i.e., can be identified with 
    the sets of orbits of non-ordered and ordered $\widetilde{T}$ respectively. The difference 
    between these two domains is purely a matter of the choice of 
    normalization: $\Delta^d$ corresponds to the section 
    $\{\sum_{i=0}^d x_i = 1\}$, while $\mathbb{P}\Lambda^{d+1}_+$ 
    corresponds to the section $\{\max_i x_i = 1\}$. For the 
    non-homogeneous ordered map $T$ acting on $\mathbb{P}\Lambda^{d+1}_+$, 
    we obtain, using the same method, the invariant density 
    $h_T(\textbf{x}) = h_{\widetilde{T}}(1, x_1,\ldots,x_d)$.
\end{remark}

\subsection{Computational aspects}\label{ssec:calculations}

We now discuss how to solve Equations~(\ref{eq:1}). Let us suppose that $G$ is connected (otherwise, consider one of its connected components instead of $G$). Consider the set of cycles $\gamma = [v_0,v_1,\ldots,v_{\text{len}(\gamma)},v_0]$ from some fixed vertex $v_0$ that generates the group of cycles of the undirected copy of $G$. Then~(\ref{eq:1}) is equivalent to the system of equations
\begin{equation}\tag{4}\label{eq:3}
    \displaystyle \small\big(M^T_{e_{\text{len}(\gamma)}}\big)^{\sgn(e_{\text{len}(\gamma)})}\ldots\big(M^T_{e_1}\big)^{\sgn(e_1)} \cdot \Omega_{v_0} = \Omega_{v_0}\cdot \big(M_{e'_{\text{len}(\gamma)}}\big)^{\sgn(e'_{\text{len}(\gamma)})}\ldots\big(M_{e'_1}\big)^{\sgn(e'_1)},
\end{equation}
for all such $\gamma$, where $e_i$ ($e'_i$) is the edge between $v_{i-1}$ and $v_{i}$ in $G$ (in $S$) and 
$$\sgn(e_i) = \begin{cases}
1 & \text{if }e_i: v_{i-1} \to v_{i}\text{ in }G, \\
-1 & \text{otherwise};
\end{cases}\quad\quad \sgn(e'_i) = \begin{cases}
1 & \text{if }e'_i: v_{i} \to v_{i-1}\text{ in }S, \\
-1 & \text{otherwise}.
\end{cases}$$

Using the \fcolorbox{green}{white}{\href{https://gitlab.com/fougeroc/MCF/-/blob/master/simplicial.py?ref_type=heads}{Sage code}}, we check whether the system~(\ref{eq:3}) admits a solution, and, if so, express all the matrices $\big(\Omega_{v}\big)_{v\in V}$ in terms of $\Omega_{v_0}$. This highlights the computational efficiency of our approach.  

\section{General definition of multidimensional continued fraction algorithms}\label{sec:mcf}

Following \cite[$\mathsection$~1]{schweigera} and \cite[$\mathsection$~1]{schweiger}, we recall the definition of MCF algorithms. It is important to note that the algorithms under consideration, all except Gauss, act on a space of vectors with unordered coordinates (see \cite[$\mathsection$~2]{berthesteinerthuswaldner} for the definition of the ordered case; from the point of view of dynamics, ordered and unordered versions are the same). The definition given below is directly compatible with the formalism of  simplicial systems introduced in Section~\ref{sec:definitions} (see also~\cite{fougeronpp}). As we will see in Section~\ref{sec:examples}, the step matrices of an MCF algorithm correspond 
precisely to the step matrices of the win-lose induction on the associated 
simplicial system, and it is this correspondence that allows one to construct 
an explicit measurable conjugacy. 

The \textit{linear version} of a $d$-dimensional MCF algorithm acts on the real vector space $\mathbb{R}_+^{d+1}$, and its \textit{projective version} acts on the simplex:
$$\Delta^d := \big\{\textbf{x} \in \mathbb{R}^{d+1}_+\big|\;\; \sum_{i=0}^dx_i=1\big\}.$$ 
We also use the following notation for the subsets of ordered vectors: 
$$\Lambda^{d+1}_{\sigma} = \big\{\textbf{x}\in \mathbb{R}^{d+1}_+\big|\;\; x_{\sigma_0} > x_{\sigma_1} > \ldots > x_{\sigma_d}\big\}\quad\text{for }\sigma \in S_{d+1},$$
where $S_{d+1}$ is the group of permutations of $d+1$ elements. The sets $\Lambda^{d+1}_{\sigma}$ form a partition of $\mathbb{R}^{d+1}_+$ up to a set of measure zero (see~\ref{fig:lambda_partition}).
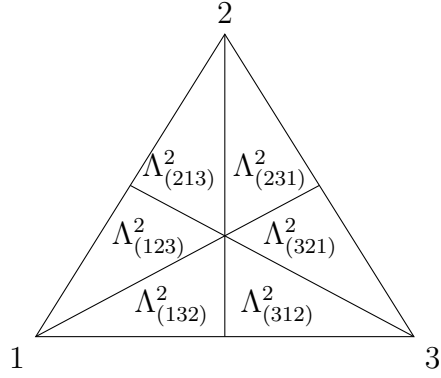
\begin{figure}
    \begin{minipage}[c]{1\textwidth}
    \center{\begin{tikzpicture}[node distance={25mm}]
        \draw(-2.5,0)--(0,4)--(2.5,0)--cycle;
        \draw(-2.5,0)--(1.25,2);
        \draw(2.5,0)--(-1.25,2);
        \draw(0,4)--(0,0);
        \node at (-0.7,0.4) {$\Lambda^2_{(132)}$};
        \node at (0.7,0.4) {$\Lambda^2_{(312)}$};
        \node at (-1,1.3) {$\Lambda^2_{(123)}$};
        \node at (1,1.3) {$\Lambda^2_{(321)}$};
        \node at (-0.6,2.2) {$\Lambda^2_{(213)}$};
        \node at (0.6,2.2){$\Lambda^2_{(231)}$};
        \node at (-2.5,0)[below left]{$1$};
        \node at (2.5,0)[below right]{$3$};
        \node at (0,4)[above]{$2$};
    \end{tikzpicture}}
    \end{minipage}
    \caption{Partition of $\mathbb{R}^{3}_+$ into sub-simplices $\Lambda^{3}_{\sigma}$ indexed by permutations $\sigma \in S_3$.}    \label{fig:lambda_partition}
\end{figure}
For the linear version, consider the matrix mapping
$$\displaystyle M: \Delta^d \to GL(d+1,\mathbb{Z}),$$
which depends on the algorithm. Then the mapping $M$ allows us to define the \textit{linear version}
$$\displaystyle F: \mathbb{R}^{d+1}_+\to\mathbb{R}^{d+1}_+,\quad\quad\textbf{x}\mapsto M\big(\rho(\textbf{x})\big)^{-1}\cdot\textbf{x}$$
of the corresponding $d$-dimensional MCF algorithm, where the mapping
$$\displaystyle \rho: \mathbb{R}^{d+1}_+ \to \Delta^d,\quad\quad (x_0,\ldots,x_d) \mapsto \|\textbf{x}\|^{-1}(x_0,x_1,\ldots,x_d),\quad\quad \|\textbf{x}\| := \sum_{i=0}^dx_i,$$
is a projection. The \textit{projective version} $f$ of the MCF algorithm  $$\displaystyle f: \Delta^{d}\to\Delta^d,\quad\quad\textbf{x}\mapsto\frac{F(\textbf{x})}{\|F(\textbf{x})\|}=\rho\big(M(\rho(\textbf{x}))^{-1}\cdot\textbf{x}\big)$$
is defined by the commutative diagram below.\\ \\
\begin{minipage}[c]{1\textwidth}
    \centering
    \begin{tikzpicture}[node distance={25mm}]
        \centering
        \node (1) {$\mathbb{R}^{d+1}_+$}; 
        \node (2) [right of=1] {$\mathbb{R}^{d+1}_+$};
        \node (3) [below of=1] {$\Delta^d$};
        \node (4) [below of=2] {$\Delta^d$};
        \draw[->] (1) -- node[midway, above] {$F$} (2);
        \draw[->] (3) -- node[midway, above] {$f$} (4);
        \draw[->] (1) -- node[midway, left] {$\rho$} (3);
        \draw[->] (2) -- node[midway, right] {$\rho$} (4);
    \end{tikzpicture}
\end{minipage}
\begin{remark}
    To simplify notation, we usually write $M(\textbf{x})$ instead of $M(\rho(\textbf{x}))$.
\end{remark}

\begin{remark}
    The matrix mapping $M$ above corresponds to the matrices we defined in Section~\ref{sec:definitions} for the win-lose induction. It connects a simplicial system with a corresponding MCF algorithm: multiplication of a vector by $M_e^{-1}$ (where $e$ is the edge we follow) is a step of the algorithm. The notion of win-lose induction formalizes this connection.
\end{remark}

\section{Examples}\label{sec:examples}

\subsection{Farey algorithm}\label{ssec:farey}
Consider the piecewise linear map $F$ (the additive Euclid algorithm) on the $2$-dimensional positive cone:
\begin{center}
    $\displaystyle F: \mathbb{R}_+^2 \to \mathbb{R}_+^2$, \\
    $ $ \\
    $\displaystyle F(x_0,x_1) = \begin{cases}
    (x_0,x_1-x_0) & \text{if } x_0<x_1, \\ 
    (x_0-x_1,x_1)&\text{otherwise}.
    \end{cases}$
\end{center}
The associated projective map $f$ is defined as:
\begin{center}
    $\displaystyle f: \Delta^1 \to \Delta^1$, \\
    $ $ \\
    $\displaystyle f(x,1-x) = \frac{F(x,1-x)}{\|F(x,1-x)\|} = \begin{cases}
    (\frac{x}{1-x},\frac{1-2x}{1-x}) & \text{if }x<\frac{1}{2}, \\ 
    (\frac{2x-1}{x},\frac{1-x}{x})&\text{otherwise}.
    \end{cases}$
\end{center}
The map $f$ is defined by the action on the first coordinate and is called the (non-ordered) \textit{Farey algorithm} (see \cite{Rauzy1981}, where it is first considered as a continued fraction algorithm). \\
Now let us construct the simplicial system $G$ for the Farey algorithm (see Fig.~\ref{fig:ss_farey}a) with the associated linear win-lose induction
$$\widetilde{T}: \{v\}\times\mathbb{R}_+^{2} \rightarrow \{v\}\times\mathbb{R}_+^{2}$$ 
$$\displaystyle \widetilde{T}(v, (x_0,x_1)=(v, M^{-1}_e\cdot\textbf{x}
), \;\; M_e = \begin{cases}
\big[\begin{smallmatrix} 1 & 1 \\ 0 & 1 \end{smallmatrix}\big] & \text{if }x_0 > x_1,\\ 
\big[\begin{smallmatrix} 1 & 0 \\ 1 & 1 \end{smallmatrix}\big] & \text{if }x_0 < x_1,
\end{cases}$$
and the projective win-lose induction
$$T: \{v\}\times\Updelta^1 \rightarrow \{v\}\times\Updelta^1,$$
$$\displaystyle T(v, (x_0, 1-x_0))=(v, T_e(\textbf{x})), \quad\quad T_e: \begin{cases}
\Updelta_e^{1} \rightarrow \Updelta^1 \\ 
\textbf{x} \mapsto \frac{M^{-1}_e \cdot \textbf{x}}{\|M^{-1}_e \cdot \textbf{x}\|_1}
\end{cases} \text{for} \;\; \textbf{x}\in\Updelta_e^{1}, \; e: v \rightarrow v,$$
where an edge partition of the simplex $\Updelta^1$ consists of two elements corresponding to edges with labels $0$ and $1$, i.e.
$$\{\Updelta^1_e\}_{e\in E} = \{\Updelta^1_0, \Updelta^1_1\} = \Big{\{}\{\textbf{x} \in \Updelta^1|\;x_0 < \frac{1}{2}\}, \{\textbf{x} \in \Updelta^1|\;x_0 > \frac{1}{2}\}\Big{\}}.$$
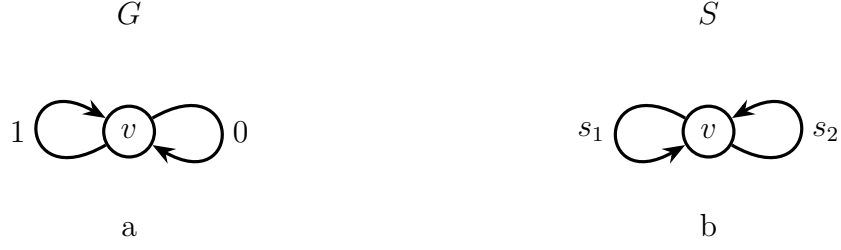
\begin{figure}[ht]
    \begin{minipage}[ht]{0.49\linewidth}
        \center{$G$ \\ 
        $ $ \\ 
        \begin{tikzpicture}[node distance={34mm}, > = Stealth, very thick, main/.style = {draw, circle}] 
            \node[main] (1) {$v$}; 
            \draw[->] (1) to [out=210, in=150, looseness=10] node[midway, left] {$1$} (1);
            \draw[->] (1) to [out=30, in=330, looseness=10] node[midway, right] {$0$} (1);
        \end{tikzpicture} \\ a}
    \end{minipage}
    \hfill
    \begin{minipage}[ht]{0.49\linewidth}
        \center{$S$ \\ 
        $ $ \\ 
        \begin{tikzpicture}[node distance={34mm}, > = Stealth, very thick, main/.style = {draw, circle}] 
            \node[main] (1) {$v$}; 
            \draw[->] (1) to [out=150, in=210, looseness=10] node[midway, left] {$s_1$} (1);
            \draw[->] (1) to [out=330, in=30, looseness=10] node[midway, right] {$s_2$} (1);
        \end{tikzpicture}\\ b}
    \end{minipage}
    \caption{Simplicial (a) and dual simplicial (b) systems for the Farey algorithm.}
    \label{fig:ss_farey}
\end{figure} \\
Then the non-homogeneous win-lose induction $T$ is conjugated to the Farey algorithm by the trivial bijection: \\
\begin{minipage}[h]{0.49\linewidth}
    \centering
    $\varphi_f: \{v\}\times\Delta^1 \to \Delta^1,$ \\
    $ $ \\
    $\varphi_f(v,x,1-x)=(x,1-x)$.
\end{minipage}
    \hfill
\begin{minipage}[ht]{0.49\linewidth}
    \begin{tikzpicture}[node distance={25mm}]
        \centering
        \node (1) {$\{v\}\times\Delta^1$}; 
        \node (2) [right of=1] {$\{v\}\times\Delta^1$};
        \node (3) [below of=1] {$\Delta^1$};
        \node (4) [below of=2] {$\Delta^1$};
        \draw[->] (1) -- node[midway, above] {$T$} (2);
        \draw[->] (3) -- node[midway, above] {$f$} (4);
        \draw[->] (1) -- node[midway, left] {$\varphi_f$} (3);
        \draw[->] (2) -- node[midway, right] {$\varphi_f$} (4);
    \end{tikzpicture}
\end{minipage} 
\medskip

Following the approach described in Section~\ref{sec:measure}, we construct the dual simplicial system (see Fig.~\ref{fig:ss_farey}b) and the realization $\widehat{T}$ of the natural extension for $\widetilde{T}$. A solution to the system of linear equations~(\ref{eq:1}) is $\Omega_{v}=\big[\begin{smallmatrix} 1 & 0 \\ 0 & 1 \end{smallmatrix}\big]$ for the labeling $(s_1,s_2) = (0,1)$, and the invariant densities for the linear and projective inductions are\\ \\
\begin{minipage}[ht]{0.49\linewidth}
    \centering
    $h_{\widetilde{T}}: \{v\}\times\mathbb{R}_+^2 \to \mathbb{R}_+$, \\
    $ $ \\
    $\displaystyle h_{\widetilde{T}}(v,\textbf{x}) = \frac{1}{2x_0x_1}$;
\end{minipage}
\hfill
\begin{minipage}[ht]{0.49\linewidth}
    \centering
    $h_{T}: \{v\}\times\Delta^1 \to \mathbb{R}_+$, \\
    $ $ \\
    $\displaystyle h_{T}(v,x,1-x) = \frac{1}{2x(1-x)}$.
\end{minipage} \\ \\ \\
Then the invariant density for the (non-ordered) Farey algorithm is 
$$\displaystyle h_{f}(x,1-x) = (\varphi_f)_*h_T(x,1-x) = h_{T}(v,x,1-x) = \frac{1}{2x(1-x)}.$$ 

\begin{remark}
    In Farey's case, the system~(\ref{eq:1}) admits another symmetric solution, 
    namely $\Omega_{v_0}=\big[\begin{smallmatrix} 0 & 1 \\ 1 & 0 \end{smallmatrix}\big]$ for the labeling $(s_1,s_2) = (1,0)$, which yields the same invariant density.
\end{remark}
 
\subsection{Gauss algorithm}\label{ssec:gauss}
Now the map $F$ is an ordered (see \cite[$\mathsection$~2]{berthesteinerthuswaldner} for the precise definition of the ordered case) Euclid algorithm on the cone $\Lambda^1_{\text{id}} := \big{\{}x\in\mathbb{R}_+^2|\;x_0 > x_1\big{\}}$:
\begin{center}
    $\displaystyle F: \Lambda^1_{\text{id}} \to \Lambda^1_{\text{id}}$, \\
    $ $ \\
    $ F(x_0,x_1) = (x_1, x_0-\big[\frac{x_0}{x_1}\big]x_1)$.
\end{center}
\medskip

The associated projective map $f$ is defined as:
\\ \\
\begin{minipage}[ht]{0.49\linewidth}
    \centering
    $f: \mathbb{P}\Lambda^1_{\text{id}} \to \mathbb{P}\Lambda^1_{\text{id}},$ \\
    \medskip
    $f(1:x)=\big(x:1-\big[\frac{1}{x}\big]x\big)=\big(1:\frac{1}{x}-\big[\frac{1}{x}\big]\big)$. \\
    $ $ \\
    $f$ is associated with $F$ by the projection \\
    $\pi: \Lambda^1_{\text{id}} \to \mathbb{P}\Lambda^1_{\text{id}}$, \\
    \medskip
    $\pi(x_0,x_1) = \big(1:\frac{x_1}{x_0}\big)$.
\end{minipage}
\hfill
\begin{minipage}[ht]{0.49\linewidth}
    \center{\begin{tikzpicture}[node distance={25mm}]
        \node (1) {$\Lambda^1_{\text{id}}$}; 
        \node (2) [right of=1] {$\Lambda^1_{\text{id}}$};
        \node (3) [below of=1] {$\mathbb{P}\Lambda^1_{\text{id}}$};
        \node (4) [below of=2] {$\mathbb{P}\Lambda^1_{\text{id}}$};
        \draw[->] (1) -- node[midway, above] {$F$} (2);
        \draw[->] (3) -- node[midway, above] {$f$} (4);
        \draw[->] (1) -- node[midway, left] {$\pi$} (3);
        \draw[->] (2) -- node[midway, right] {$\pi$} (4);
    \end{tikzpicture}}
\end{minipage} \\ \\
\medskip
The map $f$ is defined by the action on the second coordinate and is called the (ordered) \textit{Gauss algorithm} (see, e.g., \cite{Khinchin1935}). \\
\begin{figure}[ht]
    \begin{minipage}[ht]{0.49\linewidth}
        \center{$G$ \\ 
        $ $ \\ 
        \begin{tikzpicture}[node distance={34mm}, > = Stealth, very thick, main/.style = {draw, circle}] 
            \node[main] (1) {$v_0$}; 
            \node[main] (2) [right of=1] {$v_1$};
            \draw[->] (1) to [out=210, in=150, looseness=10] node[midway, left] {$0$} (1);
            \draw[->] (1) to [out=45, in=135, looseness=1] node[midway, above] {$1$} (2);
            \draw[->] (2) to [out=225, in=315, looseness=1] node[midway, below] {$0$} (1);
            \draw[->] (2) to [out=30, in=330, looseness=10] node[midway, right] {$1$} (2);
        \end{tikzpicture} \\ a}
    \end{minipage}
    \hfill
    \begin{minipage}[ht]{0.49\linewidth}
        \center{$S$ \\ 
        $ $ \\ 
        \begin{tikzpicture}[node distance={34mm}, > = Stealth, very thick, main/.style = {draw, circle}] 
            \node[main] (1) {$v_0$}; 
            \node[main] (2) [right of=1] {$v_1$};
            \draw[->] (1) to [out=150, in=210, looseness=10] node[midway, left] {$s_1$} (1);
            \draw[->] (1) to [out=315, in=225, looseness=1] node[midway, below] {$s_4$} (2);
            \draw[->] (2) to [out=135, in=45, looseness=1] node[midway, above] {$s_2$} (1);
            \draw[->] (2) to [out=330, in=30, looseness=10] node[midway, right] {$s_3$} (2);
        \end{tikzpicture}\\ b}
    \end{minipage}
    \caption{Simplicial (a) and dual simplicial (b) systems for the Gauss algorithm.}
    \label{fig:ss_gauss}
\end{figure}
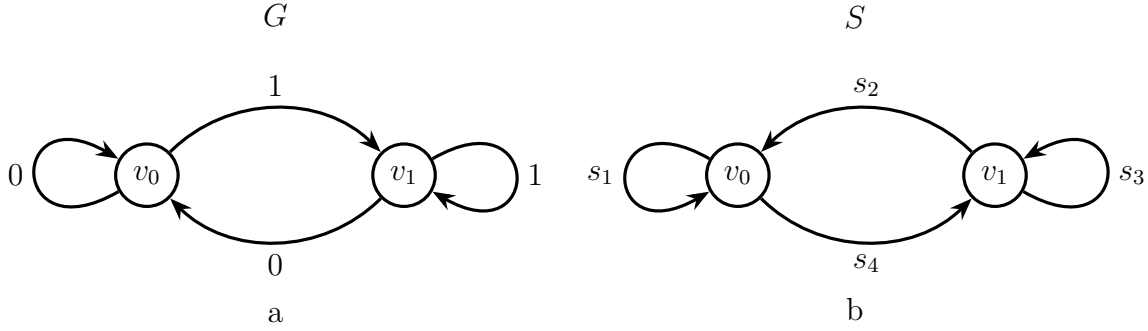 \\
Consider the simplicial system $G$ for the Gauss algorithm (see Fig.~\ref{fig:ss_gauss}). We need to define an acceleration of the win-lose induction on $G$. Recall that given 
a measurable map 
$$T: X \to X$$
and a measurable set $A \subset X$ of positive measure, the \textit{first return map}
$$T_A: A \to A$$
is defined as
\begin{align*}
    T_A(\textbf{x}) = T^{n(\textbf{x})}(\textbf{x}), \quad n(\textbf{x}) = \min\{n \geq 1 : T^n(\textbf{x}) \in A\}.
\end{align*}
Let
\begin{align*}
    D := \{v_0\}\times \Lambda_{\text{id}}^1 \sqcup \{v_1\}\times \Lambda^1_{(01)} = D_0\sqcup D_1,
    \quad\text{where }\Lambda^1_{(01)}= \big\{x\in\mathbb{R}_+^2\mid x_1>x_0\big\}.
\end{align*}
Then the first return map $\widetilde{T}_D$ of the linear win-lose induction $\widetilde{T}$ on the set $D$ is semiconjugated to $F$ by the projection $\pi_G$(not injective, unlike a conjugacy in Farey's case):\\ \\
\begin{minipage}[h]{0.49\linewidth}
    \centering
    $\pi_G: D \to \Lambda_{\text{id}}^1$, \\
    \medskip
    $\pi_G\big(\{v_0\}\times(x_0,x_1)\big) = (x_0,x_1)$,\\
    $\pi_G\big(\{v_1\}\times(x_0,x_1)\big) = (x_1,x_0)$.
\end{minipage}
\hfill
\begin{minipage}[h]{0.49\linewidth}
    \center{\begin{tikzpicture}[node distance={25mm}]
        \node (1) {$D$}; 
        \node (2) [right of=1] {$D$};
        \node (3) [below of=1] {$\Lambda_{\text{id}}^1$};
        \node (4) [below of=2] {$\Lambda_{\text{id}}^1$};
        \draw[->] (1) -- node[midway, above] {$\widetilde{T}_D$} (2);
        \draw[->] (3) -- node[midway, above] {$F$} (4);
        \draw[->] (1) -- node[midway, left] {$\pi_G$} (3);
        \draw[->] (2) -- node[midway, right] {$\pi_G$} (4);
    \end{tikzpicture}}
\end{minipage} \\ \\
We construct again a dual simplicial system (see Fig.~\ref{fig:ss_gauss}b) and a realization $\widehat{T}$ of the natural extension for $\widetilde{T}$. A solution to the system of linear equations~(\ref{eq:1}) is $\Omega_{v_0}=\big[\begin{smallmatrix} 1 & 1 \\ 1 & 0 \end{smallmatrix}\big], \Omega_{v_1}=\big[\begin{smallmatrix} 0 & 1\\ 1 & 1 \end{smallmatrix}\big]$ for the labeling $(s_1,s_2,s_3,s_4) = (0,0,1,1)$, and the invariant densities for the linear and projective inductions are:
\begin{center}
    \begin{minipage}[t]{0.49\linewidth}
        \centering
        $h_{\widetilde{T}_D}: D \to \mathbb{R}_+$, \\
        $ $ \\
        $\displaystyle h_{\widetilde{T}_D}(v_0,\textbf{x}) = \frac{1}{2(x_0+x_1)x_0}\cdot\delta_{\{\textbf{x}\in\Lambda_{\text{id}}^1\}}$, \\
        $\displaystyle h_{\widetilde{T}_D}(v_1,\textbf{x}) = \frac{1}{2x_1(x_0+x_1)}\cdot\delta_{\{\textbf{x}\in\Lambda_{(12)}^1\}}$;
    \end{minipage}
    \hfill
    \begin{minipage}[t]{0.49\linewidth}
        \centering
        $h_{F}: \Lambda_{\text{id}}^1 \to \mathbb{R}_+$, \\
        $ $ \\
        $ $ \\
        $\displaystyle h_{F}(\textbf{x}) = (\pi_G)_*h_{\widetilde{T}_D}(\textbf{x})=\frac{1}{(x_0+x_1)x_0}$. \\
        $\displaystyle\frac{}{} $
    \end{minipage}
\end{center}
Then the invariant density for the Gauss algorithm is $\displaystyle h_f(1:x) = \frac{1}{2(1+x)}$, thereby recovering the classical Gauss measure up to a normalization constant.
\begin{remark}
    In Gauss' case, the system~(\ref{eq:1}) admits another solution symmetric to the one 
    above, namely $\Omega_{v_0}=\big[\begin{smallmatrix} 1 & 1 \\ 0 & 1 \end{smallmatrix}\big], \Omega_{v_1}=\big[\begin{smallmatrix} 1 & 0\\ 1 & 1 \end{smallmatrix}\big]$ for the labeling $(s_1,s_2,s_3,s_4) = (1,1,0,0)$, which yields the same invariant density. 
\end{remark}

\subsection{Selmer algorithm}

We now consider a first multidimensional example where we apply the graph method to recover its invariant density, namely the Selmer algorithm. Its invariant density was first given in \cite[Theorem~22]{schweiger}.
\subsubsection{Definition}
The Selmer algorithm is defined by subtracting the smallest coordinate of a vector from the largest one, see \cite{selmer}. The \textit{centered} subset 
$$\Delta_{\text{center}}^d = \big\{x \in \Delta^{d}\big|\;\; x_{\sigma_0}\leq x_{\sigma_{d-1}}+x_{\sigma_d}\text{ for }\sigma\in S_d\text{ s.t. }\textbf{x}\in\Lambda^{d+1}_{\sigma}\big\}\subset\Delta^d$$ 
is invariant under the Selmer algorithm, and for almost every $\textbf{x}\in\Delta^d$ after finite time $N\geq0$, $f^N(\textbf{x})\in\Delta^d_{\text{center}}$ (see \cite[Theorem~22]{schweiger}).

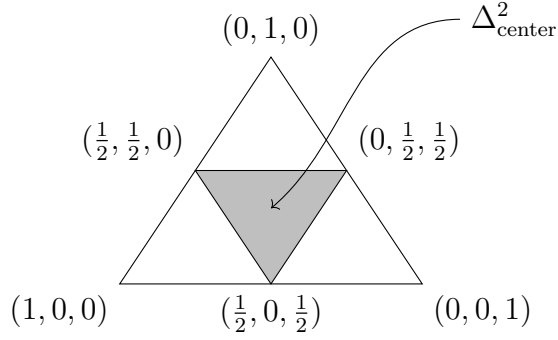
\begin{figure}
    \begin{minipage}[c]{1\textwidth}
    \center{\begin{tikzpicture}[node distance={25mm}]
    \fill[lightgray] (0,0) -- (-1,1.5) -- (1,1.5) -- cycle;
    \draw(-2,0)--(0,3)--(2,0)--cycle;
    \draw(0,0) -- (-1,1.5) -- (1,1.5)--cycle;
    \node at (-2,0)[below left]{$(1,0,0)$};
    \node at (2,0)[below right]{$(0,0,1)$};
    \node at (0,3)[above]{$(0,1,0)$};
    \node at (0,0)[below]{$(\frac{1}{2},0,\frac{1}{2})$};
    \node at (2.5,3.5)[right]{$\Delta_{\text{center}}^2$};
    \node at (-1,1.5)[above left]{$(\frac{1}{2},\frac{1}{2},0)$};
    \node at (1,1.5)[above right]{$(0,\frac{1}{2},\frac{1}{2})$};
    \draw[->] (2.5,3.5) to [out=180, in=40, looseness=1] (0,1);
    \end{tikzpicture}}
    \end{minipage}
    \caption{Invariant set for the projective Selmer algorithm in dimension $d=2$.}
    \label{fig:invset_selmer}
\end{figure}

For the Selmer algorithm, the matrix map is 
$$M=\Id + E_{\sigma_0,\sigma_d}\quad\quad\text{for }x\in\big(\mathbb{R}_+\cdot\Delta^d_{\text{center}}\big)\cap\Lambda^{d+1}_{\sigma}.$$
Its linear version is: 
\begin{center}
    $\displaystyle F: \mathbb{R}_+\cdot\Delta^{d}_{\text{center}}\to\mathbb{R}_+\cdot\Delta^{d}_{\text{center}}$, \\
    $ $ \\
    $F(\textbf{x}) = (x_0,\ldots,x_{\sigma_0}-x_{\sigma_d},\ldots,x_d)$ for $\textbf{x}\in\big(\mathbb{R}_+\cdot\Delta^d_{\text{center}}\big)\cap\Lambda^{d+1}_{\sigma}$,
\end{center}
and its projective version is:
\begin{center}
    $\displaystyle f: \Delta^d_{\text{center}}\to\Delta^d_{\text{center}}$, \\
    $ $ \\
    $f(\textbf{x}) = \big(1-x_{\sigma_d}\big)^{-1}(x_0,\ldots,x_{\sigma_0}-x_{\sigma_d},\ldots,x_d)$ for $\textbf{x}\in\Delta^d_{\text{center}}\cap\Lambda_{\sigma}^{d+1}$.
\end{center}
\subsubsection{Constructing a simplicial system}
Following \cite[$\mathsection$~5.2.2]{fougeronpp}, we construct a simplicial system for the \textit{Selmer algorithm}. Consider the directed labeled graph $G$ with $\frac{(d+1)!}{2}$ vertices, labeled as $(v_{\sigma})_{\sigma\in S_{d+1}}$; each vertex $v_{\sigma}$ has two outgoing edges: one with label $\sigma_0$ to vertex $v_{\sigma\cdot(0\ldots d-1)}$ and one with label $\sigma_d$ to vertex $v_{\sigma\cdot(0\ldots d)}$ (see Fig.~\ref{fig:ss_selmer}a). \\
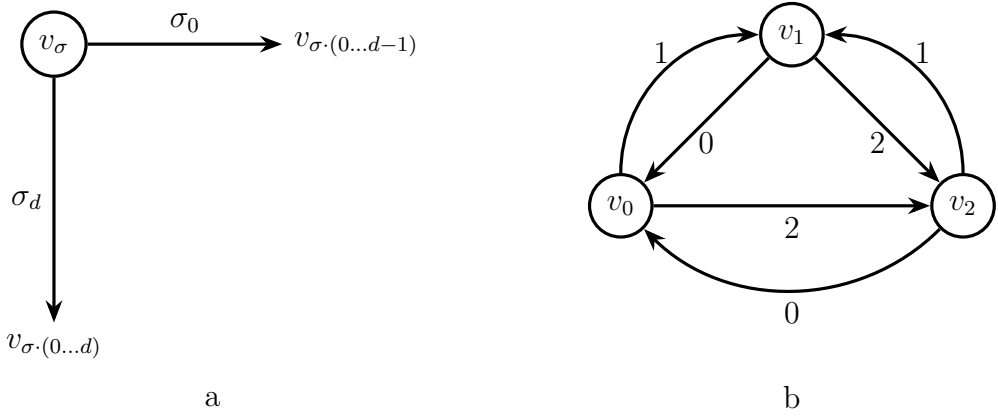
\begin{figure}[ht]
    \begin{minipage}[ht]{0.49\linewidth}
        \center{\begin{tikzpicture}[node distance={40mm}, > = Stealth, very thick, main/.style = {draw, circle}] 
        \node[main] (1) {$v_{\sigma}$}; 
        \node (2) [right of=1] {$v_{\sigma\cdot(0\ldots d-1)}$};
        \node (3) [below of=1] {$v_{\sigma\cdot(0\ldots d)}$};
        \draw[->] (1) -- node[midway, above] {$\sigma_0$} (2);
        \draw[->] (1) -- node[midway, left] {$\sigma_{d}$} (3);
        \end{tikzpicture}  \\ a}
    \end{minipage}
    \hfill
    \begin{minipage}[ht]{0.49\linewidth}
        \center{\begin{tikzpicture}[node distance={32mm}, > = Stealth, very thick, main/.style = {draw, circle}] 
        \node[main] (1) {$v_1$}; 
        \node[main] (2) [below left of=1] {$v_0$};
        \node[main] (3) [below right of=1] {$v_2$};
        \draw[->] (2) to [out=90, in=180, looseness=1] node[midway, above] {$1$} (1);
        \draw[->] (2) -- node[midway, below] {$2$} (3);
        \draw[->] (3) to [out=225, in=315, looseness=1] node[midway, below] {$0$} (2);
        \draw[->] (3) to [out=90, in=0, looseness=1] node[midway, above] {$1$} (1);
        \draw[->] (1) -- node[midway, below] {$0$} (2);
        \draw[->] (1) -- node[midway, below] {$2$} (3);
        \end{tikzpicture}  \\ $ $ \\ b}
    \end{minipage}
    \caption{Simplicial system for the Selmer algorithm in a) arbitrary dimension $d$ and b) $d=2$.}
    \label{fig:ss_selmer}
\end{figure} \\
Consider the dimension $d=2$. The linear win-lose induction $\widetilde{T}$ on $G$ is conjugated to the linear version of the Selmer algorithm $F$ by the bijection $\varphi^2_S$:\\ \\
\begin{minipage}[ht]{0.70\linewidth}
\centering
$\varphi^2_S: V\times\mathbb{R}_+^3\to \mathbb{R}_+\cdot\Delta^2_{\text{center}},$\\
\medskip
$\varphi^2_S(v,\textbf{x})=P_v\cdot\textbf{x}$, where \\
\bigskip
$P_{v_0} = \begin{bmatrix} 1 & 1 & 1 \\ 1 & 0 & 1 \\ 1 & 1 & 0 \end{bmatrix}, P_{v_1} = \begin{bmatrix} 0 & 1 & 1 \\ 1 & 1 & 1 \\ 1 & 1 & 0 \end{bmatrix}, P_{v_2} = \begin{bmatrix} 0 & 1 & 1 \\ 1 & 0 & 1 \\ 1 & 1 & 1 \end{bmatrix}$,
\end{minipage}
\hfill
\begin{minipage}[ht]{0.28\linewidth}
\center{\begin{tikzpicture}[node distance={25mm}]
\node (1) {$V\times\mathbb{R}_+^3$}; 
\node (2) [right of=1] {$V\times\mathbb{R}_+^3$};
\node (3) [below of=1] {$\mathbb{R}_+\cdot\Delta^2_{\text{center}}$};
\node (4) [below of=2] {$\mathbb{R}_+\cdot\Delta^2_{\text{center}}$};
\draw[->] (1) -- node[midway, above] {$\widetilde{T}$} (2);
\draw[->] (3) -- node[midway, above] {$F$} (4);
\draw[->] (1) -- node[midway, left] {$\varphi^2_S$} (3);
\draw[->] (2) -- node[midway, right] {$\varphi^2_S$} (4);
\end{tikzpicture}}
\end{minipage} \\ \\ \\
i.e., the bijection $\varphi^2_S$ sends the pair $(v_{\sigma_1},\textbf{x})$ to the vector $$\displaystyle P_{v_{\sigma_1}}\cdot\textbf{x}\in \big\{\textbf{y}\in\mathbb{R}_+^3|\; y_{\sigma_0}+y_{\sigma_2} > y_{\sigma_1}>y_{\sigma_0},y_{\sigma_2}\big\}.$$

\begin{figure}
    \centering
    \begin{tikzpicture}[yscale=0.7,xscale=0.7]
\node at (-4.5,1.5){$\{v_0\}$};
\node at (-3,1.5){$\times$};
\fill[Apricot] (-2,0)--(0,3)--(1,1.5)--cycle;
\draw(-2,0)--(0,3)--(2,0)--cycle;
\draw(-2,0)--(1,1.5);
\node at (0,1.7){$1<2$};
\node[below] at (-0.15,0){$x_2=0$};
\node[above left, rotate=56.30993247402] at (-0.4,2.4){$x_1=0$};
\node[above right, rotate=303.690067526] at (0.4,2.4){$x_0=0$};
\node at (-2.5,3.5)[left]{$\mathbb{R}_+\cdot\Delta^{2}_{1(2)}$};
\draw[->] (-2.5,3.5) to [out=0, in=80, looseness=1.5] (0,2.5);

\node[above] at (4.5,1.5){$\widetilde{T}$};
\draw[->] (3,1.5) to (6,1.5);

\node at (7,1.5){$\{v_1\}$};
\node at (8.5,1.5){$\times$};
\fill[lightgray] (9.5,0)--(11.5,3)--(13.5,0)--cycle;
\draw(9.5,0)--(11.5,3)--(13.5,0)--cycle;
\node at (11.5,1){$\mathbb{R}_+\cdot\Delta^{2}$};
\node[below] at (11.35,0){$x_2=0$};
\node[above left, rotate=56.30993247402] at (11.1,2.4){$x_1=0$};
\node[above right, rotate=303.690067526] at (11.9,2.4){$x_0=0$};

\node[left] at (0,-2.25){$\varphi^2_S$};
\draw[->] (0,-1) to (0,-3.5);

\fill[Apricot] (-0.5,-6.25)--(-1,-5.5)--(0,-6)--cycle;
\draw (0,-7)--(-1,-5.5)--(1,-5.5)--cycle;
\draw(-2,-7)--(0,-4)--(2,-7)--cycle;
\draw (0,-7)--(-1,-5.5)--(0,-6)--cycle;
\draw (-0.5,-6.25)--(1,-5.5);
\node at (-2,-3.5)[left]{$\Lambda_{(021)}^3\cap\mathbb{R}_+\cdot\Delta^{2}_{\text{center}}$};
\draw[->] (-2,-3.5) to [out=0, in=80, looseness=1.5] (-0.5,-6);

\node[above] at (5.75,-5.5){$F$};
\draw[->] (4.25,-5.5) to (7.25,-5.5);

\fill[lightgray] (10.5,-5.5)--(12.5,-5.5)--(11.5,-6)--cycle;
\draw (11.5,-7)--(10.5,-5.5)--(12.5,-5.5)--cycle;
\draw(9.5,-7)--(11.5,-4)--(13.5,-7)--cycle;
\draw (10.5,-5.5)--(12.5,-5.5)--(11.5,-6)--cycle;
\draw (11.5,-7)--(11.5,-6);
\node at (13,-3.5)[right]{$\mathbb{R}_+\cdot P_{v_1}\Delta^{2}_{\text{center}}$};
\draw[->] (13,-3.5) to [out=180, in=90, looseness=1.5] (11.5,-5.75);

\node[right] at (11.5,-2.25){$\varphi^2_S$};
\draw[->] (11.5,-1) to (11.5,-3.5);
\end{tikzpicture}
    \caption{The linear Selmer algorithm $F$ in dimension $d=2$ and the corresponding linear win-lose induction $\widetilde{T}$.}
    \label{fig:ss_selmer+alg}
\end{figure}
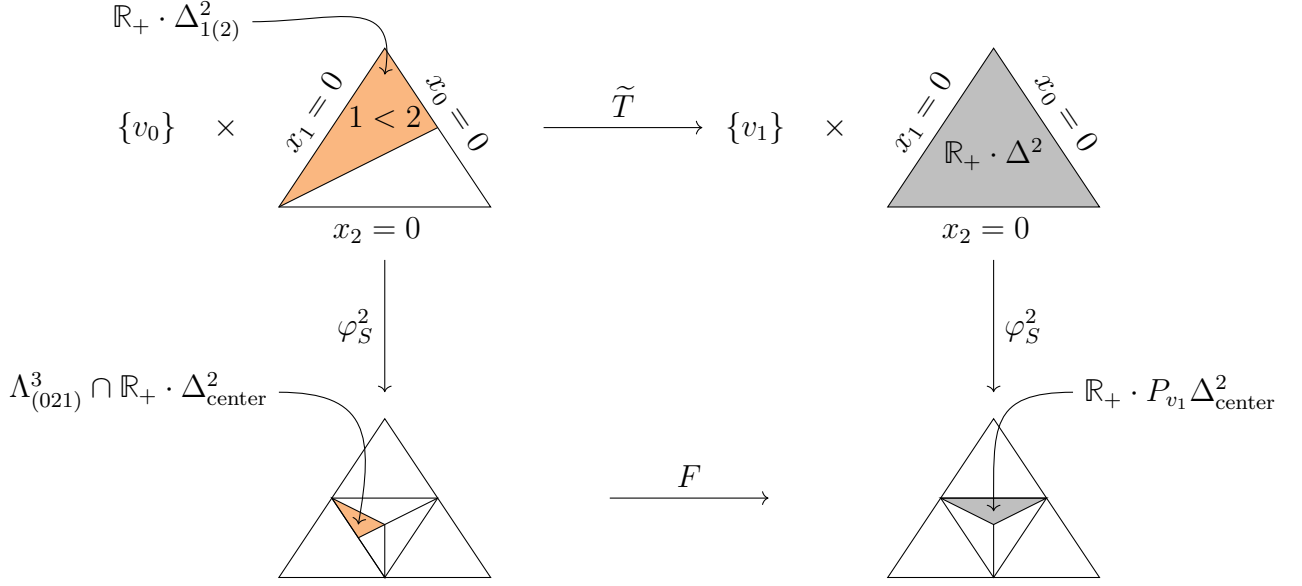

In arbitrary dimension, the matrices of $\varphi^d_S$ are:
\begin{center}
    $\displaystyle P_{v_{\sigma}} = \begin{bmatrix} P_{v_{\sigma}}^0 \\ \vdots \\ P_{v_{\sigma}}^d \end{bmatrix},\quad \big(P_{v_{\sigma}}^j\big)_k = \begin{cases}
    0 & \text{if }j=k=\sigma_{d-1}\text{ or }\sigma_d, \\
    1 & \text{if }j=\sigma_{m}, k=\sigma_n,\;\; 0<m<n\leq d, \\
    2 & \text{otherwise},
\end{cases}$
\end{center}
i.e., the bijection $\varphi^d_{S}$ sends the pair $(v_{\sigma},\textbf{x})$ to the vector in $\displaystyle \Lambda_{\sigma\cdot(0\ldots d-1)}^{d+1}\sqcup\Lambda_{\sigma\cdot(0\ldots d)}^{d+1}.$ 

\subsubsection{Invariant measure for $d=2$}

Following the approach described in Section~\ref{sec:measure}, we construct a dual simplicial system and a realization of the natural extension for $\widetilde{T}$. The system of linear equations~(\ref{eq:1}) is: 

\begin{figure}[ht]
    \center{ 
    \begin{tikzpicture}[node distance={26mm}, > = Stealth, very thick, main/.style = {draw, circle}] 
        \node[main] (1) {$v_1$}; 
        \node[main] (2) [below left of=1] {$v_0$};
        \node[main] (3) [below right of=1] {$v_2$};
        \draw[->] (1) to [out=180, in=90, looseness=1] node[midway, above] {$s_3$} (2);
        \draw[->] (3) -- node[midway, below] {$s_5$} (2);
        \draw[->] (2) to [out=315, in=225, looseness=1] node[midway, below] {$s_2$} (3);
        \draw[->] (1) to [out=0, in=90, looseness=1] node[midway, above] {$s_4$} (3);
        \draw[->] (2) -- node[midway, below] {$s_1$} (1);
        \draw[->] (3) -- node[midway, below] {$s_6$} (1);
    \end{tikzpicture}}
    \caption{Dual simplicial system for the Selmer algorithm in dimension $d=2$.}
    \label{fig:ssdual2_selmer}
\end{figure}
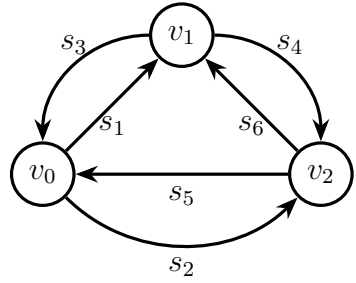 

\medskip
\begin{center}
    $\begin{cases}
        M_{1(2)}^T \cdot \Omega_{v_0} = \Omega_{v_1} \cdot M_{s_3(s_4)} \\
        M_{2(1)}^T \cdot \Omega_{v_0} = \Omega_{v_2} \cdot M_{s_5(s_6)} \\
        M_{0(2)}^T \cdot \Omega_{v_1} = \Omega_{v_0} \cdot M_{s_1(s_2)} \\
        M_{2(0)}^T \cdot \Omega_{v_1} = \Omega_{v_2} \cdot M_{s_6(s_5)} \\
        M_{0(1)}^T \cdot \Omega_{v_2} = \Omega_{v_0} \cdot M_{s_2(s_1)} \\
        M_{1(0)}^T \cdot \Omega_{v_2} = \Omega_{v_1} \cdot M_{s_4(s_3)} 
    \end{cases},$
\end{center}
and its solution is 
$$\Omega_{v_0}=\Big[\begin{smallmatrix} 1 & 1 & 1 \\ 1 & 0 & 1 \\ 1 & 1 & 0 \end{smallmatrix}\Big], \quad \Omega_{v_1}=\Big[\begin{smallmatrix} 0 & 1 & 1 \\ 1 & 1 & 1 \\ 1 & 1 & 0 \end{smallmatrix}\Big],\quad \Omega_{v_2}=\Big[\begin{smallmatrix} 0 & 1 & 1 \\ 1 & 0 & 1 \\ 1 & 1 & 1 \end{smallmatrix}\Big]$$ 
for the labeling $(s_1,s_2,s_3,s_4,s_5,s_6) = (1,2,0,2,0,1)$ (this solution is unique up to the permutation of columns of the matrices $\Omega_{v_i}$). The invariant densities for the linear and projective inductions are: 
\begin{center}
    \begin{minipage}[ht]{0.59\linewidth}
        \centering
        $h_{\widetilde{T}}: V\times\mathbb{R}_+^3 \to \mathbb{R}_+$, \\
        $ $ \\
        $\displaystyle h_{\widetilde{T}}(v_0,\textbf{x}) = \frac{1}{6(x_0+x_1+x_2)(x_0+x_1)(x_0+x_2)}$, \\
        $ $ \\
        $\displaystyle h_{\widetilde{T}}(v_1,\textbf{x}) = \frac{1}{6(x_1+x_2)(x_0+x_1+x_2)(x_0+x_1)}$, \\
        $ $ \\
        $\displaystyle h_{\widetilde{T}}(v_2,\textbf{x}) = \frac{1}{6(x_1+x_2)(x_0+x_2)(x_0+x_1+x_2)}$;
    \end{minipage}
    \hfill
    \begin{minipage}[ht]{0.39\linewidth}
        \centering
        $h_{T}: V\times\Delta^2 \to \mathbb{R}_+$, \\
        $ $ \\
        $\displaystyle h_{T}(v_0,\textbf{x}) = \frac{1}{6(1-x_2)(1-x_1)}$, \\
        $ $ \\
        $\displaystyle h_{T}(v_1,\textbf{x}) = \frac{1}{6(x_1+x_2)(1-x_2)}$, \\
        $ $ \\
        $\displaystyle h_{T}(v_2,\textbf{x}) = \frac{1}{6(x_1+x_2)(1-x_1)}$.
    \end{minipage}
\end{center}
Then the invariant density for the linear version of the Selmer algorithm is
\begin{align*}
    h_{F}(\mathbf{x}) 
    &= (\varphi^2_S)_*h_{\widetilde{T}}(\mathbf{x}) \\
    &= h_{\widetilde{T}}(v_0,P_{v_0}^{-1}\mathbf{x})\cdot\delta_{\mathbb{R}_+\cdot P_{v_0}\Delta^{2}_{\text{center}}} 
     + h_{\widetilde{T}}(v_1,P_{v_1}^{-1}\mathbf{x})\cdot\delta_{\mathbb{R}_+\cdot P_{v_1}\Delta^{2}_{\text{center}}} \\
    &\quad + h_{\widetilde{T}}(v_2,P_{v_2}^{-1}\mathbf{x})\cdot\delta_{\mathbb{R}_+\cdot P_{v_2}\Delta^{2}_{\text{center}}} \\
    &= \frac{1}{6x_0x_1x_2}\cdot\delta_{\mathbb{R}_+\cdot\Delta^{2}_{\text{center}}},
\end{align*}
and the invariant density for the projective version is 
$$\displaystyle h_{f}(\textbf{x}) = \frac{1}{6(1-x_1-x_2)x_1x_2} \quad\text{for }\textbf{x}\in \Delta^2_{\text{center}}.$$

\subsubsection{Higher dimensions}
\begin{proposition}
    There is an absolutely continuous (with respect to the Lebesgue measure) invariant measure with density $h_{S}: \mathbb{R}_+^{d+1} \to \mathbb{R}_+$ for the Selmer algorithm in dimension $d$, where
    \begin{center}
        $\displaystyle h_{S}(\textbf{x})=\frac{1}{x_0\ldots x_d}\cdot\delta_{\mathbb{R}_+\cdot\Delta^d_{\text{center}}}$.
    \end{center}
\end{proposition}

\begin{proof}
    Consider the linear win-lose induction $\widetilde{T}$ on $G$, i.e., on the simplicial system for the Selmer algorithm in dimension $d$. The dual labeling for the dual simplicial system $S$ is given in Figure~\ref{fig:ssduald_selmer}.
    
    \begin{figure}
        \begin{tikzpicture}[node distance={40mm}, > = Stealth, very thick, main/.style = {draw, circle}] 
            \node (1) {$v_{\sigma}$}; 
            \node[main] (2) [right of=1] {$v_{\sigma\cdot(0\ldots d)}$};
            \node (3) [below of=1] {$ $};
            \node (4) [below of=2] {$ $};
            \draw[->] (2) -- node[midway, above] {$\sigma_1$} (1);
            \draw[->] (3) -- node[midway, left] {$\sigma_{1}$} (1);
            \draw[->] (2) -- node[midway, right] {$\sigma_{d}$} (4);
        \end{tikzpicture}
        \caption{Dual simplicial system for the Selmer algorithm in dimension $d$.}
        \label{fig:ssduald_selmer}
    \end{figure}
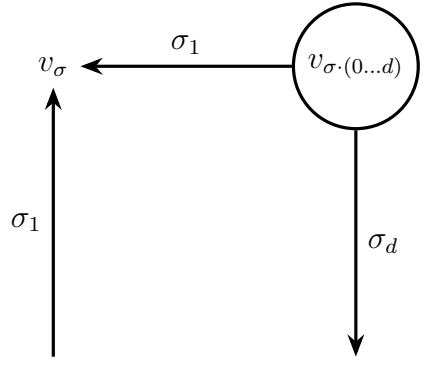
    
    As we mentioned in Section~\ref{ssec:calculations}, the solution to the system of equations~(\ref{eq:1}) is uniquely determined by the matrix $\Omega_{v_{\text{id}}}$. The matrix
    \begin{center}
        $\displaystyle \Omega_{v_{\text{id}}} = \begin{bmatrix}
            0&1&1&\ldots&1 \\ 1&1&\ldots&1&1 \\ 1&2&1&\vdots& 1\\1&2&2&1&\vdots \\1&\ldots&\ldots&1&0
        \end{bmatrix} = P_{v_{\text{id}}}^T$
    \end{center}
    \medskip
    is the solution to the system of equations~(\ref{eq:2}), and 
    $$\displaystyle h_{\widetilde{T}}(v_{\text{id}},\textbf{x}) = \frac{1}{(d+1)!}\prod_{i=0}^d\frac{1}{\big(\textbf{x}^T\Omega_{v_{\text{id}}}\big)_i}.$$
    Then the invariant density for the vectors $\textbf{x}\in\Big(\Lambda_{\sigma\cdot(01\ldots d-1)}^{d+1}\sqcup\Lambda_{\sigma\cdot(01\ldots d)}^{d+1}\Big)\cdot\delta_{\mathbb{R}_+\cdot\Delta^d_{\text{center}}}$ is
    $$\displaystyle h_S(\textbf{x}) = \frac{1}{(d+1)!}\prod_{i=0}^d\frac{1}{\Big(\textbf{x}^T\big(P_{v_{\text{id}}}^{-1}\big)^T\Omega_{v_{\text{id}}}\Big)_i} = \frac{1}{(d+1)!x_0x_1\ldots x_d}.$$
    For the other parts of the cone $\mathbb{R}^{d+1}_+$, the calculations are the same. Hence, we have found the required density (up to a constant). 
\end{proof}

\begin{remark}
    For the non-homogeneous ordered Selmer algorithm acting on $\mathbb{P}\Lambda^{d+1}_+$, we recover the invariant density $h_f(\textbf{x}) = \frac{1}{x_1\ldots x_d} \quad\text{for }\textbf{x}\in \Delta_{\text{center}}^d $, which is consistent with \cite[Theorem~22]{schweiger}.
\end{remark}

\subsection{Triangle algorithm}\label{ssec:triangle}

\subsubsection{Definition}

Let us define the following algorithm in its linear form: for a given vector, we subtract the second largest, and then the smallest coordinate, as many times as possible from the largest one. The version we consider below was recently introduced in \cite{garritynew}, where its construction was motivated by the connection to integer partitions. In \cite{garritynew}, the formula for the invariant ergodic measure of this version was also given, using transfer operators. We construct a simplicial system for this algorithm called the \textit{Triangle algorithm} and find its invariant measure using our graph method. The slow linear version of this algorithm is: 
$$\displaystyle F: \mathbb{R}^{d+1}_+ \to \mathbb{R}^{d+1}_+,$$
$$ x_{\sigma_0} \mapsto
    \begin{cases}
    x_{\sigma_0} - x_{\sigma_1} & \text{if }x_{\sigma_0} < 2x_{\sigma_d}, \\
    x_{\sigma_0} - x_{\sigma_d} & \text{otherwise}
\end{cases}\quad\quad\text{for }x \in \Lambda^{d+1}_{\sigma};$$
and its accelerated version is:
$$\displaystyle F_{\text{acc}}: \mathbb{R}^{d+1}_+ \to \mathbb{R}^{d+1}_+,$$
$$ x_{\sigma_0} \mapsto x_{\sigma_0} - x_{\sigma_1} - kx_{\sigma_d}\quad\quad \text{if } \displaystyle x_{\sigma_1}+kx_{\sigma_d} < x_{\sigma_0} < x_{\sigma_1}+(k+1)x_{\sigma_d}\quad\quad \text{ for }x \in \Lambda^{d+1}_{\sigma}.$$
The matrix map for the accelerated version is: 
$$\displaystyle M=\Id + E_{\sigma_0,\sigma_1} + kE_{\sigma_0,\sigma_d}\quad\quad\text{for }x\in\Lambda^{d+1}_{\sigma}.$$

\begin{remark}
    The slow and accelerated versions of the Triangle algorithm are related in the same way as the Farey and Gauss algorithms (see Sections~\ref{ssec:farey} and \ref{ssec:gauss}).
\end{remark} 

\subsubsection{Constructing a simplicial system}\label{ssec:ss_triangle}

Consider $d!$ cycles on $d+1$ vertices; we label these cycles as $\big\{C_{\sigma}|\;\sigma\in S_{d+1}, \sigma_0=0\big\}$. On each cycle $C_{\sigma}$, the vertices are labeled as $(v_{\sigma,\sigma_0},\ldots, v_{\sigma,\sigma_d})$, and the vertex $v_{\sigma,\sigma_i}$ has two outgoing edges: the loop with label $\sigma_i$ and the edge with label $\sigma_{i+1}$ to the vertex $v_{\sigma,\sigma_{i+1}}$, where the index $i$ is taken by $\text{mod }d+2$  (see Fig.~\ref{fig:ss_triangle}). \\
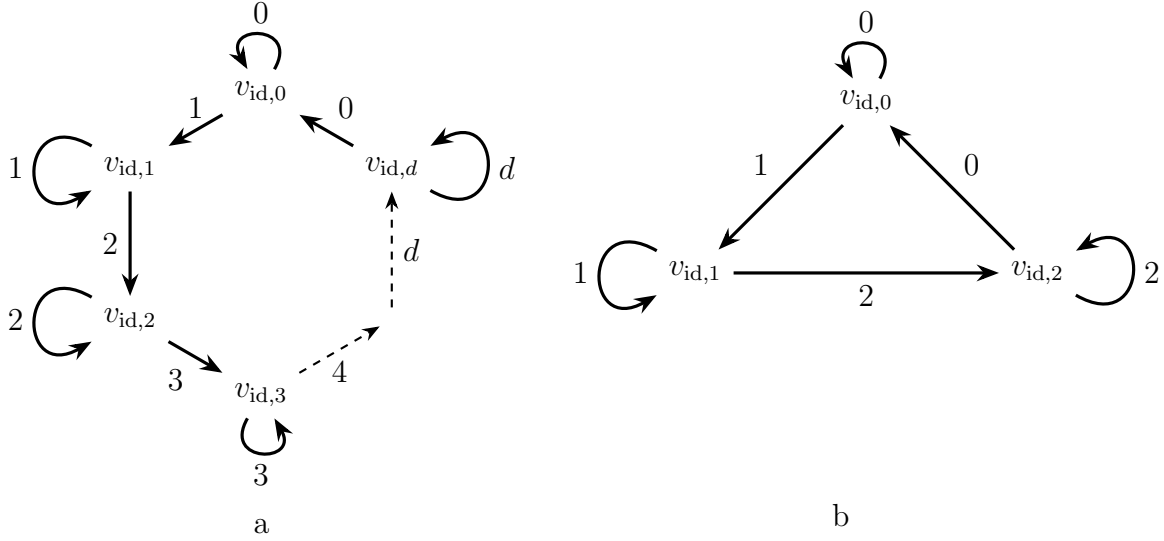
\begin{figure}[ht]
    \begin{minipage}[ht]{0.49\linewidth}
        \center{\begin{tikzpicture}[node distance={40mm}, > = Stealth, very thick] 
            \node (1) at (0,2) {$v_{\text{id},0}$};
            \node (2) at (-1.73,1) {$v_{\text{id},1}$};
            \node (3) at (-1.73,-1) {$v_{\text{id},2}$};
            \node (4) at (0,-2) {$v_{\text{id},3}$};
            \node (5) at (1.73,-1) {$ $};
            \node (6) at (1.73,1) {$v_{\text{id},d}$};
            \draw[->] (1) -- node[midway, above] {$1$} (2);
            \draw[->] (2) -- node[midway, left] {$2$} (3);
            \draw[->] (3) -- node[midway, below left] {$3$} (4);
            \draw[->][thick, dashed] (4) -- node[midway, below] {$4$} (5);
            \draw[->][thick, dashed] (5) -- node[midway, right] {$d$} (6);
            \draw[->] (6) -- node[midway, above right] {$0$} (1);
            \draw[->] (1) to [out=60, in=120, looseness=5] node[midway, above] {$0$} (1);
            \draw[->] (2) to [out=150, in=210, looseness=5] node[midway, left] {$1$} (2);
            \draw[->] (3) to [out=150, in=210, looseness=5] node[midway, left] {$2$} (3);
            \draw[->] (4) to [out=240, in=300, looseness=5] node[midway, below] {$3$} (4);
            \draw[->] (6) to [out=330, in=30, looseness=5] node[midway, right] {$d$} (6);
        \end{tikzpicture}  \\ a}
    \end{minipage}
    \hfill
    \begin{minipage}[ht]{0.49\linewidth}
        \center{\begin{tikzpicture}[node distance={32mm}, > = Stealth, very thick, main/.style = {draw, circle}] 
        \node (1) {$v_{\text{id},0}$}; 
        \node (2) [below left of=1] {$v_{\text{id},1}$};
        \node (3) [below right of=1] {$v_{\text{id},2}$};
        \draw[->] (2) to [out=150, in=210, looseness=5] node[midway, left] {$1$} (2);
        \draw[->] (2) -- node[midway, below] {$2$} (3);
        \draw[->] (3) to [out=330, in=30, looseness=5] node[midway, right] {$2$} (3);
        \draw[->] (1) to [out=60, in=120, looseness=5] node[midway, above] {$0$} (1);
        \draw[->] (1) -- node[midway, above left] {$1$} (2);
        \draw[->] (3) -- node[midway, above right] {$0$} (1);
        \end{tikzpicture}  \\ $ $ \\$ $ \\$ $ \\$ $ \\ b}
    \end{minipage}
    \caption{The cycle $C_{\text{id}}$ of the simplicial system for the Triangle algorithm in a) arbitrary dimension $d$ and b) $d=2$.}
    \label{fig:ss_triangle}
\end{figure}\\
Consider the dimension $d=2$. The linear win-lose induction $\widetilde{T}$ on $G$ is conjugated to the slow version $F$ by the bijection $\varphi^2_{\text{Tr}}$:\\ \\
\begin{minipage}[h]{0.70\linewidth}
    \centering
    $\varphi^2_{\text{Tr}}: V\times\mathbb{R}_+^3\to \mathbb{R}_+^3,$ \\
    \medskip
    $\varphi^2_{\text{Tr}}(v_{\sigma,\sigma_i},\textbf{x})=P_{v_{\sigma,\sigma_i}}\cdot\textbf{x}$, where \\
    $ $ \\
    $P_{v_{(012),0}} = \begin{bmatrix} 1 & 0 & 0 \\ 1 & 1 & 1 \\ 1 & 0 & 1 \end{bmatrix}, P_{v_{(012),1}} = \begin{bmatrix}     1 & 1 & 0 \\ 0 & 1 & 0 \\ 1 & 1 & 1 \end{bmatrix},$\\
    $ $ \\
    $ $ \\
    $P_{v_{(012),2}} = \begin{bmatrix} 1 & 1 & 1 \\ 0 & 1 & 1 \\ 0 & 0 & 1 \end{bmatrix}, P_{v_{(021),0}} = \begin{bmatrix} 1 & 0 & 0 \\ 1 & 1 & 0 \\ 1 & 1 & 1 \end{bmatrix}$,
    \\
    $ $ \\
    $ $ \\
    $P_{v_{(021),2}} = \begin{bmatrix} 1 & 0 & 1 \\ 1 & 1 & 1 \\ 0 & 0 & 1 \end{bmatrix}, P_{v_{(021),1}} = \begin{bmatrix} 1 & 1 & 1 \\ 0 & 1 & 0 \\ 0 & 1 & 1 \end{bmatrix}$,
\end{minipage}
\hfill
\begin{minipage}[ht]{0.28\linewidth}
    \center{\begin{tikzpicture}[node distance={25mm}]
        \node (1) {$V\times\mathbb{R}_+^3$}; 
        \node (2) [right of=1] {$V\times\mathbb{R}_+^3$};
        \node (3) [below of=1] {$\mathbb{R}_+^3$};
        \node (4) [below of=2] {$\mathbb{R}_+^3$};
        \draw[->] (1) -- node[midway, above] {$\widetilde{T}$} (2);
        \draw[->] (3) -- node[midway, above] {$F$} (4);
        \draw[->] (1) -- node[midway, left] {$\varphi^2_{\text{Tr}}$} (3);
        \draw[->] (2) -- node[midway, right] {$\varphi^2_{\text{Tr}}$} (4);
    \end{tikzpicture}}
\end{minipage} \\ \\ \\
i.e., the bijection $\varphi^2_{\text{Tr}}$ sends the pair $(v_{\sigma,\sigma_i},\textbf{x})$ to the vector $\displaystyle P_{v_{\sigma,\sigma_i}}\cdot\textbf{x}\in \Lambda_{\sigma\cdot(\sigma_{i-2}\sigma_{i-1}\sigma_{i})}^{3}.$ \\ \\
In arbitrary dimension, the matrices of $\varphi^d_{\text{Tr}}$ are:
\medskip
\begin{center}
    $P_{v_{\sigma,\sigma_i}} = \begin{bmatrix} P_{v_{\sigma,\sigma_i}}^0 \\ \vdots \\ P_{v_{\sigma,\sigma_i}}^d \end{bmatrix},\quad \big(P_{v_{\sigma,\sigma_i}}^j\big)_k = \begin{cases}
    0 & \text{if }\sigma'_k > \sigma'_j, \\
    1 & \text{otherwise},
    \end{cases}$\\
    $ $ \\
    $ $ \\
    $\text{where }\sigma' = \sigma\cdot(\sigma_{i-d}\ldots\sigma_{i-1}\sigma_{i})$,
\end{center}
i.e., the projection $\varphi^d_{\text{Tr}}$ sends the pair $(v_{\sigma,\sigma_i},\textbf{x})$ to the vector $\displaystyle P_{v_{\sigma,\sigma_i}}\cdot\textbf{x}\in \Lambda_{\sigma\cdot(\sigma_{i-d}\ldots\sigma_{i-1}\sigma_{i})}^{d+1}.$ \\ \\
To construct the projection for the accelerated version $F_{\text{acc}}$, we need (as in Gauss' case) to consider an acceleration of the win-lose induction on $G$. For the cycle $C_{\sigma}$, let 
$$D_{\sigma} := \bigsqcup_{i=0}^d\{v_{\sigma,\sigma_i}\}\times \big(
\mathbb{R}_+\cdot\Delta^d_{\sigma_{i+1}(\sigma_i)}\big)= \bigsqcup_{i=0}^d D_{\sigma,\sigma_i}.$$
The first return map $\widetilde{T}_D$ of the linear win-lose induction $\widetilde{T}$ to $D$ is conjugated to the accelerated version $F_{\text{acc}}$ by the bijection

\begin{center}
    $\pi^d_{\text{Tr}_{\text{acc}}}: D\to \mathbb{R}_+^{d+1},$\\
    \medskip
    $\pi^d_{\text{Tr}_{\text{acc}}}(v_{\sigma,\sigma_i},\textbf{x})=\pi^d_{\text{Tr}}\big({v_{\sigma,\;\sigma_i}},M^{-1}_{\sigma_{i+1}\sigma_i}\textbf{x}\big)$.
\end{center}

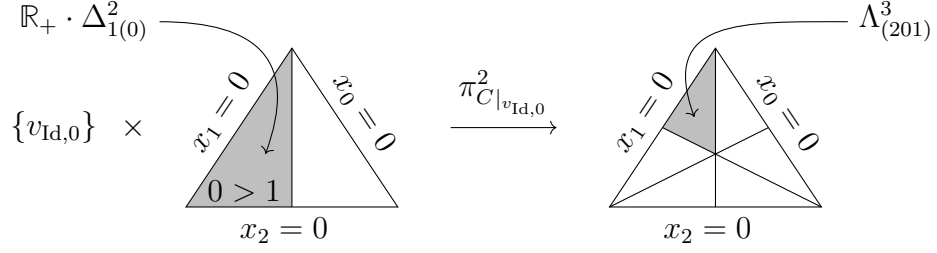
\begin{figure}
    \centering
    \begin{tikzpicture}[yscale=0.7,xscale=0.7]
        \fill[lightgray] (-2,0)--(0,0)--(0,3)--cycle;
        \fill[lightgray] (7,1.5)--(8,3)--(8,1)--cycle;
        \node at (-0.9,0.3){$0>1$};
        \node[below] at (-0.15,0){$x_2=0$};
        \node[above left, rotate=56.30993247402] at (-0.4,2.4){$x_1=0$};
        \node[above right, rotate=303.690067526] at (0.4,2.4){$x_0=0$};
        \draw(-2,0)--(0,3)--(2,0)--cycle;
        \draw(0,3)--(0,0);
        \draw(6,0)--(8,3)--(10,0)--cycle;
        \draw(8,3)--(8,0);
        \draw(9,1.5)--(6,0);
        \draw(7,1.5)--(10,0);
        \node[above] at (4,1.5) {${\pi^2_{C}}_{|_{v_{\text{Id},0}}}$};
        \draw[->] (3,1.5) to (5,1.5);
        \node at (-2.5,3.5)[left]{$\mathbb{R}_+\cdot\Delta^{2}_{1(0)}$};
        \draw[->] (-2.5,3.5) to [out=0, in=70, looseness=1.5] (-0.5,1);
        \node at (-4.5,1.5){$\{v_{\text{Id},0}\}$};
        \node at (-3,1.5){$\times$};
        \node[below] at (7.85,0){$x_2=0$};
        \node[above left, rotate=56.30993247402] at (7.6,2.4){$x_1=0$};
        \node[above right, rotate=303.690067526] at (8.4,2.4){$x_0=0$};
        \node at (10.5,3.5)[right]{$\Lambda^3_{(201)}$};
        \draw[->] (10.5,3.5) to [out=180, in=110, looseness=1.5] (7.6,1.7);
    \end{tikzpicture}
    \caption{The projection of the win-lose induction for the Triangle algorithm in dimension $d=2$.}
    \label{fig:proj2_triangle}
\end{figure}

\subsubsection{Invariant measure for $d=2$}
The solution to the system of linear equations~(\ref{eq:1})
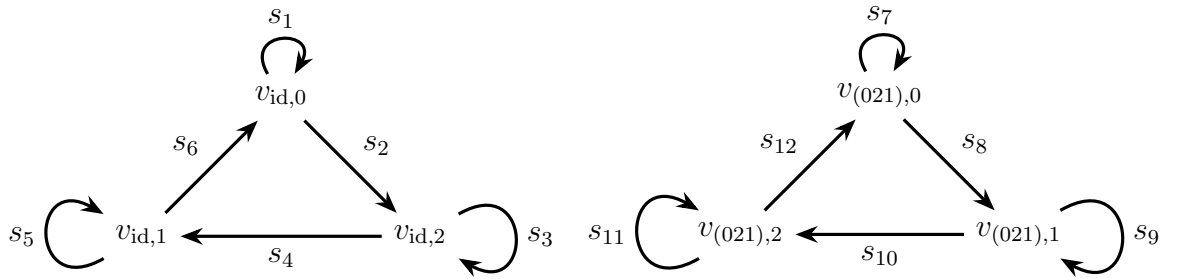
\begin{figure}[ht]
    \begin{minipage}[ht]{0.49\linewidth}
        \center{\begin{tikzpicture}[node distance={26mm}, > = Stealth, very thick, main/.style = {draw, circle}] 
            \node (1) {$v_{\text{id},0}$}; 
            \node (2) [below left of=1] {$v_{\text{id},1}$};
            \node (3) [below right of=1] {$v_{\text{id},2}$};
            \draw[->] (2) to [out=210, in=150, looseness=5] node[midway, left] {$s_5$} (2);
            \draw[->] (3) -- node[midway, below] {$s_4$} (2);
            \draw[->] (3) to [out=30, in=330, looseness=5] node[midway, right] {$s_3$} (3);
            \draw[->] (1) to [out=120, in=60, looseness=5] node[midway, above] {$s_1$} (1);
            \draw[->] (2) -- node[midway, above left] {$s_6$} (1);
            \draw[->] (1) -- node[midway, above right] {$s_2$} (3);
        \end{tikzpicture}}
    \end{minipage}
    \hfill
    \begin{minipage}[ht]{0.49\linewidth}
        \center{\begin{tikzpicture}[node distance={26mm}, > = Stealth, very thick, main/.style = {draw, circle}] 
            \node (1) {$v_{(021),0}$}; 
            \node (2) [below left of=1] {$v_{(021),2}$};
            \node (3) [below right of=1] {$v_{(021),1}$};
            \draw[->] (2) to [out=210, in=150, looseness=5] node[midway, left] {$s_{11}$} (2);
            \draw[->] (3) -- node[midway, below] {$s_{10}$} (2);
            \draw[->] (3) to [out=30, in=330, looseness=5] node[midway, right] {$s_9$} (3);
            \draw[->] (1) to [out=120, in=60, looseness=5] node[midway, above] {$s_7$} (1);
            \draw[->] (2) -- node[midway, above left] {$s_{12}$} (1);
            \draw[->] (1) -- node[midway, above right] {$s_8$} (3);
        \end{tikzpicture}}
    \end{minipage}
    \caption{Dual simplicial system for the Triangle algorithm in dimension $d=2$.}
    \label{fig:ssdual2_triangle}
\end{figure} 
is given by the labeling $(s_1,\ldots,s_{12}) = (0,2,2,1,1,0,0,1,1,2,2,0)$ of the dual simplicial system (see Fig.~\ref{fig:ssdual2_triangle}). This solution is unique up to the permutation of columns of the matrices $\Omega_{v}$: 
\begin{center}
    $\Omega_{v_{(012),0}} = \begin{bmatrix} 1 & 1 & 1 \\ 1 & 0 & 0 \\ 1 & 1 & 0 \end{bmatrix}, \;\;\Omega_{v_{(012),1}} = \begin{bmatrix}     0 & 1 & 1 \\ 1 & 1 & 1 \\ 0 & 1 & 0 \end{bmatrix},\;\; \Omega_{v_{(012),2}} = \begin{bmatrix} 0 & 0 & 1 \\ 1 & 0 & 1 \\ 1 & 1 & 1 \end{bmatrix}$,\\
    $ $ \\
    $ $ \\
    $\Omega_{v_{(021),0}} = \begin{bmatrix} 1 & 1 & 1 \\ 1 & 0 & 1 \\ 1 & 0 & 0 \end{bmatrix},\;\; \Omega_{v_{(021),2}} = \begin{bmatrix} 0 & 1 & 1 \\ 0 & 0 & 1 \\ 1 & 1 & 1 \end{bmatrix},\;\; \Omega_{v_{(021),1}} = \begin{bmatrix} 0 & 1 & 0 \\ 1 & 1 & 1 \\ 1 & 1 & 0 \end{bmatrix}$.
\end{center}
\bigskip
The invariant density for the slow version of the $2$-dimensional Triangle algorithm is:

\medskip
\begin{center}
    $h^2_{\text{Tr}}: \mathbb{R}_+^3 \to \mathbb{R}_+$,
\end{center}
\begin{align*}
    h^2_{\text{Tr}}(\mathbf{x}) 
    &= (\varphi^2_{\text{Tr}})_*h_{\widetilde{T}}(\mathbf{x}) \\
    &= \sum_{\sigma,i}h_{\widetilde{T}}(v_{\sigma,\sigma_i},P_{v_{\sigma,\sigma_i}}^{-1}\mathbf{x})
       \cdot\delta_{\Lambda_{\sigma\cdot(\sigma_{i-2}\sigma_{i-1}\sigma_{i})}^3} \\
    &= \sum_{\sigma,i}\frac{1}{6}\prod_{j=0}^2
       \frac{1}{\Bigl(\mathbf{x}^T\bigl(P_{v_{\sigma,\sigma_i}}^{-1}\bigr)^T\Omega_{v_{\sigma,\sigma_i}}\Bigr)_j}
       \cdot\delta_{\Lambda_{\sigma\cdot(\sigma_{i-2}\sigma_{i-1}\sigma_{i})}^3} \\
    &= \frac{1}{6x_0x_1x_2}.
\end{align*}
\bigskip
For the $2$-dimensional accelerated version, the invariant density is:
\begin{center}
    $h^2_{\text{Tr}_{\text{acc}}}: \mathbb{R}_+^3 \to \mathbb{R}_+$,
\end{center}
\begin{align*}
    h^2_{\text{Tr}_{\text{acc}}}(\mathbf{x}) 
    &= (\varphi^2_{\text{Tr}_{\text{acc}}})_*h_{\widetilde{T}}(\mathbf{x}) \\
    &= \sum_{\sigma,i}h_{\widetilde{T}}\bigl(v_{\sigma,\sigma_i},
       M_{\sigma_{i+1}\sigma_i}P_{v_{\sigma,\sigma_{i+1}}}^{-1}\mathbf{x}\bigr)
       \cdot\delta_{\Lambda_{(\sigma_{i+2}\sigma_{i}\sigma_{i+1})}^3} \\
    &= \sum_{\sigma,i}\frac{1}{6}\prod_{j=0}^2
       \frac{1}{\Bigl(\mathbf{x}^T\bigl(M_{\sigma_{i+1}\sigma_i}
       P_{v_{\sigma,\sigma_{i+1}}}^{-1}\bigr)^T\Omega_{v_{\sigma,\sigma_i}}\Bigr)_j}
       \cdot\delta_{\Lambda_{(\sigma_{i+2}\sigma_{i}\sigma_{i+1})}^3} \\
    &= \sum_{\sigma,i}\frac{1}{6(x_{\sigma_{i+1}}+x_{\sigma_{i+2}})x_{\sigma_{i+2}}x_{\sigma_{i}}}
       \cdot\delta_{\Lambda_{(\sigma_{i+2}\sigma_{i}\sigma_{i+1})}^3} \\
    &= \frac{1}{6}\sum_{\sigma}
       \frac{1}{(x_{\sigma_0}+x_{\sigma_2})x_{\sigma_0}x_{\sigma_1}}
       \cdot\delta_{\Lambda^3_{\sigma}}.
\end{align*}

\subsubsection{Higher dimensions}

\begin{proposition}\label{prop:desnity_triangle}
    There is an absolutely continuous (with respect to the Lebesgue measure) invariant measure with density $h_{\text{Tr}}: \mathbb{R}_+^{d+1} \to \mathbb{R}_+$ for the slow Triangle algorithm in dimension $d$, where
\begin{center}
    $\displaystyle h_{\text{Tr}}(\textbf{x})=\frac{1}{x_0\ldots x_d}$,
\end{center}
and an absolutely continuous (with respect to the Lebesgue measure) invariant measure with the density $h_{\text{Tr}_{\text{acc}}}: \mathbb{R}_+^{d+1} \to \mathbb{R}_+$ for the accelerated Triangle algorithm in dimension $d$:
\begin{center}
    $\displaystyle h_{\text{Tr}_{\text{acc}}}(\textbf{x})=\frac{1}{(x_{\sigma_0}+x_{\sigma_d})x_{\sigma_0}\ldots x_{\sigma_{d-1}}}\quad\text{for }\textbf{x}\in\Lambda^{d+1}_{\sigma}$.
\end{center}
\end{proposition}

\begin{proof}
    Consider the cycle $C_{\text{id}}$ in $G$ and the corresponding cycle in $S$. For the dual simplicial system $S$, we choose the natural labeling, see Fig.~\ref{fig:ssduald_triangle_labelling}.
    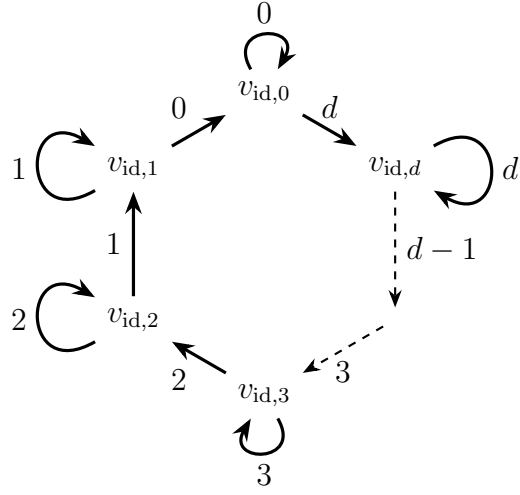
\begin{figure}
        \centering
        \begin{tikzpicture}[node distance={15mm}, > = Stealth, very thick] 
            \node (1) at (0,2) {$v_{\text{id},0}$};
            \node (2) at (-1.73,1) {$v_{\text{id},1}$};
            \node (3) at (-1.73,-1) {$v_{\text{id},2}$};
            \node (4) at (0,-2) {$v_{\text{id},3}$};
            \node (5) at (1.73,-1) {$ $};
            \node (6) at (1.73,1) {$v_{\text{id},d}$};
            \draw[->] (2) -- node[midway, above left] {$0$} (1);
            \draw[->] (3) -- node[midway, left] {$1$} (2);
            \draw[->] (4) -- node[midway, below left] {$2$} (3);
            \draw[->][thick, dashed] (5) -- node[midway, below] {$3$} (4);
            \draw[->][thick, dashed] (6) -- node[midway, right] {$d-1$} (5);
            \draw[->] (1) -- node[midway, above] {$d$} (6);
            \draw[->] (1) to [out=120, in=60, looseness=5] node[midway, above] {$0$} (1);
            \draw[->] (2) to [out=210, in=150, looseness=5] node[midway, left] {$1$} (2);
            \draw[->] (3) to [out=210, in=150, looseness=5] node[midway, left] {$2$} (3);
            \draw[->] (4) to [out=300, in=240, looseness=5] node[midway, below] {$3$} (4);
            \draw[->] (6) to [out=30, in=330, looseness=5] node[midway, right] {$d$} (6);
        \end{tikzpicture}
        \caption{The natural dual labeling for $C_{\text{id}}$.}
        \label{fig:ssduald_triangle_labelling}
    \end{figure}
    The solution to the system of equations~(\ref{eq:1}) for $C_{\text{id}}$, as for a strongly connected component of $G$, is uniquely determined by the matrix $\Omega_{\text{id},0}$. The matrix
    \begin{center}
        $\displaystyle \Omega_{v_{\text{id},0}} = \begin{bmatrix}
        1&1&1&\ldots&1 \\ 1&0&0&0&0 \\ \vdots&1&0&\vdots&\vdots\\1&\ldots&1&\ddots&\vdots \\1&\ldots&\ldots&1&0
        \end{bmatrix}$
    \end{center}
    is the solution to the system of equations~(\ref{eq:1}) for $C_{\text{id}}$, and 
    $$\displaystyle h_{\widetilde{T}}(v_{\text{id},0},\textbf{x}) = \frac{1}{(d+1)!}\prod_{i=0}^d\frac{1}{\big(\textbf{x}^T\Omega_{v_{\text{id},0}}\big)_i}.$$
    The matrix of the projection $\pi^d_{\text{Tr}}$ at this vertex is:
    \medskip
    \begin{center}
        $\displaystyle P_{v_{\text{id},0}} = \begin{bmatrix}
        1&0&0&\ldots&0 \\ 1&1&1&\ldots&1 \\ 1&0&1&\ldots&1\\\vdots&0&0&\ddots&\vdots \\1&0&\ldots&0&1
        \end{bmatrix}$,\quad\quad and since $\displaystyle \big(P_{v_{\text{id},0}}^{-1})^T\cdot\Omega_{v_{\text{id},0}} = \begin{bmatrix}
        0&0&\ldots&0&1 \\ 1&0&0&\ldots&0 \\ 0&1&0&\ldots&0\\\vdots&0&\ddots&\ddots&\vdots \\0&0&\ldots&1&0
        \end{bmatrix}$,
    \end{center}
    \medskip
    the invariant density (for the slow version) for the vectors $\textbf{x}\in\Lambda_{\sigma\cdot(12\ldots d0)}^{d+1}$ is
    $$\displaystyle h_{\text{Tr}}(\textbf{x}) = \frac{1}{(d+1)!}\prod_{i=0}^d\frac{1}{\Big(\textbf{x}^T\big(P_{v_{\text{id},0}}^{-1}\big)^T\Omega_{v_{\text{id},0}}\Big)_i} = \frac{1}{(d+1)!x_0x_1\ldots x_d}.$$
    \bigskip
    By changing the projection from $h_{\text{Tr}}$ to $h_{\text{Tr}_{\text{acc}}}$, we also easily get the invariant density for the accelerated version:
    $$\displaystyle  h_{\text{Tr}_{\text{acc}}}(\textbf{x}) = \frac{1}{(d+1)!}\prod_{i=0}^d\frac{1}{\Big(\textbf{x}^T\big(M_{1,0}P_{v_{\text{id},0}}^{-1}\big)^T\Omega_{v_{\text{id},0}}\Big)_i} = \frac{1}{(d+1)!(x_1+x_2)x_2\ldots x_dx_0}.$$
    \medskip
    For the other parts of the cone $\mathbb{R}^{d+1}_+$, the calculations are the same. Hence, we have found the required densities. 
\end{proof}

\begin{remark}
    The measures found for the slow and accelerated Triangle algorithms in this section coincide with the results of \cite[Thm. 7.1, Prop. 8.3]{garrityerg}.
\end{remark}

\section{Extended simplicial systems}\label{sec:extss}

A significant limitation of our method at this stage is that it allows us to obtain invariant densities only in the form of fractions with a numerator of $1$, according to Formula~(\ref{eq:2}). In other words, for a particular $d$-dimensional algorithm, it allows us to check as invariant domains only convex polyhedra with $d+1$ vertices. 

\subsection{Defining extended systems}
The idea of expanding the range of natural extension domains motivates the following definition.
\begin{definition}\label{def:extss}
    For a given simplicial system $G = (V,E)$, the \textit{$n$-extension} is a directed labeled graph $G^{(n)} = (V^{(n)}, E^{(n)})$ where:
    \begin{itemize}
        \item the vertex set $V^{(n)} = V \times \{1, \ldots, n\}$ consists of $n$ copies of each original vertex;
        \item edges in $E^{(n)}$ connect vertices within and between copies $V\times\{j\}$ ($j = 1,\ldots, n$) according to the initial rule in the graph $G$.
    \end{itemize}
    The corresponding projective and linear win-lose inductions $T^{(n)}$ and $\widehat{T}^{(n)}$ on $G^{(n)}$ are semiconjugated to $T$ and $\widehat{T}$, respectively, with the natural projection:
    $$\displaystyle \pi_{\text{ext}}: V^{(n)}\times\mathbb{R}_+^{d+1} \to V\times\mathbb{R}_+^{d+1},$$
    $$\displaystyle \pi_{\text{ext}}(v,j,\textbf{x}) = (v,\textbf{x})\quad\text{for }j = 1,\ldots, n.$$
\end{definition}
Let us apply the measure construction method described in Section~\ref{sec:measure} to the graph $G^{(n)}$. In the dual graph $S^{(n)}$, edges may connect vertices both within and between copies $V\times\{j\}$ as well. Now, assuming 
$$\big(\Omega_{v^{(j)}}\big)_{(v,j)\in V\times\{1,\ldots, n\}}$$ 
to be a solution to the system of equations~(\ref{eq:1}) for the pair of graphs $(G^{(n)}, S^{(n)})$ and that  $\mu_{\text{ext}}$ is an invariant measure for $\widehat{T}^{(n)}$ with the density
$$h_{\text{ext}} (v,j,\textbf{x}) = \frac{1}{(d+1)!}\prod_{i=0}^d \frac{1}{\big(\textbf{x}^T\Omega_{v^{(j)}}\big)_i},$$
the invariant measure $\mu_{\widehat{T}}$ for the initial win-lose induction $\widehat{T}$ is found as 
$$\mu_{\widehat{T}} = (\pi_{\text{ext}})_*\mu_{\text{ext}},$$ 
and has the density
$$h_{\widehat{T}} (v,\textbf{x}) = \sum_{j=1}^n\frac{1}{(d+1)!}\prod_{i=0}^d \frac{1}{\big(\textbf{x}^T\Omega_{v^{(j)}}\big)_i}.$$
Thus, the construction of the $n$-extension of the simplicial system allows us to search for arbitrary convex polyhedra as invariant domains by splitting them into polyhedra with $d+1$ vertices passing into each other, which correspond to the terms of the sum in $h_{\widehat{T}}$. An example of the use of this construction is the Brun algorithm for $d\geq 3$ (see \cite[$\mathsection$~8.6]{arnouxlabbe} and Section~\ref{ssec:brun}).

\subsection{Dual graphs for first return maps}
Another generalization of our method concerns the construction of a dual system. This generalization is motivated by the fact that first return maps of win-lose inductions arise naturally in the study of MCF algorithms: in particular, the Brun algorithm and the Modified Triangle algorithm are most naturally described as first return maps of simpler win-lose inductions onto a subset of vertices, as described in Sections~\ref{ssec:brun} and~\ref{ssec:modtr}. Indeed, the construction of a simplicial system reflects the ordering process of the coordinates at each step of an MCF algorithm, which often requires several intermediate steps before returning to a fully ordered configuration. As a consequence, the original MCF algorithm is typically conjugated not to the win-lose induction itself, but to its first return map onto the subset of vertices corresponding to fully ordered states. Thus, in order to apply the graph method to these algorithms directly, one needs a notion of a dual graph adapted to the first return map setting. We note that this generalization also opens the way to finding invariant domains of positive codimension for natural extensions of skew product MCF algorithms, a direction that will be pursued in a forthcoming paper.

Suppose we need to find an invariant measure for the first return map of the linear win-lose induction $\widehat{T}$ onto some subset 
$$\displaystyle W\times\mathbb{R}_+^{d+1} \subsetneq V\times\mathbb{R}_+^{d+1}.$$
Then, in the construction of the dual graph in Section~\ref{sec:dualss}, the condition that all edges in $E'$ are the corresponding inverted ones from $E$ is too restrictive: it suffices that the paths between vertices in $W \subsetneq V$ in the dual graph are in one-to-one correspondence with the same paths in the original graph. This means that in this case we have more possibilities for defining the dual graph, since the number of intermediate vertices may take various values. This motivates the following definition.
\begin{definition}\label{def:ftmdual}
    For a given simplicial system $G = (V,E)$ and $W \subsetneq V$, consider the first return map
    $$\widetilde{T}_W: W\times\mathbb{R}_+^{d+1} \to W\times\mathbb{R}_+^{d+1}$$
    of the corresponding linear win-lose induction $\widehat{T}$ to $W\times\mathbb{R}_+^{d+1}$. Then the \textit{$W$-dual simplicial system} $S_W$ for $G$ is a directed labeled graph $S_W = (V', E')$ such that:
    \begin{itemize}
        \item $W \subseteq V'$;
        \item for any two vertices $v,w\in W$, there is a bijection
        $$\displaystyle \psi_{vw}: \Gamma_G(v,w) \to \Gamma_{S_W}(v,w),$$
        where
        $$\displaystyle\Gamma_G(v,w) = \bigcup_{n\geq1}\big\{\gamma = \gamma_1\ldots\gamma_n \in V^n\text{-- path in }G|\; \gamma_1=v, \gamma_n=w, \gamma_j \notin W, \;\;j=2,\ldots,n-1\big\},$$
        $$\displaystyle\Gamma_{S_W}(v,w) = \bigcup_{n\geq1}\big\{\gamma' = \gamma'_1\ldots\gamma'_n \in (V')^n\text{-- path in }S_W|\; \gamma'_1=w, \gamma_n=v, \gamma'_j \notin W, \;\;j=2,\ldots,n-1\big\}.$$
    \end{itemize}
    Then applying our method to the pair $(G,S_W)$ means, instead of solving~(\ref{eq:1}), solving the system of equations:
    \begin{equation}\tag{5}\label{eq:4}
        \displaystyle M^T_{\gamma}\Omega_v = \Omega_{w}M_{\psi_{vw}(\gamma)}\quad \text{for any pair }v,w \in W\text{ and any path }\gamma \in \Gamma_{G}(v,w).
    \end{equation}
\end{definition}

\begin{conjecture}
    Assume that $\Leb(D_{\text{AL}}) > 0$, where $D_{\text{AL}}$ is the invariant domain described in~(\ref{eq:0}). Then there exists a solution $\big(\Omega_{v}\big)_{v\in V}$ to the system of equations~(\ref{eq:4}) such that every matrix $\Omega_v$ has full rank.
\end{conjecture}
This conjecture is supported by the following intuition: when the invariant domain $D_{\text{AL}}$ is a polytope, it can be decomposed into simplices, and each simplex could be associated with a copy of the corresponding vertex and matrix $\Omega_v$. In the general case, one may need to consider infinite covers of simplicial systems, which provide a more flexible framework for constructing the required solution.

\subsection{Brun algorithm}\label{ssec:brun}

Using the Brun algorithm, we illustrate the versatility of extended 
simplicial systems as a tool for computing invariant measures.
In the following two subsections, we compute the invariant densities in dimensions $d=2$ 
and $d=3$, recovering the result of~\cite{arnouxlabbe}.
\subsubsection{Definition}
The Brun algorithm is defined by subtracting the second largest coordinate of a vector from the largest one, see \cite{brun1}, \cite{brun2}. The matrix map is 
$$M=\Id + E_{\sigma_0,\sigma_1}\quad\quad\text{for }x\in\Lambda^{d+1}_{\sigma}.$$
Its linear version is: 
\begin{center}
    $\displaystyle F: \mathbb{R}^{d+1}_+\to\mathbb{R}^{d+1}_+$, \\
    $ $ \\
    $F(\textbf{x}) = (x_0,\ldots,x_{\sigma_0}-x_{\sigma_1},\ldots,x_d)$ for $\textbf{x}\in\Lambda^{d+1}_{\sigma}$;
\end{center}
and its projective version is:
\begin{center}
    $\displaystyle f: \Delta^d\to\Delta^d$, \\
    $ $ \\
    $f(\textbf{x}) = \big(1-x_{\sigma_1}\big)^{-1}(x_0,\ldots,x_{\sigma_0}-x_{\sigma_1},\ldots,x_d)$ for $\textbf{x}\in\Delta^d\cap\Lambda_{\sigma}^{d+1}$.
\end{center}
\subsubsection{Constructing a simplicial system}
Following \cite[$\mathsection$~5.2.1]{fougeronpp}, we construct a simplicial system for the \textit{Brun algorithm}. Consider the directed labeled graph $G$ with $d\cdot(d+1)!$ vertices:
$$V := \big\{v_{\sigma},\text{int}_{\sigma}^{1},\text{int}_{\sigma}^{2},\ldots,\text{int}_{\sigma}^{d-1}|\; \sigma\in S_{d+1}\big\}.$$
Each vertex $v_{\sigma}$ has two outgoing edges: one with label $\sigma_0$ to vertex $v_{\sigma\cdot(0\ldots d)}$ and one with label $\sigma_d$ to the \textit{intermediate} vertex $\text{int}_{\sigma}^{d-1}$. Each \textit{intermediate} vertex $\text{int}_{\sigma}^{i}$ has two outgoing edges: one with label $\sigma_0$ to vertex $v_{\sigma\cdot(0\ldots i)}$ and one with label $\sigma_i$ to the intermediate vertex $\text{int}_{\sigma}^{i-1}$ if $i > 1$, or to the vertex $v_{\sigma}$ if $i=1$  (see Fig.~\ref{fig:ss_brun}). The structure of $G$ reflects the successive ordering of coordinates performed by the Brun algorithm: starting from a permutation $\sigma$, the algorithm progressively reorders the coordinates of the input vector, and each intermediate vertex $\text{int}_{\sigma}^{i}$ represents an intermediate stage in this reordering, where the first $d-i$ coordinates have already been placed in the correct order. We also define the subset
$$\displaystyle V_{\text{perm}} := \big\{v_{\sigma}|\; \sigma \in S_{d+1}\big\}$$
of the set of all vertices $V$, then the construction above makes explicit the conjugacy between the first return map on 
$$V_{\text{perm}}\times\mathbb{R}_{+}^{d+1}$$ 
of the win-lose induction $\widehat{T}$ on $G$ and the Brun algorithm.
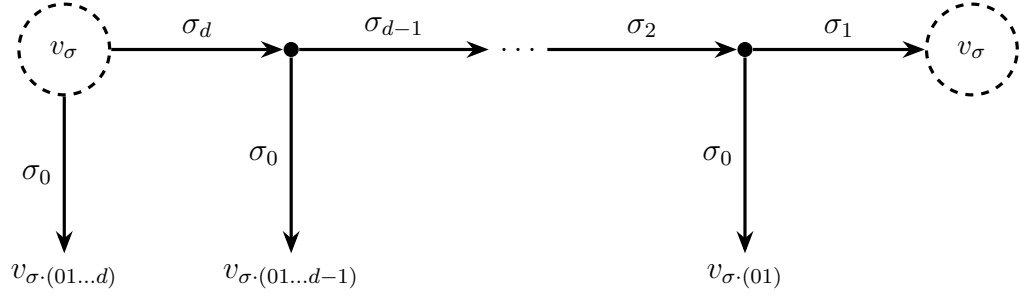
\begin{figure}[ht]
    \begin{minipage}[ht]{0.99\linewidth}
        \center{\begin{tikzpicture}[
        > = Stealth,
        node distance = 4cm,
        state/.style = {circle, draw, very thick, minimum size=1.2cm, font=\small},
        inter/.style = {circle, fill=black, inner sep=2pt}
        ]
            \node[state, dashed] (sigma) at (0,0) {$v_{\sigma}$};
            \node[state, dashed] (sigma0) at (12,0) {$v_{\sigma}$};
            \node (sigman) at (0,-3) {$v_{\sigma\cdot(01\ldots d)}$};
            \node (sigman1) at (3,-3) {$v_{\sigma\cdot(01\ldots d-1)}$};
            \node (sigma2) at (9,-3) {$v_{\sigma\cdot(01)}$};
            \node (inter) at (6, 0) {$\ldots$};
            \node[inter] (1) at (3, 0) {};
            \node[inter] (2) at (9, 0) {};
            \draw[->, very thick] (sigma) to[] node[above, pos=0.5]{$\sigma_d$} (1);
            \draw[->, very thick] (1) to[] node[above, pos=0.5]{$\sigma_{d-1}$} (inter);
            \draw[->, very thick] (inter) to[] node[above, pos=0.5]{$\sigma_2$} (2);
            \draw[->, very thick] (2) to[] node[above, pos=0.5]{$\sigma_1$} (sigma0);
            \draw[->, very thick] (sigma) to[] node[left, pos=0.5]{$\sigma_0$} (sigman);
            \draw[->, very thick] (1) to[] node[left, pos=0.5]{$\sigma_0$} (sigman1);
            \draw[->, very thick] (2) to[] node[left, pos=0.5]{$\sigma_0$} (sigma2);
        \end{tikzpicture} }
    \end{minipage}
    \caption{Simplicial system for the Brun algorithm in dimension $d$.}
    \label{fig:ss_brun}
\end{figure} \\
Consider the dimension $d=2$. The first return map $\widetilde{T}_{V_{\text{perm}}}$ on $G$ is conjugated to the map $F$ by the bijection $\varphi^2_{\text{B}}$ defined in a similar way to $\varphi^2_{\text{Tr}}$ (see Section~\ref{ssec:ss_triangle}):
\begin{minipage}[ht]{0.66\linewidth}
    \centering
    $\varphi^2_{\text{B}}: V_{\text{perm}}\times\mathbb{R}_+^3\to \mathbb{R}_+^3,$ \\
    \medskip
    $\varphi^2_{\text{B}}(v_{\sigma},\textbf{x})=P_{v_{\sigma}}\cdot\textbf{x}$, where \\
    $ $ \\
    $P_{v_{(012)}} = \begin{bmatrix} 1 & 0 & 0 \\ 1 & 1 & 0 \\ 1 & 1 & 1 \end{bmatrix}, P_{v_{(021)}} = \begin{bmatrix}     1 & 0 & 0 \\ 1 & 1 & 1 \\ 1 & 0 & 1 \end{bmatrix},$\\
    $ $ \\
    $ $ \\
    $P_{v_{(102)}} = \begin{bmatrix} 1 & 1 & 0 \\ 0 & 1 & 0 \\ 1 & 1 & 1 \end{bmatrix}, P_{v_{(120)}} = \begin{bmatrix} 1 & 1 & 1 \\ 0 & 1 & 0 \\ 0 & 1 & 1 \end{bmatrix}$,
    \\
    $ $ \\
    $ $ \\
    $P_{v_{(201)}} = \begin{bmatrix} 1 & 0 & 1 \\ 1 & 1 & 1 \\ 0 & 0 & 1 \end{bmatrix}, P_{v_{(210)}} = \begin{bmatrix} 1 & 1 & 1 \\ 0 & 1 & 1 \\ 0 & 0 & 1 \end{bmatrix}$,
\end{minipage}
\hfill
\begin{minipage}[ht]{0.32\linewidth}
    \center{\begin{tikzpicture}[node distance={35mm}]
    \node (1) {$V_{\text{perm}}\times\mathbb{R}_+^3$}; 
    \node (2) [right of=1] {$V_{\text{perm}}\times\mathbb{R}_+^3$};
    \node (3) [below of=1] {$\mathbb{R}_+^3$};
    \node (4) [below of=2] {$\mathbb{R}_+^3$};
    \draw[->] (1) -- node[midway, above] {$\widetilde{T}_{V_{\text{perm}}}$} (2);
    \draw[->] (3) -- node[midway, above] {$F$} (4);
    \draw[->] (1) -- node[midway, left] {$\varphi^2_{\text{B}}$} (3);
    \draw[->] (2) -- node[midway, right] {$\varphi^2_{\text{B}}$} (4);
\end{tikzpicture}}
\end{minipage} \\ \\ \\
i.e., the bijection $\varphi^2_{\text{B}}$ sends the pair $(v_{\sigma},\textbf{x})$ to the vector $\displaystyle P_{v_{\sigma}}\cdot\textbf{x}\in \Lambda_{\sigma}^{3}.$ \\ \\
In arbitrary dimension, the projection matrices are:
\begin{center}
    $\displaystyle P_{v_{\sigma}} = C_{\sigma}\cdot\begin{bmatrix} 1 & 0 & \ldots & 0 \\ 1 & 1 & \ldots & 0  \\ \vdots & \vdots &\ddots & \vdots \\ 1 & 1 & \ldots & 1 \end{bmatrix} \cdot C^{-1}_{\sigma},\quad\text{where }C_{\sigma}\text{ is the matrix of }\sigma \in S_{d+1}$,
\end{center}
i.e., the projection $\varphi^d_{B}$ sends the pair $(v_{\sigma},\textbf{x})$ to the vector in $\displaystyle \Lambda_{\sigma}^{d+1}.$

\subsubsection{Invariant measure for $d=2$}

Note that we can identify the intermediate vertices $\text{int}_{\sigma}^{1}$ and $\text{int}_{(\sigma_0\sigma_1)\sigma}^{1}$ for any $\sigma \in S_{d+1}$, since the pairs of outgoing edges from these two vertices coincide (the edges have the same ending vertices and labels, respectively). In the case $d=2$, this gives the property that in the new graph, any vertex has $2$ incoming edges (and $2$ outgoing, by the initial construction), and the graph becomes even more symmetric; see Fig.~\ref{fig:ss2_brun}. 

\begin{figure}[ht]
\center{ 
    \begin{tikzpicture}[scale=0.7,
        > = Stealth,
        node distance = 4cm,
        state/.style = {circle, draw, very thick, minimum size=0.8cm, font=\normalsize},
        inter/.style = {circle, fill=black, inner sep=2pt}
        ]
        \node[state] (123) at (0,0) {$123$};
        \node[state] (312) at (3,-3) {$312$};
        \node[state] (231) at (-3, -3) {$231$};
        \node[state] (213) at (0,5) {$213$};
        \node[state] (132) at (7,-6) {$132$};
        \node[state] (321) at (-7, -6) {$321$};
        \node[inter] (1) at (0, 2.5) {};
        \node[inter] (2) at (5, -4.5) {};
        \node[inter] (3) at (-5, -4.5) {};
        \draw[->, very thick] (123) to[out=200, in=70, looseness=1] node[above left, pos=0.5]{$1$} (231);
        \draw[->, very thick] (231) to[out=340, in=200, looseness=1] node[below, pos=0.5]{$2$} (312);
        \draw[->, very thick] (312) to[out=110, in=340, looseness=1] node[above right, pos=0.5]{$3$} (123);
        \draw[->, very thick] (213) to[out=250, in=120, looseness=1] node[left, pos=0.5]{$3$} (1);
        \draw[->, very thick] (1) to[out=50, in=290, looseness=1] node[right, pos=0.5]{$1$} (213);
        \draw[->, very thick] (123) to[out=70, in=300, looseness=1] node[right, pos=0.5]{$3$} (1);
        \draw[->, very thick] (1) to[out=230, in=110, looseness=1] node[left, pos=0.5]{$1$} (123);
        \draw[->, very thick] (132) to[out=120, in=350, looseness=1] node[above right, pos=0.5]{$2$} (2);
        \draw[->, very thick] (2) to[out=270, in=160, looseness=1] node[below left, pos=0.5]{$3$} (132);
        \draw[->, very thick] (312) to[out=300, in=170, looseness=1] node[below left, pos=0.5]{$2$} (2);
        \draw[->, very thick] (2) to[out=100, in=340, looseness=1] node[above right, pos=0.5]{$1$} (312);
        \draw[->, very thick] (321) to[out=60, in=190, looseness=1] node[above left, pos=0.5]{$1$} (3);
        \draw[->, very thick] (3) to[out=270, in=20, looseness=1] node[below right, pos=0.5]{$2$} (321);
        \draw[->, very thick] (231) to[out=240, in=10, looseness=1] node[below right, pos=0.5]{$1$} (3);
        \draw[->, very thick] (3) to[out=80, in=200, looseness=1] node[above left, pos=0.5]{$3$} (231);
        \draw[->, very thick] (321) to[out=110, in=190, looseness=1] node[above left, pos=0.5]{$3$} (213);
        \draw[->, very thick] (213) to[out=350, in=70, looseness=1] node[above right, pos=0.5]{$2$} (132);
        \draw[->, very thick] (132) to[out=230, in=310, looseness=1] node[below, pos=0.5]{$1$} (321);
    \end{tikzpicture}}
\caption{Simplicial system for the Brun algorithm in dimension $d=2$.}
\label{fig:ss2_brun}
\end{figure}
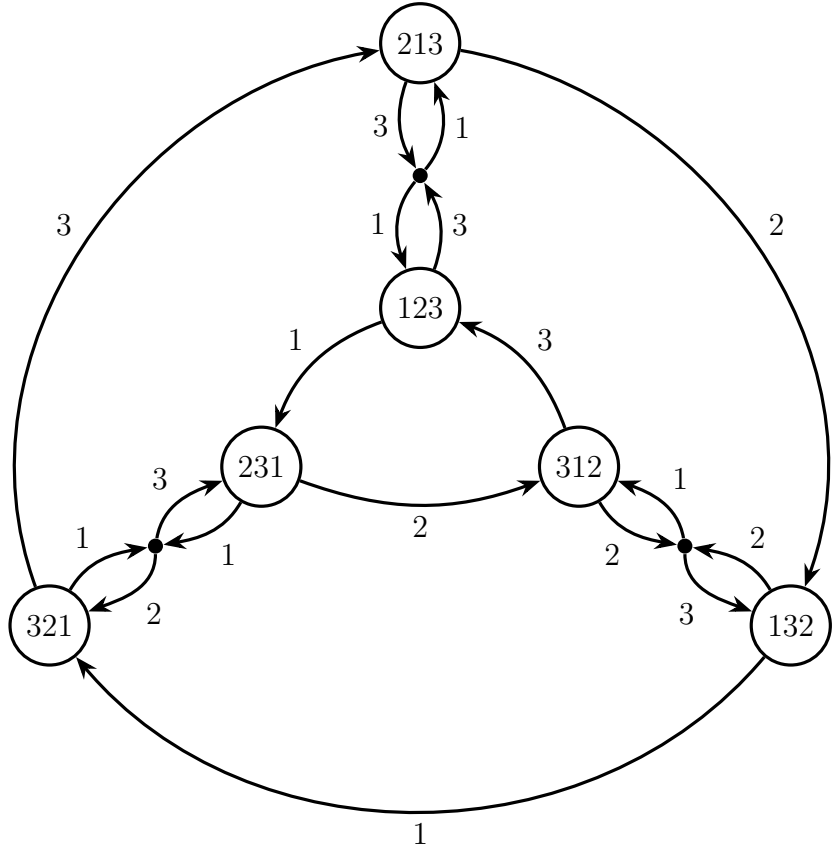 

The solution to the system of linear equations~(\ref{eq:4}) is given by the labeling of the $V_{\text{perm}}$-dual simplicial system $S_{V_{\text{perm}}}$ in Fig.~\ref{fig:ssdual2_brun}. This solution is unique up to the permutation of columns of the matrices $\Omega_{v}$:
\begin{center}
    $\Omega_{v_{(012)}} = \begin{bmatrix} 1 & 0 & 1 \\ 1 & 1 & 1 \\ 1 & 1 & 2 \end{bmatrix}, \;\;\Omega_{v_{(021)}} = \begin{bmatrix} 1 & 1 & 0 \\ 1 & 2 & 1 \\ 1 & 1 & 1 \end{bmatrix},\;\; \Omega_{v_{(102)}} = \begin{bmatrix} 1 & 1 & 1 \\ 0 & 1 & 1 \\ 1 & 1 & 2 \end{bmatrix}$,\\
    $ $ \\
    $ $ \\
    $\Omega_{v_{(120)}} = \begin{bmatrix} 2 & 1 & 1 \\ 1 & 1 & 0 \\ 1 & 1 & 1 \end{bmatrix},\;\; \Omega_{v_{(201)}} = \begin{bmatrix} 1 & 1 & 1 \\ 1 & 2 & 1 \\ 0 & 1 & 1 \end{bmatrix},\;\; \Omega_{v_{(210)}} = \begin{bmatrix} 2 & 1 & 1 \\ 1 & 1 & 1 \\ 1 & 0 & 1 \end{bmatrix}$.
\end{center}

\begin{figure}[ht]
    \center{\begin{tikzpicture}[scale=0.7,
    > = Stealth,
    node distance = 4cm,
    state/.style = {circle, draw, very thick, minimum size=0.8cm, font=\normalsize},
    inter/.style = {circle, fill=black, inner sep=2pt}
    ]
        \node[state] (123) at (0,0) {$123$};
        \node[state] (312) at (3,-3) {$312$};
        \node[state] (231) at (-3, -3) {$231$};
        \node[state] (213) at (0,5) {$213$};
        \node[state] (132) at (7,-6) {$132$};
        \node[state] (321) at (-7, -6) {$321$};
        \node[inter] (1) at (0, 2.5) {};
        \node[inter] (2) at (5, -4.5) {};
        \node[inter] (3) at (-5, -4.5) {};
        \draw[->, very thick] (231) to[out=70, in=200, looseness=1] node[above left, pos=0.5]{$3$} (123);
        \draw[->, very thick] (312) to[out=200, in=340, looseness=1] node[below, pos=0.5]{$1$} (231);
        \draw[->, very thick] (123) to[out=340, in=110, looseness=1] node[above right, pos=0.5]{$2$} (312);
        \draw[->, very thick] (1) to[out=120, in=250, looseness=1] node[left, pos=0.5]{$2$} (213);
        \draw[->, very thick] (213) to[out=290, in=50, looseness=1] node[right, pos=0.5]{$3$} (1);
        \draw[->, very thick] (1) to[out=300, in=70, looseness=1] node[right, pos=0.5]{$1$} (123);
        \draw[->, very thick] (123) to[out=110, in=230, looseness=1] node[left, pos=0.5]{$3$} (1);
        \draw[->, very thick] (2) to[out=350, in=120, looseness=1] node[above right, pos=0.5]{$1$} (132);
        \draw[->, very thick] (132) to[out=160, in=270, looseness=1] node[below left, pos=0.5]{$2$} (2);
        \draw[->, very thick] (2) to[out=170, in=300, looseness=1] node[below left, pos=0.5]{$3$} (312);
        \draw[->, very thick] (312) to[out=340, in=100, looseness=1] node[above right, pos=0.5]{$2$} (2);
        \draw[->, very thick] (3) to[out=190, in=60, looseness=1] node[above left, pos=0.5]{$3$} (321);
        \draw[->, very thick] (321) to[out=20, in=270, looseness=1] node[below right, pos=0.5]{$1$} (3);
        \draw[->, very thick] (3) to[out=10, in=240, looseness=1] node[below right, pos=0.5]{$2$} (231);
        \draw[->, very thick] (231) to[out=200, in=80, looseness=1] node[above left, pos=0.5]{$1$} (3);
        \draw[->, very thick] (213) to[out=190, in=110, looseness=1] node[above left, pos=0.5]{$1$} (321);
        \draw[->, very thick] (132) to[out=70, in=350, looseness=1] node[above right, pos=0.5]{$3$} (213);
        \draw[->, very thick] (321) to[out=310, in=230, looseness=1] node[below, pos=0.5]{$2$} (132);
    \end{tikzpicture} 
    }
    \caption{$V_{\text{perm}}$-dual simplicial system for the Brun algorithm in dimension $d=2$.}
    \label{fig:ssdual2_brun}
\end{figure}
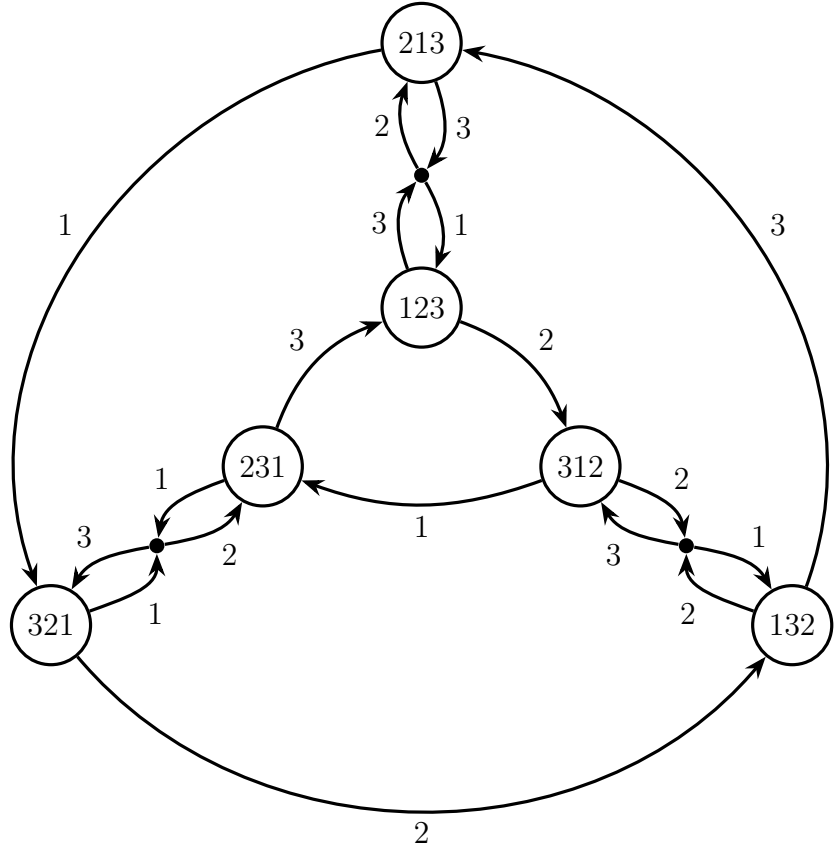

The invariant density for the $2$-dimensional Brun algorithm is then:

\medskip
\begin{center}
    $h^2_{\text{B}}: \mathbb{R}_+^3 \to \mathbb{R}_+$,
\end{center}
\begin{align*}
h^2_{\text{B}}(\mathbf{x}) 
&= (\varphi^2_{\text{B}})_*h_{\widetilde{T}_{V_{\text{perm}}}}(\mathbf{x}) \\
&= \sum_{\sigma}h_{\widetilde{T}_{V_{\text{perm}}}}\bigl(v_{\sigma},
   P_{v_{\sigma}}^{-1}\mathbf{x}\bigr)
   \cdot\delta_{\Lambda_{\sigma}^3} \\
&= \sum_{\sigma}\frac{1}{6}\prod_{j=0}^2
   \frac{1}{\Bigl(\mathbf{x}^T\bigl(P_{v_{\sigma}}^{-1}\bigr)^T
   \Omega_{v_{\sigma}}\Bigr)_j}
   \cdot\delta_{\Lambda_{\sigma}^3} \\
&= \frac{1}{6x_{\sigma_0}x_{\sigma_1}(x_{\sigma_0}+x_{\sigma_2})}
   \quad\text{for } \mathbf{x} \in \Lambda^3_{\sigma}.
\end{align*}
\subsubsection{Invariant measure for $d=3$}
In dimension $d=3$, we cannot find any solution to the system of equations~(\ref{eq:4}) for the first return map $\widehat{T}_{\text{perm}}$ of the win-lose induction conjugated to the Brun algorithm (the construction of the bijection $\varphi_B^3$ is described in the previous Section). Numerical experiments show that the invariant domain for this induction could be the union of two disjoint tetrahedrons with a common face; it has motivated the construction of the $2$-extension $G^{(2)} = (V^{(n)}, E^{(n)})$ of the simplicial system $G = (V,E)$. If we define the rule of changing the copies of vertices $V\times\{1,2\}$ as shown in Fig.~\ref{fig:ss_brun4}a, we find the solutions to the \textit{extended} system of equations~(\ref{eq:4}):
\begin{center}
    $\displaystyle \Omega_{v^{(1)}_{\sigma}} = C_{\sigma}\cdot\begin{bmatrix} 1 & 0 & 1 & 1 \\ 1 & 1 & 1 & 1  \\ 1 & 1 & 2 & 2 \\ 1 & 1 & 2 & 3 \end{bmatrix} \cdot C^{-1}_{\sigma},\quad \Omega_{v^{(2)}_{\sigma}} = C_{\sigma}\cdot\begin{bmatrix} 1 & 0 & 1 & 1 \\ 1 & 1 & 1 & 1  \\ 1 & 1 & 2 & 1 \\ 1 & 1 & 3 & 2 \end{bmatrix} \cdot C^{-1}_{\sigma}$.
\end{center}
The corresponding dual labeling is given in Fig.~\ref{fig:ss_brun4}b.

The invariant density for the $3$-dimensional Brun algorithm is then:

\medskip
\begin{center}
    $h^3_{\text{B}}: \mathbb{R}_+^4 \to \mathbb{R}_+$,
\end{center}
\begin{align*}
h^3_{\text{B}}(\mathbf{x}) 
&= (\varphi^3_{\text{B}}\circ\pi_{\text{ext}})_*h_{\widetilde{T}^{(2)}_{V_{\text{perm}}}}(\mathbf{x}) \\
&= \sum_{\sigma,\,i=1,2}h_{\widetilde{T}_{V_{\text{perm}}}}\bigl(v^{(i)}_{\sigma},
   P_{v^{(i)}_{\sigma}}^{-1}\mathbf{x}\bigr)
   \cdot\delta_{\Lambda_{\sigma}^4} \\
&= \sum_{\sigma,\,i=1,2}\frac{1}{6}\prod_{j=0}^3
   \frac{1}{\Bigl(\mathbf{x}^T\bigl(P_{v^{(i)}_{\sigma}}^{-1}\bigr)^T
   \Omega_{v^{(i)}_{\sigma}}\Bigr)_j}
   \cdot\delta_{\Lambda_{\sigma}^4} \\
&= \frac{1}{6x_{\sigma_0}x_{\sigma_1}(x_{\sigma_0}+x_{\sigma_2}+x_{\sigma_3})}
   \cdot\Biggl[\frac{1}{x_{\sigma_0}+x_{\sigma_2}}
   +\frac{1}{x_{\sigma_0}+x_{\sigma_3}}\Biggr]
   \quad\text{for } \mathbf{x} \in \Lambda^4_{\sigma}.
\end{align*}
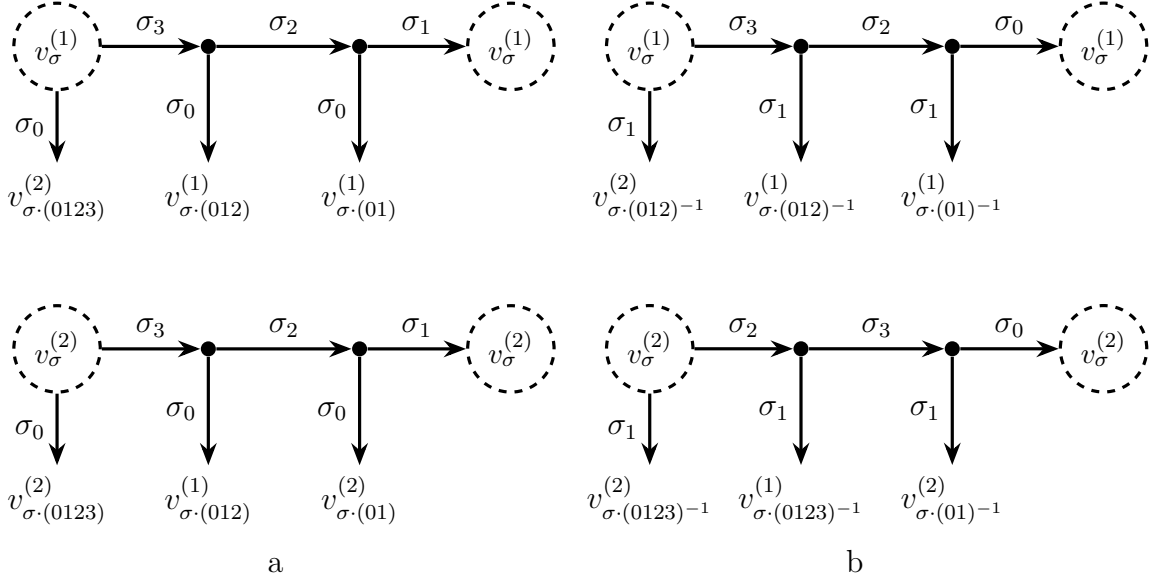
\begin{figure}[ht]
    \begin{minipage}[ht]{0.49\linewidth}
        \center{
        \begin{tikzpicture}[
        > = Stealth,
        node distance = 4cm,
        state/.style = {circle, draw, very thick, minimum size=1cm, font=\small},
        inter/.style = {circle, fill=black, inner sep=2pt}
        ]
            \node[state, dashed] (sigma) at (0,0) {$v^{(1)}_{\sigma}$};
            \node[state, dashed] (sigma0) at (6,0) {$v^{(1)}_{\sigma}$};
            \node (sigman) at (0,-2) {$v^{(2)}_{\sigma\cdot(0123)}$};
            \node (sigman1) at (2,-2) {$v^{(1)}_{\sigma\cdot(012)}$};
            \node (sigma2) at (4,-2) {$v^{(1)}_{\sigma\cdot(01)}$};
            \node[inter] (1) at (2, 0) {};
            \node[inter] (2) at (4, 0) {};
            \draw[->, very thick] (sigma) to[] node[above, pos=0.5]{$\sigma_3$} (1);
            \draw[->, very thick] (1) to[] node[above, pos=0.5]{$\sigma_2$} (2);
            \draw[->, very thick] (2) to[] node[above, pos=0.5]{$\sigma_1$} (sigma0);
            \draw[->, very thick] (sigma) to[] node[left, pos=0.5]{$\sigma_0$} (sigman);
            \draw[->, very thick] (1) to[] node[left, pos=0.5]{$\sigma_0$} (sigman1);
            \draw[->, very thick] (2) to[] node[left, pos=0.5]{$\sigma_0$} (sigma2);
            \node[state, dashed] (sigmac) at (0,-4) {$v^{(2)}_{\sigma}$};
            \node[state, dashed] (sigma0c) at (6,-4) {$v^{(2)}_{\sigma}$};
            \node (sigmanc) at (0,-6) {$v^{(2)}_{\sigma\cdot(0123)}$};
            \node (sigman1c) at (2,-6) {$v^{(1)}_{\sigma\cdot(012)}$};
            \node (sigma2c) at (4,-6) {$v^{(2)}_{\sigma\cdot(01)}$};
            \node[inter] (3) at (2, -4) {};
            \node[inter] (4) at (4, -4) {};
            \draw[->, very thick] (sigmac) to[] node[above, pos=0.5]{$\sigma_3$} (3);
            \draw[->, very thick] (3) to[] node[above, pos=0.5]{$\sigma_2$} (4);
            \draw[->, very thick] (4) to[] node[above, pos=0.5]{$\sigma_1$} (sigma0c);
            \draw[->, very thick] (sigmac) to[] node[left, pos=0.5]{$\sigma_0$} (sigmanc);
            \draw[->, very thick] (3) to[] node[left, pos=0.5]{$\sigma_0$} (sigman1c);
            \draw[->, very thick] (4) to[] node[left, pos=0.5]{$\sigma_0$} (sigma2c);
        \end{tikzpicture} \\ a}
    \end{minipage}
    \hfill
    \begin{minipage}[ht]{0.49\linewidth}
        \center{
        \begin{tikzpicture}[
        > = Stealth,
        node distance = 4cm,
        state/.style = {circle, draw, very thick, minimum size=1cm, font=\small},
        inter/.style = {circle, fill=black, inner sep=2pt}
        ]
            \node[state, dashed] (sigma) at (0,0) {$v^{(1)}_{\sigma}$};
            \node[state, dashed] (sigma0) at (6,0) {$v^{(1)}_{\sigma}$};
            \node (sigman) at (0,-2) {$v^{(2)}_{\sigma\cdot(012)^{-1}}$};
            \node (sigman1) at (2,-2) {$v^{(1)}_{\sigma\cdot(012)^{-1}}$};
            \node (sigma2) at (4,-2) {$v^{(1)}_{\sigma\cdot(01)^{-1}}$};
            \node[inter] (1) at (2, 0) {};
            \node[inter] (2) at (4, 0) {};
            \draw[->, very thick] (sigma) to[] node[above, pos=0.5]{$\sigma_3$} (1);
            \draw[->, very thick] (1) to[] node[above, pos=0.5]{$\sigma_2$} (2);
            \draw[->, very thick] (2) to[] node[above, pos=0.5]{$\sigma_0$} (sigma0);
            \draw[->, very thick] (sigma) to[] node[left, pos=0.5]{$\sigma_1$} (sigman);
            \draw[->, very thick] (1) to[] node[left, pos=0.5]{$\sigma_1$} (sigman1);
            \draw[->, very thick] (2) to[] node[left, pos=0.5]{$\sigma_1$} (sigma2);
            \node[state, dashed] (sigmac) at (0,-4) {$v^{(2)}_{\sigma}$};
            \node[state, dashed] (sigma0c) at (6,-4) {$v^{(2)}_{\sigma}$};
            \node (sigmanc) at (0,-6) {$v^{(2)}_{\sigma\cdot(0123)^{-1}}$};
            \node (sigman1c) at (2,-6) {$v^{(1)}_{\sigma\cdot(0123)^{-1}}$};
            \node (sigma2c) at (4,-6) {$v^{(2)}_{\sigma\cdot(01)^{-1}}$};
            \node[inter] (3) at (2, -4) {};
            \node[inter] (4) at (4, -4) {};
            \draw[->, very thick] (sigmac) to[] node[above, pos=0.5]{$\sigma_2$} (3);
            \draw[->, very thick] (3) to[] node[above, pos=0.5]{$\sigma_3$} (4);
            \draw[->, very thick] (4) to[] node[above, pos=0.5]{$\sigma_0$} (sigma0c);
            \draw[->, very thick] (sigmac) to[] node[left, pos=0.5]{$\sigma_1$} (sigmanc);
            \draw[->, very thick] (3) to[] node[left, pos=0.5]{$\sigma_1$} (sigman1c);
            \draw[->, very thick] (4) to[] node[left, pos=0.5]{$\sigma_1$} (sigma2c);
        \end{tikzpicture} \\ b}
    \end{minipage}
    \caption{(a) $2$-extension of the simplicial system for the Brun algorithm in dimension $d=3$ and (b) its $V_{\text{perm}}$-dual.}
    \label{fig:ss_brun4}
\end{figure}

\subsection{Modified Triangle algorithm}\label{ssec:modtr}
\subsubsection{Definition}\label{sssec:triangledef}
The Modified Triangle algorithm is defined by subtracting the smallest
coefficient from the largest one as many times as possible when all other
coefficients have already been subtracted. Its linear version is: 
$$\displaystyle F: \mathbb{R}^{d+1}_+ \to \mathbb{R}^{d+1}_+,$$
\footnotesize
$$ x_{\sigma_0} \mapsto\begin{cases}
    x_{\sigma_0} - \sum_{i=1}^{j} x_{\sigma_i} & \text{if }\displaystyle x_{\sigma_0} < \sum_{i=1}^{j+1}x_{\sigma_i},\;\;j<d-1, \\
    x_{\sigma_0} - \sum_{i=1}^{d-1} x_{\sigma_i}-kx_{\sigma_d} & \text{if }\displaystyle \sum_{i=1}^{d-1} x_{\sigma_i}+kx_{\sigma_d} < x_{\sigma_0} < \sum_{i=1}^{d-1} x_{\sigma_i}+(k+1)x_{\sigma_d},\;\;k\geq0.
\end{cases}$$
\normalsize
For the vectors from the second line, the matrix map is 
$$\displaystyle M=\Id + \sum_{i=1}^{d-1}E_{\sigma_0,\sigma_i} + kE_{\sigma_0,\sigma_d}\quad\quad\text{for }x\in\Lambda^{d+1}_{\sigma}.$$

This version of the Triangle algorithm was proposed by Garrity \cite{garrity}, before the version we described in Section~\ref{ssec:triangle}. However, we call this version modified because its construction, and (as we will see later) the question of the invariant domain and measure, is more complex than for the other version. Note that estimates of Lyapunov exponents for this version are given in \cite{berthesteinerthuswaldner}.

\begin{figure}[ht]
    \begin{minipage}[ht]{0.99\linewidth}
        \center{\begin{tikzpicture}[
        > = Stealth,
        node distance = 4cm,
        state/.style = {circle, draw, very thick, minimum size=1cm, font=\small},
        inter/.style = {circle, fill=black, inner sep=2pt}
        ]
            \node[state] (sigma) at (0,0) {$v_{\sigma}$};
            \node[state, dashed] (sigman0) at (12,-3) {$v_{\sigma\cdot(01\ldots d)}$};
            \node[state, dashed] (sigman) at (0,-3) {$v_{\sigma\cdot(01\ldots d)}$};
            \node (sigman1) at (3,-3) {$v_{\sigma\cdot(01\ldots d-1)}$};
            \node (sigma2) at (9,-3) {$v_{\sigma\cdot(012)}$};
            \node (inter) at (6, 0) {$\ldots$};
            \node[inter] (1) at (3, 0) {};
            \node[inter] (2) at (9, 0) {};
            \node[inter] (3) at (12, 0) {};
            \draw[->, very thick] (sigma) to[] node[above, pos=0.5]{$\sigma_d$} (1);
            \draw[->, very thick] (1) to[] node[above, pos=0.5]{$\sigma_{d-1}$} (inter);
            \draw[->, very thick] (inter) to[] node[above, pos=0.5]{$\sigma_3$} (2);
            \draw[->, very thick] (2) to[] node[above, pos=0.5]{$\sigma_1$} (3);
            \draw[->, very thick] (1) to[] node[left, pos=0.5]{$\sigma_0$} (sigman1);
            \draw[->, very thick] (2) to[] node[left, pos=0.5]{$\sigma_0$} (sigma2);
            \draw[->, very thick] (sigma) to[] node[left, pos=0.5]{$\sigma_0$} (sigman);
            \draw[->, very thick] (3) to[] node[left, pos=0.5]{$\sigma_0$} (sigman0);
            \draw[->, very thick] (3) edge[loop right, min distance=15mm, looseness=8] node[right]{$\sigma_d$} (3);
        \end{tikzpicture}  }
    \end{minipage}
    \caption{Simplicial system for the Modified Triangle algorithm in dimension $d$.}
    \label{fig:ss_mtr}
\end{figure}
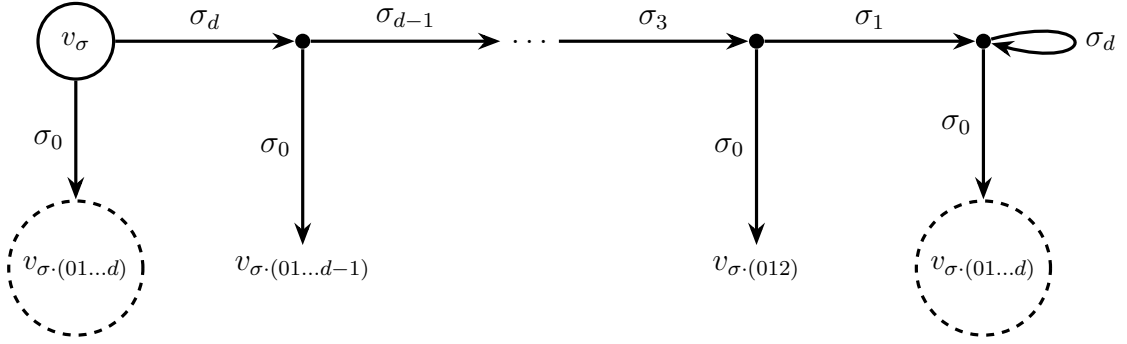 

\subsubsection{Constructing  a  simplicial  system and an invariant domain}\label{sec:fractdomain}

As for the Brun algorithm, the simplicial system $G$ for the Modified Triangle algorithm is defined with $d\cdot(d+1)!$ vertices:
$$V := \big\{v_{\sigma},\text{int}_{\sigma}^{1},\text{int}_{\sigma}^{2},\ldots,\text{int}_{\sigma}^{d-1}|\; \sigma\in S_{d+1}\big\}.$$
Each vertex $v_{\sigma}$ has two outgoing edges, one with label $\sigma_0$ to vertex $v_{\sigma\cdot(0\ldots d)}$, and one with label $\sigma_d$ to the \textit{intermediate} vertex $\text{int}_{\sigma}^{d-1}$. Each \textit{intermediate} vertex $\text{int}_{\sigma}^{i}$ has two outgoing edges: with label $\sigma_0$ to vertex $v_{\sigma\cdot(0\ldots i)}$ and with label $\sigma_i$ to the intermediate vertex $\text{int}_{\sigma}^{i-1}$ if $i > 1$. The intermediate vertex $\text{int}_{\sigma}^{1}$ has an outgoing loop labeled by $\sigma_d$ and an outgoing edge to $v_{\sigma}$ labeled by $\sigma_0$ (see Fig.~\ref{fig:ss_mtr}). Then the first return map $\widehat{T}_{V_{\text{perm}}}$ is conjugated to the Modified Triangle algorithm in the same way as in Brun's case, with the same matrices for the bijection $\varphi^d_{\text{MTr}}$. Note that the corresponding first return map is constructed according to the fact that we have not defined the slow version of the algorithm; so the number of loops in the intermediate vertex $\text{int}^1_{\sigma}$ that precede the return to $V_{\text{perm}}$ corresponds to the number of subtractions of the smallest coordinate $x_{\sigma_d}$.

\begin{figure}[ht]
    \begin{minipage}[ht]{0.49\linewidth}
        \center{
        \begin{tikzpicture}[
        > = Stealth,
        node distance = 4cm,
        state/.style = {circle, draw, very thick, minimum size=1cm, font=\small},
        inter/.style = {circle, fill=black, inner sep=2pt}
        ]
            \node[state] (sigma) at (0,0) {$v_{\sigma}$};
            \node (sigma2) at (0,-3) {$v_{\sigma\cdot(012)}$};
            \node[inter] (1) at (3, 0) {};
            \draw[->, very thick] (sigma) to[] node[above, pos=0.5]{$\sigma_2$} (1);
            \draw[->, very thick] (sigma) to[] node[left, pos=0.5]{$\sigma_0$} (sigma2);
            \draw[->, very thick] (1) to[] node[above left, pos=0.5]{$\sigma_0$} (sigma2);
            \draw[->, very thick] (1) edge[loop right, min distance=15mm, looseness=8] node[right]{$\sigma_2$} (1);
        \end{tikzpicture} \\ $ $ \\ $ $ \\ a}
    \end{minipage}
    \hfill
    \begin{minipage}[ht]{0.49\linewidth}
        \center{
        \begin{tikzpicture}[
        > = Stealth,
        node distance = 4cm,
        state/.style = {circle, draw, very thick, minimum size=1cm, font=\small},
        inter/.style = {circle, fill=black, inner sep=2pt}
        ]
            \node[state] (sigma) at (0,0) {$v_{\sigma}$};
            \node (sigma3) at (0,-4) {$v_{\sigma\cdot(0123)}$};
            \node (sigma2) at (6,0) {$v_{\sigma\cdot(012)}$};
            \node[inter] (1) at (3, 0) {};
            \node[inter] (2) at (3, -3) {};
            \draw[->, very thick] (sigma) to[] node[above, pos=0.5]{$\sigma_3$} (1);
            \draw[->, very thick] (sigma) to[] node[left, pos=0.5]{$\sigma_0$} (sigma3);
            \draw[->, very thick] (1) to[] node[above, pos=0.5]{$\sigma_0$} (sigma2);
            \draw[->, very thick] (1) to[] node[right, pos=0.5]{$\sigma_2$} (2);
            \draw[->, very thick] (2) edge[loop right, min distance=15mm, looseness=8] node[right]{$\sigma_3$} (2);
            \draw[->, very thick] (2) to[] node[above, pos=0.5]{$\sigma_0$} (sigma3);
        \end{tikzpicture}  \\ b}
    \end{minipage}
    \caption{Simplicial system for the Modified Triangle algorithm in dimension (a) $d=2$ and (b) $d=3$.}
    \label{fig:ss_mtr23}
\end{figure}
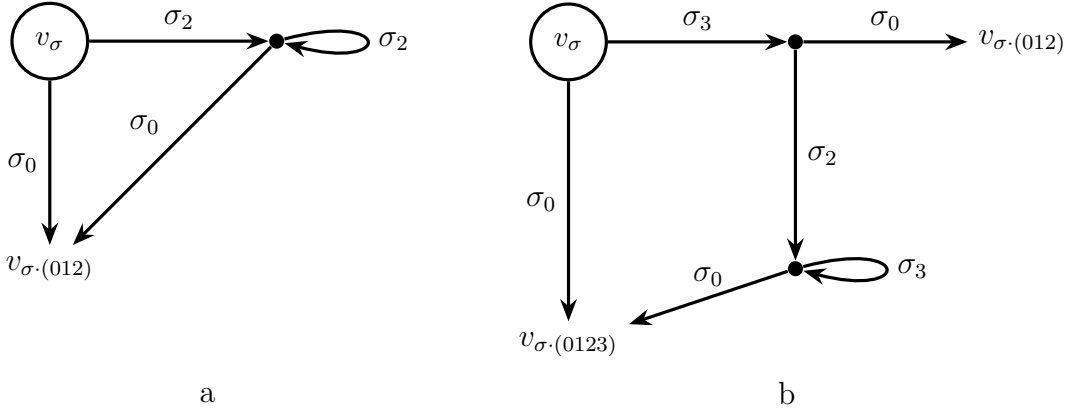

According to numerical experiments, it appears that in dimension $d=3$ 
there is no finite $n$-extension for the simplicial system $G$ for which 
a solution to the system~(\ref{eq:4}) exists. This 
suggests that the invariant domain for the natural extension of the Triangle algorithm in this 
dimension is essentially more complex than in the case $d=2$: namely, we 
conjecture that it is nonconvex and fractal. More precisely, numerical 
experiments indicate that the invariant domain for the associated win-lose 
induction is not a finite union of rational polyhedral cones, but rather 
decomposes into an infinite hierarchy of subcones. The following conjecture 
provides an explicit description of the first level of this decomposition.
\begin{conjecture}
    The invariant domain for the Triangle algorithm in dimension $d=3$ is nonconvex and fractal. For the associated win-lose induction, the invariant domain is contained in
    $$\bigsqcup_{\sigma\in S_4,i}\{v_{\sigma,\sigma_i}\}\times\Big(\Omega^1_{\sigma,\sigma_i}\mathbb{R}^4_+\sqcup \Omega^2_{\sigma,\sigma_i}\mathbb{R}^4_+\sqcup \Omega^3_{\sigma,\sigma_i}\mathbb{R}^4_+\Big),$$
    where 
    \begin{center}
        $\displaystyle \Omega^1_{\sigma,\sigma_i} = C_{\sigma'}\cdot\begin{bmatrix} 1 & 0 & 1 & 1 \\ 1 & 1 & 1 & 1  \\ 1 & 1 & 2 & 1 \\ 1 & 1 & 3 & 2 \end{bmatrix} \cdot C^{-1}_{\sigma'}, \quad\Omega^2_{\sigma,\sigma_i} = C_{\sigma'}\cdot\begin{bmatrix} 1 & 0 & 1 & 1 \\ 1 & 1 & 1 & 1  \\ 1 & 1 & 2 & 2 \\ 1 & 1 & 2 & 3 \end{bmatrix} \cdot C^{-1}_{\sigma'},$ \\
        $ $ \\
        $\displaystyle\Omega^3_{\sigma,\sigma_i} = C_{\sigma'}\cdot\begin{bmatrix} 1 & 0 & 0 & 1 \\ 1 & 1 & 0 & 2  \\ 2 & 1 & 1 & 3 \\ 2 & 1 & 1 & 4 \end{bmatrix} \cdot C^{-1}_{\sigma'},$ \\
        $ $ \\
    \end{center}
    $\sigma' = \sigma\cdot(\sigma_{i-d}\ldots\sigma_{i-1}\sigma_{i})$ and $C_{\sigma'}$ is the matrix for the permutation $\sigma'$.
\end{conjecture}
The search for the invariant domain and measure for $d \geq 3$ will be discussed in an upcoming paper.

\begin{figure}[ht]
\center{\includegraphics[width=0.8\linewidth]{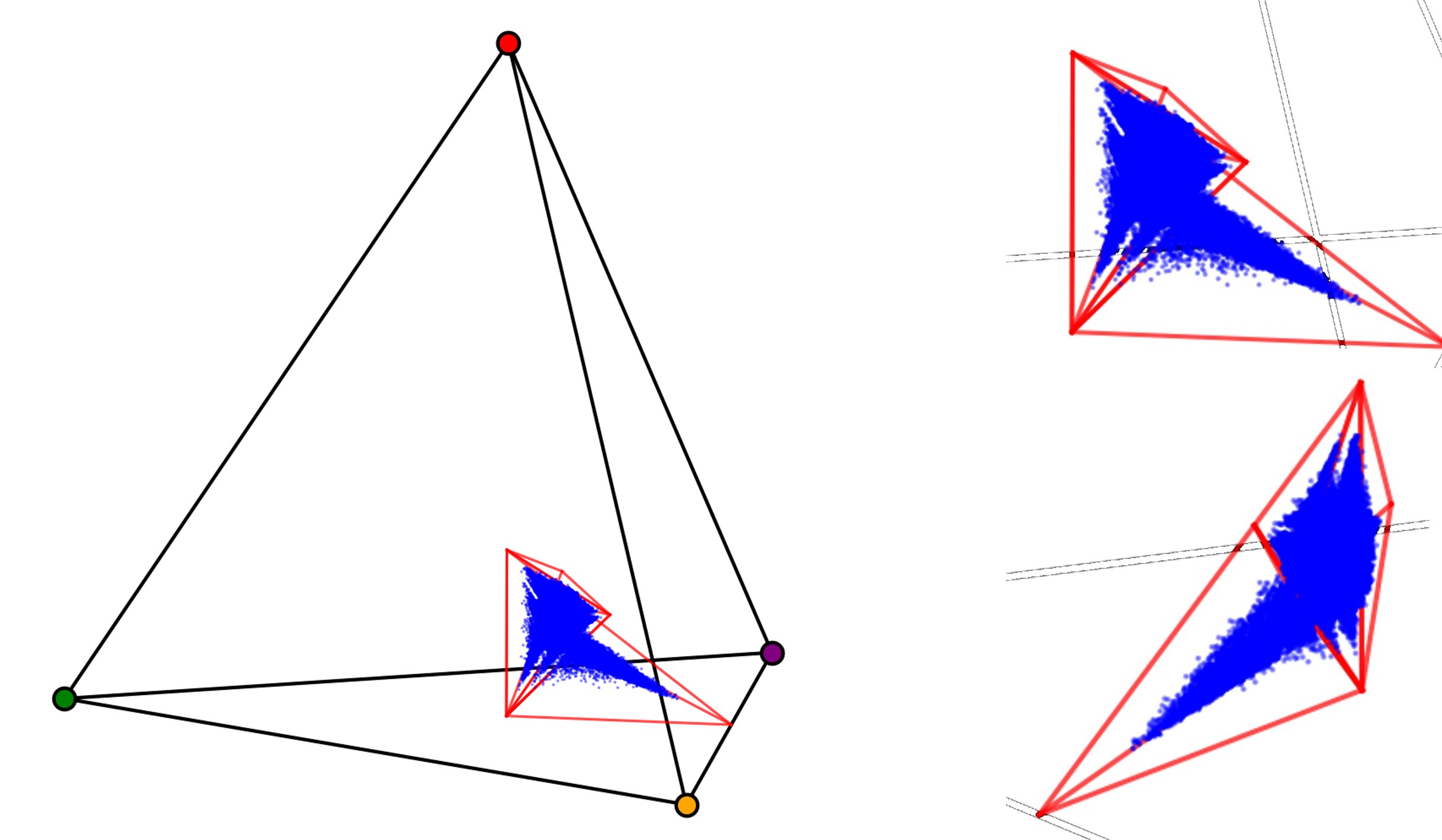}}
\caption{Sage approximations for the $v_{\text{id},0}$-part of the invariant domain of the win-lose induction conjugated to the Modified Triangle algorithm in dimension $d=3$.}
\label{fig:mtr4_approx}
\end{figure}

\section{Ergodicity}\label{sec:ergodicity}

The invariant measures for the algorithms under consideration have been 
constructed explicitly in the preceding sections. It remains to prove that 
these measures are the unique ergodic invariant measures equivalent to the 
Lebesgue measure. We do this by verifying a non-degeneracy condition for 
the win-lose induction, which, by a criterion of Fougeron, implies both 
uniqueness and ergodicity. We note that Fougeron's criterion 
\cite[Theorem~1.1]{fougeronpp} establishes the existence of such a measure 
without providing its explicit form; the present paper complements this 
result by constructing the measure explicitly via the graph method and then 
verifying the non-degeneracy condition to conclude ergodicity. We stress that the proofs in this section are combinatorial, relying only 
on the structure of the underlying simplicial system and requiring no 
analytic input beyond Fougeron's criterion.

For the Selmer algorithm, ergodicity was already established 
in~\cite{fougeronpp}, and we refer the reader to that work. For the 
Triangle algorithm, we prove ergodicity in arbitrary dimension, that is, the invariant measure constructed in the present paper is the unique 
ergodic invariant measure equivalent to the Lebesgue measure; this recovers 
and gives an independent combinatorial proof of a recent result of Garrity 
and Lehmann Duke~\cite[$\mathsection$~6]{garrityerg}, who obtained the same 
measure via an approach based on transfer operators. For the Modified 
Triangle algorithm, we establish ergodicity of the invariant measure in 
arbitrary dimension; we note that in this case an explicit formula for the 
invariant density is not at our disposal.

\begin{definition}\cite[$\mathsection$~3.2.2]{fougeronpp}
    Let $\mathcal{L}\subsetneq \{0,\ldots,d\}$. We  denote by $G_{\mathcal{L}}$ the subgraph of $G$ with the same set of vertices $V$ and a set of edges defined as follows: for any $v\in V$,\\
    \indent $\bullet\;$ if $|l(v_{\text{out}}\cap\mathcal{L})|\neq0$, then the set of outgoing edges for $v$ in $G_{\mathcal{L}}$ is
    $$\displaystyle v^{\mathcal{L}}_{\text{out}} = \{e\in v_{\text{out}}|\;l(e)\in\mathcal{L}\},$$
    \indent $\bullet\;$ otherwise, $\displaystyle v^{\mathcal{L}}_{\text{out}} = v_{\text{out}}.$
\end{definition}

\begin{definition}\label{def:nondeg}\cite[$\mathsection$~3.2.3]{fougeronpp}
    A simplicial system is \textit{non-degenerating} if: \\
    $1)$ from every vertex there exists a path along which each label in $\{0,\ldots,d\}$ appears, \\
    $2)$ for all $\emptyset\subsetneq\mathcal{L}\subsetneq \{0,\ldots,d\}$ and all $\mathcal{C}$ that are strongly connected components of $G_{\mathcal{L}}$,
    one of the following is true: \\
    \indent (a) for all vertices $v$ in $\mathcal{C}$, the cardinality $|l(v_{\text{out}}\cap\mathcal{L})|\leq1$, \\
    \indent (b) from every vertex in $\mathcal{C}$, there is a path in $G$ labeled in $\mathcal{L}$ leaving $\mathcal{C}$.
\end{definition}

\begin{remark}
    The first condition means that every symbol loses infinitely often; the second condition means that every symbol in a degenerated subgraph always loses eventually to a symbol in the complementary set.
\end{remark}

In terms of convergence, the non-degenerating property implies weak convergence of an algorithm. If a simplicial system is non-degenerating, then, during the application of an algorithm, any finite path appears infinitely many times (see [Fou$20$,\;Prop.$3.2$]). Then it is enough to consider a finite path which corresponds to a strictly positive matrix and an acceleration of the algorithm as a first return map on a corresponding subsimplex. In terms of ergodicity, non-degenerating property prove that the win-lose induction admits a unique ergodic measure equivalent to the Lebesgue measure \cite[Thm~1.1, 1.2.]{fougeronpp}.

\begin{proposition}\cite[Prop.~5.10]{fougeronppnew} The simplicial system for the Selmer algorithm in dimension $d$ is non-degenerating.
\end{proposition}

\begin{corollary}
    The Selmer algorithm in dimension $d$ admits a unique invariant ergodic measure equivalent to the Lebesgue measure with the density
$$\displaystyle h(\textbf{x}) = \frac{1}{x_0\ldots x_d} \quad\text{for }\textbf{x}\in \mathbb{R}_+\cdot\Delta_c^d.$$
\end{corollary}

\begin{proposition}
    The simplicial system for the Triangle algorithm in dimension $d$ is non-degenerating.
\end{proposition}

\begin{proof}
    The first condition holds by the construction of $G$: in any $C_{\sigma}$, in a cycle over all vertices, all labels appear. \\
    Consider the second condition. Any strongly connected component $\mathcal{C}$ of $G_{\mathcal{L}}$ is contained in $C_{\sigma}$ for some $\sigma\in S_{d+1}$. It contains either only this vertex or all vertices in $C_{\sigma}$. Suppose the latter is true; then all edges between different vertices in $C_{\sigma}$ are also present in $\mathcal{C}$. Consider such an edge with a label in $\{0\ldots d\} \diagdown \mathcal{L}$; then the loop from the same vertex also has a label in $\{0\ldots d\} \diagdown \mathcal{L}$. Following the incoming edges, we obtain that every label in $C_{\sigma}$ is in $\{0\ldots d\} \diagdown \mathcal{L}$, i.e., $\mathcal{L} = \emptyset$, a contradiction. \\
    Then any $\mathcal{C}$ contains just one vertex and the second condition holds: any vertex in $G$ has exactly $2$ outgoing edges; if $|l(v_{\text{out}}\cap\mathcal{L})| > 1$, then the path along any of these 2 edges is labeled in $\mathcal{L}$ and leaves $\mathcal{C}$.
\end{proof}

\begin{corollary}
    The Triangle algorithm in dimension $d$ admits a unique invariant ergodic measure equivalent to the Lebesgue measure with the density
    $$\displaystyle h(\textbf{x})=\frac{1}{x_0\ldots x_d}.$$
\end{corollary}

\begin{corollary}
    The accelerated version of the Triangle algorithm in dimension $d$ admits a unique invariant ergodic measure equivalent to the Lebesgue measure with the density
$$\displaystyle h(\textbf{x})=\frac{1}{(x_{\sigma_0}+x_{\sigma_d})x_{\sigma_0}\ldots x_{\sigma_{d-1}}}\quad\text{for }\textbf{x}\in\Lambda^{d+1}_{\sigma}.$$
\end{corollary}

\begin{remark}
    The ergodicity of the slow and accelerated Triangle algorithms was proved recently in \cite[$\mathsection$~6]{garrityerg}.
\end{remark}

\begin{proposition}
    The simplicial system for the Modified Triangle algorithm in dimension $d$ is non-degenerating.
\end{proposition}

\begin{proof}
    We verify the conditions of Definition~\ref{def:nondeg}.\\
    Condition~$1$ holds by construction of $G$: for any vertex $v_{\sigma}$, in a cycle over vertices
    $$\displaystyle v_{\sigma}, v_{\sigma\cdot(01\ldots d)}, v_{\sigma\cdot(01\ldots d)^2},\ldots,v_{\sigma\cdot(01\ldots d)^d},v_{\sigma},$$
    all labels $\sigma_0,\sigma_1,\ldots,\sigma_d$ appear. \\
    Consider Condition~$2$. Suppose that the vertex $v_{\sigma}$ is in the strongly connected component $\mathcal{C}$ of $G_{\mathcal{L}}$ and consider the subgraph $G_{v_{\sigma}}$ corresponding to $v_{\sigma}$; see Fig.~\ref{fig:ss_mtr}. If $\sigma_0 \in \{0\ldots d\} \diagdown \mathcal{L}$, Condition~$2$(a) holds for all vertices
    $$\displaystyle v_{\sigma},\text{int}_{\sigma}^{1},\text{int}_{\sigma}^{2},\ldots,\text{int}_{\sigma}^{d-1}$$
    in $G_{v_{\sigma}}$ (since each of them has $2$ outgoing edges and one is labeled by $\sigma_0$).\\
    Suppose $\sigma_0 \in \mathcal{L}$. Consider the set
    $$\displaystyle m(\sigma, \mathcal{L}) = \text{max}\big\{i|\;\sigma([d-i+1,\ldots,d]) \subset \mathcal{L}\big\},\quad\quad 0 \leq m(\sigma,\mathcal{L}) \leq d-1,$$
    where $m(\sigma,\mathcal{L}) = 0$ if $\sigma_d \notin \mathcal{L}$. \\
    If $m(\sigma,\mathcal{L}) = d-1$, consider the path
    $$\displaystyle v_{\sigma}, \text{int}^{d-1}_{\sigma}, \text{int}^{d-2}_{\sigma}, \ldots, \text{int}^2_{\sigma}, v_{\sigma\cdot(012)},$$
    whose edges have labels
    $$\displaystyle \sigma_d,\sigma_{d-1},\ldots,\sigma_3,\sigma_0,$$
    all in $\mathcal{L}$. This path exits $\mathcal{C}$ since in $G_{\sigma\cdot(012)}$ all the edges labeled by $\sigma_1$ are deleted, and then there is no path in $G_{\mathcal{L}}$ from the vertex $v_{\sigma\cdot(012)}$ to any other non-intermediate vertex. Hence Condition~$2$(b) holds. \\
    If $m = m(\sigma,\mathcal{L}) < d-1$, consider the path
    $$\displaystyle v_{\sigma}, \text{int}^{d-1}_{\sigma}, \text{int}^{d-2}_{\sigma}, \ldots, \text{int}^{d-m}_{\sigma}, v_{\sigma\cdot(01\ldots d-m)},$$
    whose edges have labels
    $$\displaystyle \sigma_d,\sigma_{d-1},\ldots,\sigma_{d-m+1},\sigma_0,$$
    all in $\mathcal{L}$. This path exits $\mathcal{C}$ since $\mathcal{L} \neq \{0,1,\ldots, d\}$, but any path from $v_{\sigma}$ in $G_{\mathcal{L}}$ could end only at $v_{\sigma'}$ such that if $\sigma'_i = \sigma_0$ for some $i$, then $\sigma'_j \in \mathcal{L}$ for all $j > i$. Hence Condition~$2$(b) holds again.
\end{proof}

\begin{corollary}
    The Modified Triangle algorithm in any dimension $d$ admits a unique invariant ergodic measure equivalent to the Lebesgue measure.
\end{corollary}
\medskip

\printbibliography
\end{document}